\documentclass[smallextended,referee,envcountsect]{svjour3}

\smartqed
\usepackage{graphicx}
\usepackage{amsmath,amssymb,amsfonts}
\usepackage{enumitem}
\usepackage{booktabs}
\usepackage{algorithmic}
\usepackage[ruled,vlined,linesnumbered]{algorithm2e}
\usepackage[caption=false]{subfig}
\usepackage{epstopdf}
\usepackage{xcolor}
\usepackage{hyperref}
\usepackage{newtxtext,newtxmath}
\usepackage[margin = 2.4cm]{geometry}

\definecolor{myciteblue}{RGB}{0,92,175}
\definecolor{myrefteal}{RGB}{0,110,120}
\definecolor{myurlred}{RGB}{150,50,50}

\hypersetup{
    colorlinks=true,
    citecolor=blue,
    linkcolor=myciteblue,
    urlcolor=myciteblue
}

\DeclareMathOperator*{\argmin}{arg\,min}
\newcommand{\dom}{\ensuremath{\operatorname{dom}}}
\newcommand{\ipH}[2]{\langle #1,#2\rangle_P}
\newcommand{\nH}[1]{\lVert #1\rVert_P}

\SetKwInput{KwIn}{Input}
\DontPrintSemicolon

\journalname{JOTA}

\makeatletter

\renewcommand\normalsize{%
   \@setfontsize\normalsize{12pt}{14pt}%
   \abovedisplayskip=3 mm plus6pt minus 4pt
   \belowdisplayskip=3 mm plus6pt minus 4pt
   \abovedisplayshortskip=0.0 mm plus6pt
   \belowdisplayshortskip=2 mm plus4pt minus 4pt
   \let\@listi\@listI
}

\renewcommand\small{%
   \@setfontsize\small{10pt}{12pt}%
   \abovedisplayskip 8.5\p@ \@plus3\p@ \@minus4\p@
   \abovedisplayshortskip \z@ \@plus2\p@
   \belowdisplayshortskip 4\p@ \@plus2\p@ \@minus2\p@
   \def\@listi{\leftmargin\leftmargini
               \parsep 0\p@ \@plus1\p@ \@minus\p@
               \topsep 4\p@ \@plus2\p@ \@minus4\p@
               \itemsep0\p@}%
   \belowdisplayskip \abovedisplayskip
}

\renewcommand\large{\@setfontsize\large{12pt}{14pt}}
\renewcommand\LARGE{\@setfontsize\LARGE{14pt}{16pt}}
\renewcommand\Large{\@setfontsize\Large{16pt}{18pt}}

\normalsize
\makeatother

\begin{document}

\title{The Golden Ratio Proximal ADMM with Norm Independent Step-Sizes for Separable Convex Optimization}

\titlerunning{Golden Ratio Proximal ADMM with Norm Independent Step-Sizes}

\author{Santanu Soe \and V. Vetrivel}

\authorrunning{S. Soe and V. Vetrivel}

\institute{%
Santanu Soe \at
Department of Mathematics, Indian Institute of Technology Madras, Chennai 600036, India; and
School of Mathematics and Statistics, The University of Melbourne, Parkville, VIC 3010, Australia; 
 \email{\href{ma22d002@smail.iitm.ac.in, santanu.soe@student.unimelb.edu.au}{ma22d002@smail.iitm.ac.in, santanu.soe@student.unimelb.edu.au}}
\and
Vellaichamy Vetrivel \at
Department of Mathematics, Indian Institute of Technology Madras, Chennai 600036, India; 
\email{\href{vetri@iitm.ac.in}{vetri@iitm.ac.in}}
}

\maketitle

\begin{abstract}
In this work, we propose two step-size strategies for the Golden ratio proximal ADMM (GrpADMM) to solve linearly constrained separable convex optimization problems. Both strategies eliminate explicit operator norm estimates by relying on inexpensive local information computed at the current iterate and requiring no backtracking. However, the key difference is that the second step-size strategy allows recovery from poor initial steps and can increase from iteration to iteration. Under standard assumptions, we establish global convergence of the generated iterates and derive sublinear convergence rates for both algorithms. We also obtain pointwise convergence rate results for the iterates of the algorithms. In addition, we show that the first proposed step-size rule for GrpADMM reduces to the fixed step-size counterpart when the initial step-size is chosen below a certain threshold. Preliminary numerical experiments demonstrate the practical adaptability and effectiveness of the proposed approaches.
\end{abstract}



\medskip


\section{Introduction}\label{sec:intro}
In this work, we consider the linearly constrained, separable convex optimization model
\begin{equation}\label{eq:model}
  \min_{x\in\mathbb{R}^{q},\, w\in\mathbb{R}^{p}} \; g(x) + f(w)
  \quad \text{subject to} \quad A x + B w = b,
\end{equation}
where $g:\mathbb{R}^q \to \mathbb{R}\cup\{+\infty\}$ and $f:\mathbb{R}^p \to \mathbb{R}\cup\{+\infty\}$ are proper, closed, convex functions (not necessarily differentiable), $A\in\mathbb{R}^{m\times q}$ and $B\in\mathbb{R}^{m\times p}$ are given linear operators, and $b\in\mathbb{R}^m$. 
Side constraints (e.g., bounds, sparsity, indicator restrictions) can be encoded via the \emph{effective domains} of $g$ and $f$.
The model~\eqref{eq:model} captures a broad spectrum of applications in signal and image processing, machine learning, statistical learning, and large-scale optimization problems; see, e.g., \cite{boyd2011distributed,chambolle2011first,yang2013linearized,tao2011recovering,yuan2012alternating,padcharoen2019augmented}.

\section{Preliminaries and assumptions}\label{subsec:notation}
Given vectors $u,v\in\mathbb{R}^n$, $\langle u,v\rangle$ denotes the standard inner product, and $\|u\|:=\sqrt{\langle u,u\rangle}$ denotes the associated norm. For a proper, closed, and convex function $h$, its \emph{effective domain} is
\[
    \operatorname{dom}(h) := \{\, x \;:\; h(x) < +\infty \,\},
\]
and its \emph{subdifferential} at $x\in \operatorname{dom}(h)$ is defined as
\[
    \partial h(x) := \{\, s \;:\; h(y) \ge h(x) + \langle s, y-x\rangle \;,~~ \forall y \,\}.
\]
We write $\operatorname{ri}(C)$ for the \emph{relative interior} of a convex set $C$. Given a matrix $M\in\mathbb{R}^{m\times m}$, $M^\top$ denotes its transpose and $I$ denotes the identity matrix. The set of all $m\times m$ real \emph{symmetric positive semidefinite} (resp.\ \emph{positive definite}) matrices is denoted by $\mathbb{S}_+^m$ (resp.\ $\mathbb{S}_{++}^m$). Alternatively, given $M\in\mathbb{S}_+^m$ (resp.\ $\mathbb{S}_{++}^m$), we write $M\succeq0$ (resp. \ $M\succ 0$). For $M\in \mathbb{S}_+^m$ and $y,z\in\mathbb{R}^m$, we let
\[
    \langle y,z\rangle_M := y^\top M z,
    \qquad
    \|y\|_M := \sqrt{\langle y,y\rangle_M}.
\]
We denote by $\lambda_{\min}(M)$ the smallest eigenvalue of $M$, and $\|A\| := \sup\{\|Ax\| : \|x\|=1\}$ for the operator norm of a linear map $A$. We write $\operatorname{blkdiag}(M_1,\dots,M_r)$ for the \emph{block-diagonal matrix} with diagonal blocks $M_1,\dots,M_r$. In particular,  for scalar $a_1,a_2,\ldots,a_r$ and identity matrix $I\in\mathbb{R}^{d\times d}$, $\operatorname{blkdiag}(a_1I,\dots,a_rI)$ denotes the block-diagonal matrix whose $i$-th block is $a_iI$. Given a nonempty set $C \subseteq \mathbb{R}^d$ and a point $z \in \mathbb{R}^d$, we define the distance from $z$ to $C$ by
\[
\operatorname{dist}(z,C):=\inf_{u\in C}\|z-u\|.
\]
If $C$ is closed and convex, then the \emph{Euclidean projection} of $z$ onto $C$ is denoted by $\Pi_C(z)$, and is defined as
\[
\Pi_C(z):=\arg\min_{u\in C}\|z-u\|.
\]
In this case, the distance and the projection are related through
\[
\operatorname{dist}(z,C)=\|z-\Pi_C(z)\|.
\]
We denote by $\mathbb{N}:=\{1,2,\ldots\}$ the set of all positive integers, and by $\mathbb{R}$ the set of all real numbers. Moreover, we write $\mathbb{R}_+^n:=\{u\in\mathbb{R}^n:\ u_i\geq 0,\ i=1,\ldots,n\}$ for the nonnegative orthant in $\mathbb{R}^n$, and $\Delta^n:=\{u\in\mathbb{R}_+^n:\ \mathbf{1}^\top u=1\}$ for the \emph{probability simplex} in $\mathbb{R}^n$, where $\mathbf{1}$ denotes the all ones vector of appropriate dimension. The Euclidean projection onto $\mathbb{R}_+^n$ is denoted by $\Pi_{\mathbb{R}_+^n}$ and is given componentwise by
\[
\Pi_{\mathbb{R}_+^n}(z)
=
\bigl(\max\{z_1,0\},\dots,\max\{z_n,0\}\bigr)
\quad \text{for all } z=(z_1,\dots,z_n)\in\mathbb{R}^n.
\]

Throughout, we denote the Golden ratio by $\varphi := \frac{1+\sqrt{5}}{2}.$ Depending on the algorithm under consideration, the parameter $\psi$ will be chosen in $(1,\varphi]$ or $(1,\varphi)$, or a larger admissible interval specified later. For a sequence $(x_k)\in\mathbb{R}^n$ such that $\lim_{k\to\infty}x_k = \alpha$, we alternatively write either $x_k\to\alpha$ or $\|x_k-\alpha\|\to0$.

\medskip
Given a \emph{Lagrange multiplier} $y\in\mathbb{R}^m$ for the linear equality constraint $Ax+Bw=b$, and a penalty parameter $\sigma>0$, the objective, Lagrangian function, and the augmented Lagrangian function associated with~\eqref{eq:model} are
\begin{align}
    \Phi(x,w) &:= g(x)+f(w)\nonumber,\\
    \mathbb{L}(x,w,y) &:= \Phi(x,w) + \langle y,\, Ax + Bw - b\rangle,\label{eq:lagrangian}\\
    \mathbb{L}_\sigma(x,w,y) &:= \mathbb{L}(x,w,y) + \frac{\sigma}{2}\,\|Ax + Bw - b\|^2\label{eq:aug_lagrangian} .
\end{align}

We now state the following blanket assumption, which will be used throughout the paper.
\begin{assumption}\label{assump:basic}
\leavevmode
\begin{enumerate}
\item The solution set of~\eqref{eq:model} is nonempty.
\item There exist $\tilde x\in \operatorname{ri}(\dom g)$ and $\tilde w\in \operatorname{ri}(\operatorname{dom} f)$ such that
$A\tilde x + B\tilde w = b$.
\end{enumerate}
\end{assumption}
Under Assumption~\ref{assump:basic}, it follows from \cite[Corollaries 28.2.2 and 28.3.1]{rockafellar1970convex} that a pair $(x^\star,w^\star)$ solves~\eqref{eq:model} if and only if there exists $y^\star\in\mathbb{R}^m$ such that $(x^\star,w^\star,y^\star)$ is a saddle point of $\mathbb{L}$, which is equivalently characterized by the following inequality
\begin{equation}\label{eq:saddle}
    \mathbb{L}(x^\star,w^\star, y) \;\le\; \mathbb{L}(x^\star,w^\star, y^\star) \;\le\; \mathbb{L}(x,w,y^\star)
    \quad \text{for all } (x,w,y)\in\mathbb{R}^q\times\mathbb{R}^p\times\mathbb{R}^m .
\end{equation}
We denote the optimal objective value by
\[
    \Phi^\star := \min\{\Phi(x,w): Ax + Bw = b\}.
\]

We end this section by stating the following useful lemmas.

\begin{lemma} \cite{Chen2023GRPADMM}\label{lemm:cosrule}
Let $P\in\mathbb{S}_m^{+}$. Then, for any $a,b,c,d\in\mathbb{R}^m$ and  $\theta\in\mathbb{R}$, we have
\begin{subequations}\label{eq:cosrule}
\begin{align}
\label{eq:cosrule:fourpoint}
2\,\ipH{a-b}{c-d}
  &= \nH{a-d}^{2}+\nH{b-c}^{2}-\nH{a-c}^{2}-\nH{b-d}^{2}, \\[1mm]
\label{eq:cosrule:convexcomb}
\nH{(1-\theta)a+\theta b}^{2}
  &= (1-\theta)\nH{a}^{2}+\theta\nH{b}^{2}-\theta(1-\theta)\nH{a-b}^{2}.
\end{align}
\end{subequations}
\end{lemma}

\begin{lemma}\cite{Chen2023GRPADMM}\label{lem:fejer-basic}
Let $(\alpha_k)_{k\ge0}$ and $(\delta_k)_{k\ge0}$ be two nonnegative real sequences.
Assume that there exists a natural number $\tilde k\in\mathbb{N}$ such that
\[
\alpha_{k+1}\le \alpha_k-\delta_k\quad\text{for all } k\geq \tilde k.
\]
Then $(\alpha_k)$ has a finite limit and $\sum_{k=\tilde k}^{\infty}\delta_k<\infty$.
\end{lemma}
\begin{lemma}\label{usefullemma_3}
Given $a,b,p,q\in\mathbb{R}$ with $p+q>0$, we have
\[
\frac{pq}{p+q}(a+b)^2 \le p\,a^2 + q\,b^2.
\]
\end{lemma}
\begin{proof}
For the sake of completeness, we provide a proof. Observe that
\[
(p+q)(pa^2+qb^2)-pq(a+b)^2
= p^2a^2-2pqab+q^2b^2
= (pa-qb)^2\ge 0.
\]
Since $p+q>0$, the conclusion follows immediately after dividing by $p+q$.
\end{proof}
\begin{lemma}\label{lemm: sum un and vn}
Let $0<q_1<1$, and let $(a_n)$ and $(b_n)$ be nonnegative sequences such that $a_n\le q_1\,a_{n-1}+b_n$ for all $n\ge 1$.
If $\sum_{n=1}^\infty b_n<\infty$, then $\sum_{n=1}^\infty a_n<\infty$.
\end{lemma}
\section{Literature review}
In the literature, several methods have been proposed for solving \eqref{eq:model}, including the method of multipliers \cite{Hestenes1969,powell1969method}, also known as the augmented Lagrangian method (ALM), the alternating direction method of multipliers (\ref{eq:ADMM}) \cite{GabayMercier1976,GlowinskiMarroco1975}, proximal ADMM (\ref{eq:PADMM}) \cite{Eckstein1994,shefi2014rate}, and their different variants. Although ALM jointly solves \eqref{eq:model} using the augmented Lagrangian function, it does not take advantage of the separable structure of the objective function, whereas \ref{eq:ADMM} exploits the separable structure of the objective and decomposes problem \eqref{eq:model} into simpler subproblems that can be solved relatively easily. Given $(x_0,w_0,y_0)$ and a penalty parameter $\sigma>0$, the iteration scheme of ADMM is
\begin{equation}
\label{eq:ADMM}
\left\{
  \begin{aligned}
    x_{k+1}&:= \arg\min_x\ \mathbb{L}_\sigma(x,w_k,y_k),\\
    w_{k+1}&:= \arg\min_w\ \mathbb{L}_\sigma(x_{k+1},w,y_k),\\
    y_{k+1}&:= y_k + \sigma\,(A x_{k+1} + B w_{k+1} - b).
  \end{aligned}
\right.
\tag{ADMM}
\end{equation}
Under mild assumptions, the objective values and feasibility residuals converge, and the method enjoys an ergodic sublinear rate \cite{boyd2011distributed,EcksteinBertsekas1992}. However, in general, the \emph{primal iterates} may fail to converge. To address this, Eckstein \cite{Eckstein1994} and subsequent works \cite{ParikhBoyd2014Prox,he2002new} introduce quadratic proximal terms in the $x$ and $w$-updates with weighted matrices. Given $(x_0,w_0,y_0)$, $\sigma>0$ and weights $S\in S^{q}_{+}$, $T\in S^{p}_{+}$, the PADMM iterates take the form
\begin{equation}\label{eq:PADMM}
\left\{
\begin{aligned}
x_{k+1}
&:= \arg\min_x\ \mathbb{L}_\sigma(x,w_k,y_k) + \tfrac12\|x-x_k\|_S^2,\\[2pt]
w_{k+1}
&:= \arg\min_w\ \mathbb{L}_\sigma(x_{k+1},w,y_k) + \tfrac12\|w-w_k\|_T^2,\\[2pt]
y_{k+1}
&:= y_k + \sigma\,(A x_{k+1} + B w_{k+1} - b).
\end{aligned}
\right.
\tag{PADMM}
\end{equation}

Appropriate choices of $(S, T)$ can simplify subproblems and stabilise the iterates of PADMM. For example, when $S=T=0$, \ref{eq:PADMM} reduces to \ref{eq:ADMM}. Furthermore, when $S=\tfrac{1}{\tau}I-\sigma A^\top A$ and $T=\tfrac{1}{\eta}I-\sigma B^\top B$, where $\tau, \eta>0$, the subproblems of the \ref{eq:PADMM} can be solved using the proximal operators of $f$ and $g$, in which case, the resulting algorithm is known as the linearized ADMM; see \cite{chen2015inertial,he2002new,wang2012linearized}. In particular, when $B=-I$ and $b=0$, the authors \cite{shefi2014rate} proved the sublinear rate results for the linearized ADMM, measured by function value residual and constraint violation. Furthermore, it was shown in \cite{shefi2014rate} that whenever $S\succ 0$, or $S=0$ and $A$ has full column rank, the sequence generated by \eqref{eq:PADMM} converges to a saddle point of $\mathbb{L}$. We also refer the reader to \cite{ouyang2015accelerated,nesterov1983method} for inertial and accelerated variants of related methods. Various symmetric and generalised ADMM variants have been proposed for separable convex optimization. For example, Bai et al.~\cite{bai2018generalized} introduced GS-ADMM for multi-block problems, and its sublinear nonergodic and linear convergence properties were further studied in \cite{bai2021convergence}. Recent developments also include convex-combination, stochastic, inexact, and accelerated ADMM schemes; see, e.g., \cite{xu2017accelerated,li2019accelerated,chen2015inertial,chen2018stochastic}. In particular, Wang et al.~\cite{wang2025convex} proposed a convex combined symmetric ADMM for separable convex optimisation, while Bai et al.~\cite{bai2022inexact,bai2022symmetric} developed stochastic accelerated variants. Moreover, Han et al.~\cite{han2018linear} established linear convergence of ADMM for convex composite programming. From a broader operator-splitting perspective, Bo\c{t} and Csetnek~\cite{bot2019admm,bot2015forwardbackward} studied ADMM and related primal--dual schemes in monotone-operator frameworks. Beyond the convex setting, proximal and Bregman-style ADMM variants have also been investigated for nonconvex and nonsmooth problems; see, for example \cite{bot2020proximal,li2026proximal,liu2023bregman,liu2024bregman,bai2026proximal,liu2025half}. In this paper, we only focus on the case where the component functions of \eqref{eq:model} are convex. We also note that the literature on ADMM and its variants is extremely vast, and a comprehensive review is beyond the scope of this paper.

An interesting variant of \eqref{eq:PADMM} was recently studied by Chen et al. \cite[Algorithm 1]{Chen2023GRPADMM}. They proposed a proximal ADMM based on a Golden ratio extrapolation, namely GrpADMM, which converges to a solution under more relaxed parameter choices. Given $(x_0,w_0,y_0)$ with $u_0=x_0$, $S\in S^{q}_{+}$, $T\in S^{p}_{+}$ and $\psi\in(1,\varphi]$, the iteration scheme of GrpADMM is 
\begin{equation}\label{eq:GrpADMM}
\left\{
\begin{aligned}
u_k &:=\tfrac{\psi-1}{\psi}\,x_{k-1}+\tfrac{1}{\psi}\,u_{k-1},\\[2pt]
x_k &:=\arg\min_x\ \mathbb{L}(x,w_{k-1},y_{k-1})+\tfrac{1}{2\tau}\|x-u_k\|_{S}^{2},\\[2pt]
w_k&:=\arg\min_w\ \mathbb{L}_{\sigma}(x_k,w,y_{k-1})+\tfrac12\|w-w_{k-1}\|_{T}^{2},\\[2pt]
y_k&:=y_{k-1}+\sigma\,(A x_k+B w_k-b),
\end{aligned}
\right.
\tag{GrpADMM}
\end{equation}
where $\tau, \sigma>0$ are positive step-sizes. Note that although $w$ and $y$-updates of \ref{eq:GrpADMM} and \ref{eq:PADMM} are the same, in the $x$-update, the convex combination of the golden ratio $u_k$ is used, which is the combination of all the previous iterates $x_0, x_1$, all the way up to $x_{k-1}$. In the special case when $S=I, T=0$ and $B=-I$ with $b=0$, \ref{eq:GrpADMM} reduces to the GRPDA algorithm proposed by Chang and Yang \cite{chang2021golden}. Under the requirements $S,T\succ0$, the sequence generated by \ref{eq:GrpADMM} converges to a saddle point of $\mathbb{L}$, provided $\tau\sigma\|A\|^2<\psi\lambda_{\min}(S)$ is satisfied, where $\psi\in(1,\varphi]$, see \cite[Theorem 2.1]{Chen2023GRPADMM}. An advantage of \ref{eq:GrpADMM} is that, when $S=I$ and $T=\tfrac{1}{\eta}I-\sigma B^\top B$, the step-sizes need to satisfy $\tau\sigma\|A\|^2<\psi$ and $\eta\sigma\|B\|^2<1,$
with $\psi\in(1,\varphi]$. In contrast, for \ref{eq:PADMM}, the conditions are
$\tau\sigma\|A\|^2<1$ and $\eta\sigma\|B\|^2<1.$ Since $\psi>1$, the constraint on $\tau$ and $\sigma$ in \ref{eq:GrpADMM} is strictly less restrictive than in \ref{eq:PADMM}, permitting a broader range of admissible parameter choices. Nonetheless, both \eqref{eq:PADMM} and \eqref{eq:GrpADMM} require prior knowledge of $\|A\|$ in order to choose suitable step-sizes $\tau$ and $\sigma$. For large-scale convex optimisation problems, however, computing or accurately estimating $\|A\|$ can be expensive and, in some cases, infeasible. Thus, a natural question, also suggested in the conclusion of \cite{Chen2023GRPADMM}, is the following:
\begin{quote}
    Can one develop variants of \eqref{eq:GrpADMM} that do not require the computation of $\|A\|$ for choosing admissible values of $\tau$ and $\sigma$?
\end{quote}
\noindent
In this paper, we answer this question affirmatively. In particular, we design iterative step-size rules that avoid explicit dependence on such parameters. To this end, we propose two step-size strategies for solving \eqref{eq:model} based on \eqref{eq:GrpADMM}. In the first strategy, the primal step-size sequence $(\tau_k)$ is decreasing, and it converges to a positive constant. This crucial fact enables us to prove the global convergence of Algorithm \ref{alg:1}. A similar type of strategy has been proposed in \cite[Algorithm 2]{soe2026golden} to solve three-operator splitting problems in which one of the operators is globally smooth. However, our framework is more general than the one considered in \cite{soe2026golden}, and in this case, the smooth part is zero; see Remark \ref{connec_with_soe et al._paper}. In the second step-size strategy, $(\tau_k)$ is allowed to be non-decreasing, at the cost of modifying the proximal term in the $w$-subproblem compared with \eqref{eq:GrpADMM}, and restricting the choice of the parameter $\psi$ so that the golden ratio $\varphi$ is excluded. The latter strategy can be advantageous when the initial step-size is chosen too conservatively, since the step-size may increase along the iterations. By contrast, in the former strategy, if $\tau_0$ is less than or equal to a certain threshold, the step-size may become fixed, which can lead to very slow convergence; see Remark~\ref{Rmk:tau fixed case} for a detailed discussion.

The main contributions of this paper can be summarised as follows.

\begin{itemize}
    \item In Section~\ref{sec_4}, we propose a decreasing step-size rule for \eqref{eq:GrpADMM} that provides a local estimate of $\|A\|$ ( see \eqref{eq:steps tau_k} in Algorithm~\ref{alg:1}), without using any backtracking procedure. For this algorithm, we establish global convergence of the generated iterates, together with ergodic sublinear convergence rates in terms of the objective residual and the feasibility violation. In the worst-case scenario, when the initial step-size is poorly chosen, we show that our scheme reduces to the existing fixed step-size \ref{eq:GrpADMM} algorithm.

\item  We further prove pointwise convergence of the iterates generated by Algorithm 1. In addition,
we enlarge the admissible range of the parameter $\psi$ from the Golden ratio $\phi$ to
$1+\sqrt{3}$. This wider range may lead to faster convergence in practice, as observed in
\cite{chang2022grpda,soe2026golden} for related special cases of \eqref{eq:model}.

    \item In Section~\ref{sec_5}, we develop a non-decreasing step-size strategy for a modified variant of \eqref{eq:GrpADMM}. Here, the modification refers to the scaled proximal term in the $w$-subproblem of Algorithm~\ref{alg:2} by $\sigma_k^{-1}$, where $\sigma_k$ is determined through \eqref{eq:steps_alg_2_tau_k}. This modification yields a Fej\'er monotonicity property (see Lemma~\ref{lemma_2}). In this setting, the primal step-sizes converge to a positive constant, while still being allowed to increase along the iterations. We prove global convergence of Algorithm~\ref{alg:2}, and also establish its ergodic and pointwise convergence properties.

    \item Finally, in Section~\ref{numerical_section}, we present numerical experiments on several benchmark problems to demonstrate the practical performance of the proposed methods and their advantages over existing approaches.
\end{itemize}

\section{Main results}\label{sec_4}
In this section, we propose a decreasing step-size strategy for \ref{eq:GrpADMM}, where one does not require the explicit computation of $\|A\|$; this is achieved by evaluating the primal steps ($\tau_k$) as the minimum of the previous step-size and an appropriately scaled inverse of a local estimate of $\|A\|$, as presented in Algorithm~\ref{alg:1}. 
For notational convenience, we define the following local approximation of the norm of the operator~$A$:
\[
L_k :=
\dfrac{\|Ax_k-Ax_{k-1}\|}{\|x_k-x_{k-1}\|}, ~~\text{if } x_k \neq x_{k-1}
\]

\begin{algorithm}[H]
\caption{The \ref{eq:GrpADMM} with decreasing step-size to solve \eqref{eq:model}}\label{alg:1}
\KwIn{Let $S\in\mathbb{S}^{q}_{++}$, $T\in\mathbb{S}^{p}_{++}$. Choose $x_{0}\in\mathbb{R}^{q}$, $w_{0}\in\mathbb{R}^{p}$ and $y_{0}\in\mathbb{R}^{m}$ with $u_{0}=x_{0}$. Let $\tau_0>0$, $\beta>0$, $\psi\in(1,\varphi]$ and $0<\mu<\frac{\psi}{2}$.}

\For{$k=1,2,\ldots$}{
  \textbf{Step 1} (Compute)
  \begin{equation}\label{eq:grp_u}
    u_{k}=\frac{\psi-1}{\psi}\,x_{k-1}+\frac{1}{\psi}\,u_{k-1},
  \end{equation}
    \begin{equation}\label{eq:grp_x}
    x_{k}=\argmin_{x}\!\left\{\mathbb{L}(x,w_{k-1},y_{k-1})+\frac{1}{2\tau_{k-1}}\|x-u_{k}\|_{S}^{2}\right\}.
  \end{equation}

  \textbf{Step 2} (Update)
\begin{equation}\label{eq:steps tau_k}
    \tau_k = \min\left\{\tau_{k-1}, \frac{\mu\sqrt{\lambda_{\min}(S)}}{\sqrt{\beta}L_k}\right\}.
\end{equation}

  \textbf{Step 3} (Compute)
 \begin{equation}\label{eq:grp_w}
    w_{k}=\argmin_{w}\!\left\{\mathbb{L}_{\beta\tau
    _k}(x_{k},w,y_{k-1})+\frac{1}{2}\|w-w_{k-1}\|_{T}^{2}\right\},
  \end{equation}
  \begin{equation}\label{eq:grp_y}
    y_{k}=y_{k-1}+\beta\tau
    _k\big(Ax_{k}+Bw_{k}-b\big).
  \end{equation}
}
\end{algorithm}
\begin{remark}\label{Rmk_tau_k_bdd_below}
When $x_k=x_{k-1}$, under the convention $\frac{0}{0}=+\infty$, $\tau_k$ is set to $\tau_{k-1}$. From the update rule \eqref{eq:steps tau_k} and the boundedness of the linear operator $A$, observe that
\begin{equation}\label{tau_k_bdd_below}
    \tau_k \;\geq\min\!\left\{\tau_{k-1},\;
    \frac{\mu\sqrt{\lambda_{\min}(S)}}{\sqrt{\beta}\,\|A\|}\right\}.
\end{equation}
By using induction, we obtain
\[
\tau_k\;\ge\;\underline{\tau}:=\min\!\left\{\tau_0,\;
    \frac{\mu\sqrt{\lambda_{\min}(S)}}{\sqrt{\beta}\,\|A\|}\right\}~~\text{for all}~k.
\]
Moreover, \eqref{eq:steps tau_k} implies that the sequence $(\tau_k)$ is monotonically decreasing. Hence, $(\tau_k)$ is convergent and, by \eqref{tau_k_bdd_below}, its limit satisfies $\lim_{k\to\infty}\tau_k\ge\underline{\tau}>0$.
\end{remark}

\begin{remark}\label{Rmk:tau fixed case}
From \eqref{eq:steps tau_k}, see that if $\tau_0\le \frac{\mu\sqrt{\lambda_{\min}(S)}}{\sqrt{\beta}\|A\|}$, then $\tau_k=\tau_0$ for all $k\ge 1$. Indeed, since $L_k\le \|A\|$ $\forall k$, we have
\[
\frac{\mu\sqrt{\lambda_{\min}(S)}}{\sqrt{\beta}L_k}\ge \frac{\mu\sqrt{\lambda_{\min}(S)}}{\sqrt{\beta}\|A\|}\ge \tau_0,
\]
and hence $\tau_1=\tau_0$. Repeating the same argument inductively, we obtain $\tau_k=\tau_0$  $\forall k\ge 1$. Consequently, $\sigma_k=\beta\tau_k=\beta\tau_0$ $\forall k\ge 1$, and therefore noting $0<\mu<1<\psi$, we have
\[
\tau_k\sigma_k\|A\|^2=\beta\tau_0^2\|A\|^2\le \mu^2\lambda_{\min}(S)<\psi\lambda_{\min}(S),
\]
which is precisely the step-size condition required in \eqref{eq:GrpADMM}.
\end{remark}

\begin{remark}\label{connec_with_soe et al._paper}
In \cite[Algorithm~2]{soe2026golden}, the authors introduced a decreasing step-size algorithm for a special case of \eqref{eq:model}, namely when $B=-I$ and $b=0$. In addition, when $S=I$ and $T=0$, Algorithm~\ref{alg:1} simplifies to the method studied in \cite{soe2026golden}. Thus, Algorithm~\ref{alg:1} can be regarded as a natural extension of the algorithm proposed in \cite{soe2026golden}.
\end{remark}

\begin{lemma}\label{lem:energy}
Under Assumption \ref{assump:basic}, let $\{(u_k,x_k,w_k,y_k,\tau_k)\}_{k\ge1}$ be generated by Algorithm~\ref{alg:1}. Let $(\bar x,\bar w)$ be any solution of \eqref{eq:model}. Then, there exists a natural number $k_1$ such that, for all $k\geq k_1$ and for any $y\in\mathbb{R}^m$, the following holds.
\begin{align}\label{eq:step-ineq-k}
2\tau_k\!\left(\mathbb{L}(x_k,w_k,y)-\Phi^\star\right)
&\le \frac{\psi}{\psi-1}\big(\|u_{k+1}-\bar x\|_{S}^{2}-\|u_{k+2}-\bar x\|_{S}^{2}\big) + \big(\|\bar w-w_{k-1}\|_{T_{k-1}}^{2}\nonumber\\&
\quad -\|\bar w-w_k\|_{T_k}^{2}\big)
       + \frac{1}{\beta}\big(\|y-y_{k-1}\|^{2}
      -\|y-y_k\|^{2}\big) - \|w_k-w_{k-1}\|_{T_k}^{2}\nonumber\\&
     \quad -\frac{1}{\beta}(1-\mu)\|y_k-y_{k-1}\|^{2}
      -\frac{\psi\tau_k}{\tau_{k-1}}\|x_k-u_{k+1}\|_{S}^{2}.
\end{align}
\end{lemma}
\begin{proof}
From the $x$– and $w$–subproblems \eqref{eq:grp_x} and \eqref{eq:grp_w}, we have 
\begin{align}
0 &\in \partial g(x_k) + A^{\top}y_{k-1} + \frac{1}{\tau_{k-1}}S(x_k-u_k),                                    \label{eq:opt-x}\\
0 &\in \partial f(w_k) + B^{\top}y_k + T(w_k-w_{k-1}).\label{eq:opt-w}
\end{align}
By the convexity of $g$ and $f$, \eqref{eq:opt-x}–\eqref{eq:opt-w} imply that, for all $x\in\mathbb{R}^q$ and $w\in\mathbb{R}^p$
\begin{align}
g(x_k)-g(x)
&\le \langle A^{\top}y_{k-1},\,x-x_k\rangle
   +\frac{1}{\tau_{k-1}}\langle S(x_k-u_k),\,x-x_k\rangle,                                                       \label{eq:sg-1}\\
f(w_k)-f(w)
&\le \langle B^{\top}y_k,\,w-w_k\rangle
   +\langle T(w_k-w_{k-1}),\,w-w_k\rangle.                                                                       \label{eq:sg-2}
\end{align}
Again, similar to \eqref{eq:sg-1}, we have
\begin{align}\label{eq:sg-1_k+1}
g(x_{k+1})-g(x)
&\le \langle A^{\top}y_k,\,x-x_{k+1}\rangle
   +\frac{1}{\tau_{k}}\langle S(x_{k+1}-u_{k+1}),\,x-x_{k+1}\rangle.                                       
\end{align}
Plugging $x=x_{k+1}$ in \eqref{eq:sg-1}, $w=\bar w$ in \eqref{eq:sg-2}, $x=\bar x$ in \eqref{eq:sg-1_k+1}, and then adding them together by using the fact that $\Phi^\star = \Phi(\bar x, \bar w)$, and an elementary calculation yields 
\begin{align}\label{eq:ineq_1}
\Phi(x_k,w_k)-\Phi^\star
&\le\langle y_{k-1},\,A(x_{k+1}-x_k)\rangle
   + \langle y_k,\,A(\bar x-x_{k+1})+B(\bar w-w_k)\rangle\nonumber\\
&\quad + \frac{1}{\tau_{k-1}}\langle S(x_k-u_k),\,x_{k+1}-x_k\rangle
      + \frac{1}{\tau_k}\langle S(x_{k+1}-u_{k+1}),\,\bar x-x_{k+1}\rangle \nonumber\\
&\quad + \langle T(w_k-w_{k-1}),\,\bar w-w_k\rangle. 
\end{align}
Note that using $A\bar x+ B \bar w= b$, we have
\begin{align}\label{eq:derive_1}
    &\langle y_{k-1},\,A(x_{k+1}-x_k)\rangle
   + \langle y_k,\,A(\bar x-x_{k+1})+B(\bar w-w_k)\rangle\nonumber\\
   & = \langle y_{k-1},\,A(x_{k+1}-x_k)\rangle
   -\langle y_k,\,Ax_{k+1}+Bw_k-b\rangle\nonumber\\
   &=\langle y_{k-1},\,A(x_{k+1}-x_k)\rangle- \langle y_k,\, A(x_{k+1}-x_k)\rangle - \langle y_k,\, Ax_k+Bw_k-b\rangle\nonumber\\
   &=\langle y_{k-1}-y_k, \, A(x_{k+1}-x_k)\rangle- \langle y_k,\, Ax_k+Bw_k-b\rangle.
\end{align}
By substituting \eqref{eq:derive_1} into \eqref{eq:ineq_1} and then adding $\langle y,\,  A x_k + B w_k - b\rangle$ to both sides, by the facts that 
\[
\mathbb{L}(x_k,w_k,y)-\Phi^\star
=\big(\Phi(x_k,w_k)-\Phi^\star\big)+\langle y,\, A x_k + B w_k - b\rangle~~\text{and}~~x_k-u_k=\psi(x_k-u_{k+1}),
\]
we obtain
\begin{align*}
\mathbb{L}(x_k,w_k,y)-\Phi^\star
&\le\langle y_{k-1}-y_k, \, A(x_{k+1}-x_k)\rangle+\langle y- y_k,\, Ax_k+Bw_k-b\rangle\nonumber\\
&\quad + \frac{\psi}{\tau_{k-1}}\langle S(x_k-u_{k+1}),\,x_{k+1}-x_k\rangle
      + \frac{1}{\tau_k}\langle S(x_{k+1}-u_{k+1}),\,\bar x-x_{k+1}\rangle \nonumber\\
&\quad + \langle T(w_k-w_{k-1}),\,\bar w-w_k\rangle. 
\end{align*}
Moreover, from \eqref{eq:grp_y}, noting $Ax_k+Bw_k-b = \dfrac{1}{\beta\tau_k}(y_k-y_{k-1})$, we have
\begin{align}\label{eq: derive_final_ineq_1}
\mathbb{L}(x_k,w_k,y)-\Phi^\star
&\le\langle y_{k-1}-y_k, \, A(x_{k+1}-x_k)\rangle+\dfrac{1}{\beta\tau_k}\langle y-y_k, \,y_k-y_{k-1}\rangle\nonumber\\
&\quad + \frac{\psi}{\tau_{k-1}}\langle S(x_k-u_{k+1}),\,x_{k+1}-x_k\rangle
      + \frac{1}{\tau_k}\langle S(x_{k+1}-u_{k+1}),\,\bar x-x_{k+1}\rangle \nonumber\\
&\quad + \langle T(w_k-w_{k-1}),\,\bar w-w_k\rangle. 
\end{align}
By applying Lemma~\ref{lemm:cosrule}–\eqref{eq:cosrule:fourpoint} to the last four terms in the RHS of \eqref{eq: derive_final_ineq_1}, we obtain
\begin{align}\label{eq:after-fourpoint-y}
\tau_k\left(\mathbb{L}(x_k,w_k,y)-\Phi^\star\right)
&\le \tau_k\langle y_{k-1}-y_k,\,A(x_{k+1}-x_k)\rangle\nonumber\\&
   \quad + \frac{1}{2\beta}\Big(\|y-y_{k-1}\|^{2}-\|y-y_k\|^{2}-\|y_k-y_{k-1}\|^{2}\Big) \nonumber\\
&\quad + \frac{\psi\tau_k}{2\tau_{k-1}}\Big(\|x_{k+1}-u_{k+1}\|_S^2-\|x_k-u_{k+1}\|_S^2-\|x_{k+1}-x_k\|_S^2\Big)\nonumber\\
&\quad + \frac{1}{2}\Big(\|\bar x-u_{k+1}\|_S^2-\|x_{k+1}-u_{k+1}\|_S^2-\|\bar x-x_{k+1}\|_S^2\Big)\nonumber\\
&\quad + \frac{\tau_k}{2}\Big(\|\bar w-w_{k-1}\|_T^2-\|w_k-w_{k-1}\|_T^2-\|\bar w-w_k\|_T^2\Big).
\end{align}
By applying the Cauchy--Schwarz inequality, and from \eqref{eq:steps tau_k} with the fact that $S$ is positive definite (i.e., $\|v\|_{S}^2\geq \lambda_{\min}(S)\|v\|^2$), we have
\begin{align}\label{Cauchy-Schwarz_1}
\tau_k\langle y_{k-1}-y_k,\,A(x_{k+1}-x_k)\rangle
&\le \tau_k\,\|Ax_{k+1}-Ax_k\|\,\|y_{k-1}-y_k\| \nonumber\\
&= \dfrac{\tau_k}{\tau_{k+1}}\tau_{k+1}\|Ax_{k+1}-Ax_k\|\,\|y_{k-1}-y_k\| \nonumber\\
&\le \dfrac{\mu\tau_k}{\tau_{k+1}}\,\sqrt{\frac{\lambda_{\min}(S)}{\beta}}\;\|x_{k+1}-x_k\|\,\|y_{k-1}-y_k\|\nonumber\\
& \leq\dfrac{\mu\tau_k^2}{\tau_{k+1}^2}\frac{\lambda_{\min}(S)}{2}\|x_{k+1}-x_k\|^2 +\frac{\mu}{2\beta}\|y_{k-1}-y_k\|^2\nonumber\\
&\leq\dfrac{\mu\tau_k^2}{2\tau_{k+1}^2}\|x_{k+1}-x_k\|_{S}^2 +\frac{\mu}{2\beta}\|y_{k-1}-y_k\|^2.
\end{align}
Substituing \eqref{Cauchy-Schwarz_1} into \eqref{eq:after-fourpoint-y} yields
\begin{align}\label{eq:after-fourpoint-y_1}
2\tau_k\left(\mathbb{L}(x_k,w_k,y)-\Phi^\star\right)
&\le \dfrac{\mu\tau_k^2}{\tau_{k+1}^2}\|x_{k+1}-x_k\|_{S}^2 +\frac{\mu}{\beta}\|y_{k-1}-y_k\|^2\nonumber\\&
   \quad + \frac{1}{\beta}\Big(\|y-y_{k-1}\|^{2}-\|y-y_k\|^{2}-\|y_k-y_{k-1}\|^{2}\Big) \nonumber\\
&\quad + \frac{\psi\tau_k}{\tau_{k-1}}\Big(\|x_{k+1}-u_{k+1}\|_S^2-\|x_k-u_{k+1}\|_S^2-\|x_{k+1}-x_k\|_S^2\Big)\nonumber\\
&\quad + \Big(\|\bar x-u_{k+1}\|_S^2-\|x_{k+1}-u_{k+1}\|_S^2-\|\bar x-x_{k+1}\|_S^2\Big)\nonumber\\
&\quad + \tau_k\Big(\|\bar w-w_{k-1}\|_T^2-\|w_k-w_{k-1}\|_T^2-\|\bar w-w_k\|_T^2\Big).
\end{align}
By Lemma~\ref{lemm:cosrule}–\eqref{eq:cosrule:convexcomb} and \eqref{eq:grp_u}, we have
\begin{equation}\label{eq:gr-id-k}
\|x_{k+1}-\bar x\|_{S}^{2}
= \frac{\psi}{\psi-1}\,\|u_{k+2}-\bar x\|_{S}^{2}
  - \frac{1}{\psi-1}\,\|u_{k+1}-\bar x\|_{S}^{2}
  + \frac{1}{\psi}\,\|u_{k+1}-x_{k+1}\|_{S}^{2}.
\end{equation}
Using \eqref{eq:gr-id-k} into \eqref{eq:after-fourpoint-y_1} and rearranging, we obtain
\begin{align}\label{eq:step-ineq-k-2}
2\tau_k\!\left(\mathbb{L}(x_k,w_k,y)-\Phi^\star\right)
&\le \frac{\psi}{\psi-1}\big(\|u_{k+1}-\bar x\|_{S}^{2}-\|u_{k+2}-\bar x\|_{S}^{2}\big) + \tau_k\big(\|\bar w-w_{k-1}\|_{T}^{2}\nonumber\\&
-\|\bar w-w_k\|_{T}^{2}\big)
       + \frac{1}{\beta}\big(\|y-y_{k-1}\|^{2}
      -\|y-y_k\|^{2}\big)\nonumber\\&
      - \left(\frac{\psi\tau_k}{\tau_{k-1}}-\frac{\mu\tau_k^{2}}{\tau_{k+1}^{2}}\right)\|x_{k+1}-x_k\|_{S}^{2} + \left(\frac{\psi\tau_k}{\tau_{k-1}}-1-\frac{1}{\psi}\right)\|x_{k+1}-u_{k+1}\|_{S}^{2}\nonumber\\&
      - \tau_k\|w_k-w_{k-1}\|_{T}^{2}
     - \frac{1}{\beta}(1-\mu)\|y_k-y_{k-1}\|^{2}- \frac{\psi\tau_k}{\tau_{k-1}}\|x_k-u_{k+1}\|_{S}^{2}.
\end{align}
From \eqref{eq:steps tau_k}, notice that $\tau_{k}\le \tau_{k-1}$~$\forall k\geq1$.
Let $T_k:=\tau_k T$. Since $T\succ0$, we have $0\prec T_k\preceq T_{k-1}$.
By the monotonicity of weighted norms, it follows that for any $z\in\mathbb{R}^p$, $\|z\|_{T_k}^{2}\le \|z\|_{T_{k-1}}^{2}$. In particular,
\begin{equation}\label{eq:Tk-monotone}
\tau_k\|\bar w-w_{k-1}\|_{T}^{2}
=\|\bar w-w_{k-1}\|_{T_k}^{2}
\le \|\bar w-w_{k-1}\|_{T_{k-1}}^{2}.
\end{equation}
Since $(\tau_k)$ is decreasing and $\psi\in(1,\varphi]$, we have 
\begin{align}\label{eq: non-negativity}
    \frac{\psi\tau_k}{\tau_{k-1}}-1-\frac{1}{\psi} &\leq \psi-1-\frac{1}{\psi}\nonumber\\
    &\leq 0.
\end{align}
Furthermore, by Remark \ref{Rmk_tau_k_bdd_below}, we get $\lim_{k\to\infty}\dfrac{\psi\tau_k}{\tau_{k-1}}-\dfrac{\mu\tau_k^{2}}{\tau_{k+1}^{2}} = \psi-\mu>\mu$. Therefore, there exists a natural number $k_1$ such that 
\begin{equation}\label{eq: limit_1}
    \frac{\psi\tau_k}{\tau_{k-1}}-\frac{\mu\tau_k^{2}}{\tau_{k+1}^{2}}>\mu>0~~\forall k\geq k_1.
\end{equation}
Hence, for all $k\ge k_1$, combining \eqref{eq:Tk-monotone}, \eqref{eq: non-negativity} and \eqref{eq: limit_1} with \eqref{eq:step-ineq-k-2}, we derive Lemma \ref{lem:energy}.
\end{proof}
\begin{theorem}\label{thm:conv}
Let $(\bar x,\bar w,\bar y)$ be a saddle point of~$\mathbb{L}$. Under Assumption~\ref{assump:basic},
let the sequence $\{(x_k,w_k,y_k)\}_{k\ge1}$ be generated by Algorithm~\ref{alg:1}. Then
$\{(x_k,w_k,y_k)\}_{k\ge1}$ converges to a saddle point of $\mathbb{L}$.
\end{theorem}

\begin{proof}
Since $(\bar x, \bar w, \bar y)$ is a saddle point of $\mathbb{L}$, we have $\Phi^\star = \Phi(\bar x, \bar w) = \mathbb{L}(\bar x, \bar w, \bar y)$ as $A\bar x+ B\bar w=b$ and from \eqref{eq:saddle}, we obtain $\mathbb{L}(x_k,w_k, \bar y) - \Phi^\star\geq0~~\forall k$. Now, by applying Lemma~\ref{lem:energy} with $y=\bar y$, we have
\begin{equation*}
     a_{k+1}(\bar y)\leq a_k(\bar y) - b_k,~~\forall k\geq k_1,
\end{equation*}
where
\begin{align}\label{eq:a_k and b_k_kon_incr}
    a_k(\bar y)&:=\frac{\psi}{\psi-1}\|u_{k+1}-\bar x\|_{S}^{2}
     +\|\bar w-w_{k-1}\|_{T_{k-1}}^{2}
     +\frac{1}{\beta}\|\bar y-y_{k-1}\|^{2}\nonumber,\\
b_k&:=\|w_k-w_{k-1}\|_{T_k}^{2}
     +\frac{\psi\tau_k}{\tau_{k-1}}\|x_k-u_{k+1}\|_{S}^{2}+\frac{1}{\beta}(1-\mu)\|y_k-y_{k-1}\|^{2}.
\end{align}
It is easy to observe that both $(a_k)$ and $(b_k)$ are non-negative sequences since $S, T_k\succ0$ for all $k$ and $1-\mu>0$. Hence, by Lemma \ref{lem:fejer-basic},  $\lim_{k\to\infty} a_k(\bar y)$ exists and $\lim_{k\to\infty}b_k=0$. Consequently, we have 
\begin{equation}\label{eq:vanish-incr}
\|w_k-w_{k-1}\|\to0,
\qquad
\|x_k-u_{k+1}\|\to0,
\qquad
\|y_k-y_{k-1}\|\to0.
\end{equation}
Again $x_k-u_k=\psi(x_k-u_{k+1})$ implies $\|x_k-u_k\|\to0$. Since $\{a_k(\bar y)\}$ is bounded and $S,\,T_{k-1}\succ0$, we have that 
$\{u_k\}$, $\{w_k\}$ and $\{y_k\}$ are bounded sequences. Then by \eqref{eq:vanish-incr}, $\{x_k\}$ is bounded. Let 
$\{(x_{k_j},w_{k_j},y_{k_j})\}$ be any subsequence of $\{(x_k, w_k, y_k)\}$ such that $(x_{k_j},w_{k_j},y_{k_j})\to(x^\ast,w^\ast,y^\ast)$ as $j\to\infty$. Then, by running the same arguements as in \eqref{eq:sg-1} and \eqref{eq:sg-2}, we obtain
\begin{align*}
g(x_{k_j})-g(x)
&\le\langle A^\top y_{{k_j}-1},\,x-x_{k_j}\rangle+\tfrac{1}{\tau_{{k_j}-1}}\langle S(x_{k_j}-u_{k_j}),\,x-x_{k_j}\rangle,\\
f(w_{k_j})-f(w)
&\le\langle B^\top y_{k_j},\,w-w_{k_j}\rangle+\langle T(w_{k_j}-w_{k_j-1}),\,w-w_{k_j}\rangle.
\end{align*}
Passing to the limit along $k=k_j$ and using $\|x_k-u_k\|\to0$, $\|w_k-w_{k-1}\|\to0$, 
$\|y_k-y_{k-1}\|\to0$, $\lim_{k\to\infty}\tau_k\ge \underline\tau>0$, and the lower semi-continuity of $f$ and $g$, we have  
\begin{equation}\label{eq:limit-subgrad}
\begin{cases}
g(x^\ast)-g(x)\le\langle A^\top y^\ast,\,x-x^\ast\rangle &\forall x\in\mathbb{R}^q,\\[0.2em]
f(w^\ast)-f(w)\le\langle B^\top y^\ast,\,w-w^\ast\rangle &\forall w\in\mathbb{R}^p,
\end{cases}
\end{equation}
that is,
$0\in\partial g(x^\ast)+A^\top y^\ast$ and $0\in\partial f(w^\ast)+B^\top y^\ast$.
Moreover, by Remark \ref{Rmk_tau_k_bdd_below} and $y_{k_{j}}-y_{k_{j}-1}=\beta\tau_{k_j}(Ax_{k_j}+Bw_{k_j}-b)$, we obtain $A x^\ast+B w^\ast=b$.
Using \eqref{eq:limit-subgrad} and $A x^\ast+B w^\ast=b$, we have
\[
\mathbb{L}(x^\ast,w^\ast,y)\ =\ \Phi(x^\ast,w^\ast)+\langle y,\,Ax^\ast+B w^\ast-b\rangle\ =\ \Phi(x^\ast,w^\ast)
\quad \forall y,
\]
and, for all $(x,w)$
\[
\Phi(x^\ast,w^\ast)-\Phi(x,w)\ \le\ \langle y^\ast,\, A x+B w-b\rangle.
\]
This is equivalent to $\mathbb{L}(x^\ast,w^\ast,y^\ast)\ \le\ \mathbb{L}(x,w,y^\ast)$ for all $(x,w)$. Thus $(x^\ast,w^\ast,y^\ast)$ is a saddle point of $\mathbb{L}$.

Furthermore, from \eqref{eq:a_k and b_k_kon_incr}, notice that $\{a_k(y^\ast)\}$ is non-increasing and convergent. Moreover, along the subsequence $(k_j)$ used above, from \eqref{eq:vanish-incr} and the fact that $(\tau_k)$ is convergent in Remark \ref{Rmk_tau_k_bdd_below}, we have
$u_{k_j+1}\to x^\ast$, $w_{k_j}\to w^\ast$, $y_{k_j}\to y^\ast$ as $j\to\infty$. Thus
\[
\lim_{j\to\infty} a_{k_j}(y^\ast)
=0.
\]
Since $\lim_{k\to\infty} a_k(y^\ast)$ exists and a subsequence
converges to $0$, necessarily $\lim_{k\to\infty} a_k(y^\ast)=0$.
Therefore
\[
\|u_{k+1}-x^\ast\|_S\to0,\qquad
\|w_k-w^\ast\|_{T_k}\to0,\qquad
\|y_k-y^\ast\|\to0.
\]
Using \eqref{eq:vanish-incr}, we also obtain $\|x_k-x^\ast\|\to0$. Furthermore, 
\begin{equation*}
    \|w_k-w^*\|^2_{\bar T} \le \|w_k-w^*\|^2_{T_k} + \|T_k-\bar T\|\|w_k-w^*\|^2.
\end{equation*}
Since $\|w_k-w^\ast\|_{T_k}\to0$,  $\lim_{k\to\infty}T_k=\lim_{k\to\infty}\tau_kT=\bar T\geq \underline{\tau}T\succ0$ and $\{w_k\}$ is bounded, we have $\|w_k-w^\ast\|\to0$ as $k\to\infty$. Therefore, the sequence $\{(x_k,w_k,y_k)\}$ converges to the saddle point $(x^\ast,w^\ast,y^\ast)$ of $\mathbb{L}$. This completes the proof.
\end{proof}
\subsection{Sublinear Rate of Convergence}
In this section, we derive the sublinear rate result for Algorithm \ref{alg:1}.

\begin{theorem}[Ergodic $\mathcal{O}(1/N)$ sublinear rate]\label{thm:ergodic-image}
Let $\{(u_k,x_k,w_k,y_k,\tau_k)\}_{k\ge1}$ be generated by Algorithm~\ref{alg:1}.
Assume that $(x^\star, w^\star, y^\star)$ is a saddle point of $\mathbb{L}$ and $c\geq2\|y^\star\|$ for some $c>0$. Then, there exists $C_1>0$ such that
\[
\big|\Phi(\tilde x_N,\tilde w_N)-\Phi^\star\big|\le \frac{C_1}{N'}~~~\text{and}~~~
\|A\tilde x_N+B\tilde w_N-b\|\le \frac{2C_1}{c\,N'},
\]
where
\begin{equation}\label{eq: ergodic_x_k_w_k}
\tilde x_N:=\frac{1}{N'}\sum_{k=k_1}^{N}x_k,\quad
\tilde w_N:=\frac{1}{N'}\sum_{k=k_1}^{N}w_k,~~~\text{and}~~~
N':=N-k_1+1.    
\end{equation}
\end{theorem}

\begin{proof}
Let $y\in\mathbb{R}^m$ be arbitrary. From \eqref{eq:step-ineq-k}, since the three terms
$-\|w_k-w_{k-1}\|_{T_k}^{2}$, $-\tfrac{1}{\beta}(1-\mu)\|y_k-y_{k-1}\|^{2}$, and
$-\tfrac{\psi\tau_k}{\tau_{k-1}}\|x_k-u_{k+1}\|_{S}^{2}$ are nonpositive, discarding them gives
\begin{align}\label{telescope_eqn}
 2\tau_k\big(\mathbb{L}(x_k,w_k,y)-\Phi^\star\big)&\le
\frac{\psi}{\psi-1}\big(\|u_{k+1}-x^\star\|_{S}^{2}-\|u_{k+2}-x^\star\|_{S}^{2}\big)
+\big(\|w^\star-w_{k-1}\|_{T_{k-1}}^{2}\nonumber\\&
-\|w^\star-w_k\|_{T_k}^{2}\big)
+\frac{1}{\beta}\big(\|y-y_{k-1}\|^{2}-\|y-y_k\|^{2}\big)~~\forall k\ge k_1.
\end{align}
Summing over $k=k_1,\dots,N$, telescopes the right-hand side of \eqref{telescope_eqn} to
\begin{equation}\label{eq:Theta-sum}
\sum_{k=k_1}^{N} 2\tau_k\big(\mathbb{L}(x_k,w_k,y)-\Phi^\star\big)\ \le\ 
\Delta_{k_1}(y),
\end{equation}
where
\[
\Delta_{k_1}(y):=\frac{\psi}{\psi-1}\|u_{k_1+1}-x^\star\|_{S}^{2}
+\|w^\star-w_{k_1-1}\|_{T_{k_1-1}}^{2}
+\frac{1}{\beta}\|y-y_{k_1-1}\|^{2}.
\]
By Remark \ref{Rmk_tau_k_bdd_below}, we have
$\tau_k\ge \underline\tau$. Thus applying this to \eqref{eq:Theta-sum} yields
\[
\sum_{k=k_1}^{N}\big(\mathbb{L}(x_k,w_k,y)-\Phi^\star\big)\ \le\ \frac{\Delta_{k_1}(y)}{2\underline\tau}.
\]
By dividing $N'$ and using the convexity of $\mathbb{L}(\cdot,\cdot,y)$ in $(x,w)$ for any $y$, we obtain
\begin{equation*}
\mathbb{L}(\tilde x_N,\tilde w_N,y)-\Phi^\star \ \le\ \frac{\Delta_{k_1}(y)}{2\underline\tau\,N'}.
\end{equation*}
Since $\mathbb{L}(\tilde x_N,\tilde w_N,y)=\Phi(\tilde x_N,\tilde w_N)+\langle y,\,A\tilde x_N+B\tilde w_N-b\rangle$, we have
\begin{equation}\label{eq:phi-plus-lin}
\Phi(\tilde x_N,\tilde w_N)-\Phi^\star
+\langle y,\,A\tilde x_N+B\tilde w_N-b\rangle
\ \le\ \frac{\Delta_{k_1}(y)}{2\underline\tau\,N'}.
\end{equation}
Now, taking maximum to the both sides of \eqref{eq:phi-plus-lin} over the ball $\{y:\|y\|\le c\}$, and using
$\max_{\|y\|\le c}\langle y,r\rangle=c\|r\|$ and
$\max_{\|y\|\le c}\|y-y_{k_1-1}\|^2\le(\|y_{k_1-1}\|+c)^2$ yields
\begin{equation}\label{eq:24-analogue}
\Phi(\tilde x_N,\tilde w_N)-\Phi^\star + c\|A\tilde x_N+B\tilde w_N-b\|
\ \le\ \frac{C_1}{N'},
\end{equation}
where
\[
C_1:=\frac{1}{2\underline\tau}\left(
\frac{\psi}{\psi-1}\|u_{k_1+1}-x^\star\|_{S}^{2}
+\|w^\star-w_{k_1-1}\|_{T_{k_1-1}}^{2}
+\frac{1}{\beta}\,(\|y_{k_1-1}\|+c)^2
\right).
\]
Therefore, from \eqref{eq:24-analogue}, we have
\begin{equation}\label{eq:25-analogue}
\Phi(\tilde x_N,\tilde w_N)-\Phi^\star \ \le\ \frac{C_1}{N'}.
\end{equation}
Since $(x^\star,w^\star,y^\star)$ is a saddle point
\begin{align*}
\Phi^\star=\mathbb{L}(x^\star,w^\star,y^\star)\ \le\ \mathbb{L}(\tilde x_N,\tilde w_N,y^\star)
&=\Phi(\tilde x_N,\tilde w_N)+\langle y^\star,\,A\tilde x_N+B\tilde w_N-b\rangle\nonumber\\&
\ \le\ \Phi(\tilde x_N,\tilde w_N)+\|y^\star\|\,\|A\tilde x_N+B\tilde w_N-b\|.
\end{align*}
As $c\ge 2\|y^\star\|$, we have $\Phi^\star-\Phi(\tilde x_N,\tilde w_N)\le \tfrac{c}{2}\|A\tilde x_N+B\tilde w_N-b\|$.
Combining this result with \eqref{eq:24-analogue} and~\eqref{eq:25-analogue} gives
\[
\|A\tilde x_N+B\tilde w_N-b\|
\ \le\ \frac{2C_1}{cN'}~~~~\text{and}\quad
\big|\Phi(\tilde x_N,\tilde w_N)-\Phi^\star\big|\ \le\ \frac{C_1}{N'}.
\]
Thus, we obtain the desired result.
\end{proof}

\subsection{Pointwise rate of convergence}
In this subsection, we establish the \emph{pointwise} convergence rate for Algorithm~\ref{alg:1}. In contrast to the ergodic convergence rate proved in Theorem~\ref{thm:ergodic-image}, which is based on the averaged iterates, the result below concerns the \emph{actual iterates} generated by Algorithm \ref{alg:1}. More precisely, we show that among the first $N$ iterates, there exists at least one iterate whose primal feasibility residual and dual optimality residuals are of order $\mathcal{O}(1/\sqrt{N})$. This provides the standard best-iterate pointwise complexity bound for ADMM-type methods in the convex setting; see \cite{adona2019iteration,he2015non,gonccalves2018pointwise} for more details.

To facilitate this, for each $k\ge 1$, we define the primal feasibility residual
\[
r_k^p:=Ax_k+Bw_k-b.
\]
Moreover, using the optimality conditions \eqref{eq:opt-x}--\eqref{eq:opt-w}, we have
\begin{align}
\eta_k^x&:= -\frac{1}{\tau_{k-1}}S(x_k-u_k)+A^\top(y_k-y_{k-1}),
\label{eq:def-etax}\\
\eta_k^w&:= -T(w_k-w_{k-1}).
\label{eq:def-etaw}
\end{align}
This is equivalent to $\eta_k^x\in \partial g(x_k)+A^\top y_k$
and $\eta_k^w\in \partial f(w_k)+B^\top y_k.$ Accordingly, we define the \emph{pointwise residual measure} (prm)
\begin{equation*}
\mathcal{R}_k:=
\|r_k^p\|^2
+\operatorname{dist}^2\!\bigl(0,\partial g(x_k)+A^\top y_k\bigr)
+\operatorname{dist}^2\!\bigl(0,\partial f(w_k)+B^\top y_k\bigr).
\end{equation*}

\begin{theorem}\label{thm:pointwise_alg1_clean}
Let $(x^\star,w^\star,y^\star)$ be a saddle point of $\mathbb{L}$, and let
$\{(u_k,x_k,w_k,y_k,\tau_k)\}_{k\ge1}$ be generated by Algorithm~\ref{alg:1}.
Let $k_1\in\mathbb{N}$ be the natural number given by Lemma~\ref{lem:energy}, and set
\[
A_0:=a_{k_1}(y^\star),\qquad
N':=N-k_1+1.
\]
Then, for every $N\ge k_1$, there exists an index $j_N\in\{k_1,\dots,N\}$ such that
\begin{align}
\mathbb{L}(x_{j_N},w_{j_N},y^\star)-\Phi^\star
&\le \frac{A_0}{2\underline{\tau}\,N'},\label{eq:pointwise-gap-clean}\\
\|r_{j_N}^p\|
&\le
\sqrt{\frac{A_0}{\beta\underline{\tau}^2(1-\mu)\,N'}},\label{eq:pointwise-rp-clean}\\
\operatorname{dist}\bigl(0,\partial f(w_{j_N})+B^\top y_{j_N}\bigr)
&\le
\sqrt{\frac{A_0\|T\|}{\underline{\tau}\,N'}},\label{eq:pointwise-distf-clean}\\
\operatorname{dist}\bigl(0,\partial g(x_{j_N})+A^\top y_{j_N}\bigr)
&\le
\sqrt{
\frac{1}{N'}
\left(
\frac{2A_0\psi\|S\|\tau_0}{\underline{\tau}^3}
+
\frac{2A_0\beta\|A\|^2}{1-\mu}
\right)}.\label{eq:pointwise-distg-clean}
\end{align}

Consequently, there exists a constant $C_{\rm prm}>0$ such that $\mathcal R_{j_N}\le \frac{C_{\rm prm}}{N'},$ where
\begin{equation}\label{eq:def-cprm-pointwise}
C_{\rm prm}:=
A_0\left(
\frac{1}{\beta\underline{\tau}^2(1-\mu)}
+\frac{\|T\|}{\underline{\tau}}
+\frac{2\psi\|S\|\tau_0}{\underline{\tau}^3}
+\frac{2\beta\|A\|^2}{1-\mu}
\right).
\end{equation}
\end{theorem}

\begin{proof}
Since $(x^\star,w^\star,y^\star)$ is a saddle point of $\mathbb{L}$, we have
$\mathbb{L}(x_k,w_k,y^\star)-\Phi^\star\ge0$ for all $k\ge1$. Hence, by Lemma~\ref{lem:energy}
with $y=y^\star$, for all $k\ge k_1$,
\begin{equation}\label{eq:pw-clean-1}
2\tau_k\bigl(\mathbb{L}(x_k,w_k,y^\star)-\Phi^\star\bigr)
+\|w_k-w_{k-1}\|_{T_k}^2
+\frac{1-\mu}{\beta}\|y_k-y_{k-1}\|^2
+\frac{\psi\tau_k}{\tau_{k-1}}\|x_k-u_{k+1}\|_S^2
\le a_k(y^\star)-a_{k+1}(y^\star),
\end{equation}
where 
\begin{equation*}
a_{k}(y^\star)
=
\frac{\psi}{\psi-1}\|u_{k+1}-x^\star\|_S^2
+\|w^\star-w_{k-1}\|_{T_{k-1}}^2
+\frac1\beta\|y^\star-y_{k-1}\|^2.
\end{equation*}
By applying Remark \ref{Rmk_tau_k_bdd_below} and using $(\tau_k)$ is a decreasing sequence, we obtain $\frac{\psi\tau_k}{\tau_{k-1}}\ge \frac{\psi\underline\tau}{\tau_0}~\forall k$. Now, summing \eqref{eq:pw-clean-1} from $k=k_1$ to $N$, and using Remark~\ref{Rmk_tau_k_bdd_below}, we have
\begin{align}\label{eq:pw-clean-sum}
2\underline\tau\sum_{k=k_1}^N\bigl(\mathbb{L}(x_k,w_k,y^\star)-\Phi^\star\bigr)
+\sum_{k=k_1}^N\|w_k-w_{k-1}\|_{T_k}^2
&+\frac{1-\mu}{\beta}\sum_{k=k_1}^N\|y_k-y_{k-1}\|^2\nonumber\\
+\frac{\psi\underline\tau}{\tau_0}\sum_{k=k_1}^N\|x_k-u_{k+1}\|_S^2
&\le A_0.
\end{align}
Since $1-\mu>1-\frac{\psi}{2}>0$ and $\tau_k>0~\forall k$, there exists $j_N\in\{k_1,\dots,N\}$ such that $\mathbb{L}(x_{j_N},w_{j_N},y^\star)-\Phi^\star
\le~\frac{A_0}{2\underline\tau\,N'}$. This proves \eqref{eq:pointwise-gap-clean}. Furthermore, from \eqref{eq:pw-clean-sum}, we have
\begin{align}
\|w_{j_N}-w_{j_N-1}\|_{T_{j_N}}^2
&\le \frac{A_0}{N'}, \label{eq:pw-w-est}\\
\|y_{j_N}-y_{j_N-1}\|^2
&\le \frac{A_0\beta}{(1-\mu)N'}, \label{eq:pw-y-est}\\
\|x_{j_N}-u_{j_N+1}\|_S^2
&\le \frac{A_0\tau_0}{\psi\underline\tau\,N'}. \label{eq:pw-x-est}
\end{align}
Next, from \eqref{eq:grp_y},
$r_k^p=\frac{1}{\beta\tau_k}(y_k-y_{k-1})$, and hence
\[
\|r_{j_N}^p\|^2
=
\frac{1}{\beta^2\tau_{j_N}^2}\|y_{j_N}-y_{j_N-1}\|^2
\le
\frac{1}{\beta^2\underline\tau^2}\|y_{j_N}-y_{j_N-1}\|^2.
\]
Therefore using \eqref{eq:pw-y-est}, we obtain \eqref{eq:pointwise-rp-clean}. Further, since $\eta_{j_N}^w\in\partial f(w_{j_N})+B^\top y_{j_N}$, we have
\[
\operatorname{dist}\bigl(0,\partial f(w_{j_N})+B^\top y_{j_N}\bigr)
\le \|\eta_{j_N}^w\|.
\]
By \eqref{eq:def-etaw} and $T_{j_N}=\tau_{j_N}T$
\begin{align*}
\|\eta_{j_N}^w\|^2
=
\|T(w_{j_N}-w_{j_N-1})\|^2
&\le
\|T\|\,\|w_{j_N}-w_{j_N-1}\|_T^2\nonumber\\
&=
\frac{\|T\|}{\tau_{j_N}}\|w_{j_N}-w_{j_N-1}\|_{T_{j_N}}^2
\le
\frac{\|T\|}{\underline\tau}\|w_{j_N}-w_{j_N-1}\|_{T_{j_N}}^2.
\end{align*}
Combining this with \eqref{eq:pw-w-est} gives \eqref{eq:pointwise-distf-clean}. Finally, since $\eta_{j_N}^x\in\partial g(x_{j_N})+A^\top y_{j_N}$
\[
\operatorname{dist}\bigl(0,\partial g(x_{j_N})+A^\top y_{j_N}\bigr)
\le \|\eta_{j_N}^x\|.
\]
Using \eqref{eq:def-etax}, the identity $x_{j_N}-u_{j_N}=\psi(x_{j_N}-u_{j_N+1})$, and the inequality
$(a+b)^2\le 2a^2+2b^2$, we obtain
\begin{align*}
\|\eta_{j_N}^x\|^2
&=
\left\|-\frac{1}{\tau_{j_N-1}}S(x_{j_N}-u_{j_N})+A^\top(y_{j_N}-y_{j_N-1})\right\|^2\\
&\le
\frac{2\psi^2}{\tau_{j_N-1}^2}\|S(x_{j_N}-u_{j_N+1})\|^2
+
2\|A\|^2\|y_{j_N}-y_{j_N-1}\|^2\\
&\le
\frac{2\psi^2\|S\|}{\underline\tau^2}\|x_{j_N}-u_{j_N+1}\|_S^2
+
2\|A\|^2\|y_{j_N}-y_{j_N-1}\|^2.
\end{align*}
Now applying \eqref{eq:pw-x-est} and \eqref{eq:pw-y-est} yields
\[
\|\eta_{j_N}^x\|^2
\le
\frac{1}{N'}
\left(
\frac{2A_0\psi\|S\|\tau_0}{\underline\tau^3}
+
\frac{2A_0\beta\|A\|^2}{1-\mu}
\right),
\]
which proves \eqref{eq:pointwise-distg-clean}. Moreover, since
\[
\mathcal R_{j_N}
=
\|r_{j_N}^p\|^2
+\operatorname{dist}^2\bigl(0,\partial g(x_{j_N})+A^\top y_{j_N}\bigr)
+\operatorname{dist}^2\bigl(0,\partial f(w_{j_N})+B^\top y_{j_N}\bigr),
\]
adding the three derived estimates \eqref{eq:pointwise-rp-clean}--\eqref{eq:pointwise-distg-clean} yields $\mathcal R_{j_N}\le \frac{C_{\rm prm}}{N'},$ where $C_{\rm prm}$ is given by \eqref{eq:def-cprm-pointwise}. This completes the proof.
\end{proof}

\subsection{A new extended GrpADMM algorithm}

In this section, we extend the admissible range of the parameter $\psi$ from $(1,\varphi]$ to $(1,1+\sqrt{3})$. This extension may be advantageous in practice, since a larger value of $\psi$ allows the iterates $x_k$ and $u_k$ to remain closer during the iterations.

\begin{algorithm}[H]
\caption{Extended \ref{eq:GrpADMM} with decreasing step-size for solving \eqref{eq:model}}\label{alg:extened psi}
\KwIn{Let $S\in\mathbb{S}^{q}_{++}$ and $T\in\mathbb{S}^{p}_{++}$. Choose $x_{0}\in\mathbb{R}^{q}$, $w_{0}\in\mathbb{R}^{p}$, and $y_{0}\in\mathbb{R}^{m}$ with $u_{0}=x_{0}$. Let $\tau_0>0$, $\beta>0$, $\psi\in(1,1+\sqrt 3)$, and $0<\mu<\frac{\psi(2+2\psi-\psi^2)}{2(\psi+1)}$.}

\For{$k=1,2,\ldots$}{
  \textbf{Step 1} (Compute)
  \begin{equation*}
    u_{k}=\frac{\psi-1}{\psi}\,x_{k-1}+\frac{1}{\psi}\,u_{k-1},
  \end{equation*}
  \begin{equation*}
    x_{k}=\argmin_{x}\left\{\mathbb{L}(x,w_{k-1},y_{k-1})+\frac{1}{2\tau_{k-1}}\|x-u_{k}\|_{S}^{2}\right\}.
  \end{equation*}

  \textbf{Step 2} (Update)
  \begin{equation*}
    \tau_k=\min\left\{\tau_{k-1},\,\frac{\mu\sqrt{\lambda_{\min}(S)}}{\sqrt{\beta}\,L_k}\right\}.
  \end{equation*}

  \textbf{Step 3} (Compute)
  \begin{equation*}
    w_{k}=\argmin_{w}\left\{\mathbb{L}_{\beta\tau_k}(x_{k},w,y_{k-1})+\frac{1}{2}\|w-w_{k-1}\|_{T}^{2}\right\},
  \end{equation*}
  \begin{equation*}
    y_{k}=y_{k-1}+\beta\tau_k\big(Ax_{k}+Bw_{k}-b\big).
  \end{equation*}
}
\end{algorithm}

\begin{remark}\label{Remark on extended psi}
Note that the extension of the range of $\psi$ in Algorithm~\ref{alg:extened psi} does not affect the conclusion of Remark~\ref{Rmk_tau_k_bdd_below}. In particular, the sequence $(\tau_k)$ remains convergent and converges to a positive limit.
\end{remark}

\begin{remark}\label{Rmk:tau_fixed_case_extended}
Let $c_\psi:=\frac{\psi(2+2\psi-\psi^2)}{\psi+1}.$ If $\tau_0\le \frac{\mu\sqrt{\lambda_{\min}(S)}}{\sqrt{\beta}\|A\|}$, then, by an argument analogous to that in Remark~\ref{Rmk:tau fixed case}, we obtain $\tau_k=\tau_0$ and $\sigma_k=\beta\tau_k=\beta\tau_0$ for all $k\ge 1$. Therefore, using $0<\mu<\frac{c_{\psi}}{2}<1$, we get
\[
\tau_k\sigma_k\|A\|^2
=
\beta\tau_0^2\|A\|^2
\le
\mu^2\lambda_{\min}(S)<\mu\lambda_{\min}(S)
<c_\psi\,\lambda_{\min}(S).
\]
Hence, Algorithm~\ref{alg:extened psi} reduces to a new extended version corresponding to the \ref{eq:GrpADMM} algorithm with fixed step-size.
\end{remark}

\begin{theorem}\label{thm:conv_1}
Let $(\bar x,\bar w,\bar y)$ be a saddle point of $\mathbb{L}$. Under Assumption~\ref{assump:basic}, let $\{(x_k,w_k,y_k)\}_{k\ge1}$ be generated by Algorithm~\ref{alg:extened psi}. Then $\{(x_k,w_k,y_k)\}_{k\ge1}$ converges to a saddle point of $\mathbb{L}$.
\end{theorem}

\begin{proof}
We follow the proof of Lemma~\ref{lem:energy} up to \eqref{eq:step-ineq-k-2}. For convenience, set
$D_k:=\|x_{k+1}-u_{k+1}\|_S$, $B_k:=\|x_k-u_{k+1}\|_S$, $C_k:=\|x_{k+1}-x_k\|_S$, and $\theta_k:=\frac{\tau_k}{\tau_{k-1}}$. We claim that there exists $\bar k\in\mathbb{N}$ such that $\forall k\ge \bar k$
\begin{equation}\label{eq:key_estimate_extpsi_rewrite}
(1+\tfrac1\psi-\psi\theta_k)D_k^2+\psi\theta_k B_k^2
\ge
\frac{\psi\theta_k(1+\psi-\psi^2\theta_k)}{\psi+1}\,C_k^2.
\end{equation}
From this point, we distinguish two cases.

\smallskip
\noindent\textbf{Case 1:} Let $\psi\in(1,\varphi]$. Since $(\tau_k)$ is decreasing, $\theta_k\le 1$, and 
$1+\frac1\psi-\psi\theta_k \ge 1+\frac1\psi-\psi \ge 0$. Applying Lemma~\ref{usefullemma_3} with $p=1+\frac1\psi-\psi\theta_k$, $q=\psi\theta_k$, and noting $C_k\le D_k+B_k$, we obtain
\[
\frac{\psi\theta_k(1+\psi-\psi^2\theta_k)}{\psi+1}(D_k+B_k)^2
\le
\left(1+\frac1\psi-\psi\theta_k\right)D_k^2+\psi\theta_k B_k^2.
\]
Since $\psi\in(1,\varphi]$ and $\theta_k\le 1$, the coefficient $\frac{\psi\theta_k(1+\psi-\psi^2\theta_k)}{\psi+1}$ is nonnegative. Hence, using $C_k\le D_k+B_k$, we get \eqref{eq:key_estimate_extpsi_rewrite}.

\smallskip
\noindent\textbf{Case 2:} Let $\psi\in(\varphi,1+\sqrt3)$. By Remark~\ref{Remark on extended psi}, the sequence $(\tau_k)$ converges to a positive limit, and therefore $\theta_k\to 1$ as $k\to\infty$. Hence $\lim_{k\to\infty}\psi^2\theta_k-\psi-1 = \psi^2-\psi-1>0$. Thus there exists $\bar k\in\mathbb N$ such that $\psi^2\theta_k-\psi-1>0$ for all $k\ge \bar k$. Moreover, noting $D_k\le C_k+B_k$, and applying Lemma~\ref{usefullemma_3} with
$p=\frac{\psi\theta_k(\psi^2\theta_k-\psi-1)}{\psi+1}$ and $q=\psi\theta_k$, we get, for all $k\ge \bar k$
\[
\frac{pq}{p+q}D_k^2 \le pC_k^2+qB_k^2.
\]
Since
$\frac{pq}{p+q}=\frac{\psi^2\theta_k-\psi-1}{\psi}=\psi\theta_k-1-\frac1\psi$,
this becomes
\[
\left(\psi\theta_k-1-\frac1\psi\right)D_k^2
\le
\frac{\psi\theta_k(\psi^2\theta_k-\psi-1)}{\psi+1}C_k^2+\psi\theta_k B_k^2,
\]
which is again equivalent to \eqref{eq:key_estimate_extpsi_rewrite}. Therefore,  for all $k\ge\bar k$, combining \eqref{eq:key_estimate_extpsi_rewrite} with \eqref{eq:step-ineq-k-2}, we obtain
\begin{align}
2\tau_k\bigl(\mathbb{L}(x_k,w_k,y)-\Phi^\star\bigr)
&\le
\frac{\psi}{\psi-1}\bigl(\|u_{k+1}-\bar x\|_S^2-\|u_{k+2}-\bar x\|_S^2\bigr)
+\tau_k\bigl(\|\bar w-w_{k-1}\|_T^2-\|\bar w-w_k\|_T^2\bigr) \nonumber\\
&\quad
+\frac1\beta\bigl(\|y-y_{k-1}\|^2-\|y-y_k\|^2\bigr) \nonumber\\
&\quad
-\left(\frac{\psi\theta_k(2+2\psi-\psi^2\theta_k)}{\psi+1}
-\mu\frac{\tau_k^2}{\tau_{k+1}^2}\right)\|x_{k+1}-x_k\|_S^2 \nonumber\\
&\quad
-\tau_k\|w_k-w_{k-1}\|_T^2
-\frac{1-\mu}{\beta}\|y_k-y_{k-1}\|^2.
\label{eq:step_ineq_extpsi_final}
\end{align}
Moreover, by Remark~\ref{Remark on extended psi}, we have $\theta_k\to 1$ and $\tau_k^2/\tau_{k+1}^2\to 1$ as $k\to\infty$. Hence
\[
\lim_{k\to\infty}\frac{\psi\theta_k(2+2\psi-\psi^2\theta_k)}{\psi+1}
-\mu\frac{\tau_k^2}{\tau_{k+1}^2}
=
\frac{\psi(2+2\psi-\psi^2)}{\psi+1}-\mu.
\]
Since $0<\mu<\frac{\psi(2+2\psi-\psi^2)}{2(\psi+1)}$, it follows that $\frac{\psi(2+2\psi-\psi^2)}{\psi+1}-\mu>\mu$. Therefore, there exists $k_0\in\mathbb N$ such that
\[
\frac{\psi\theta_k(2+2\psi-\psi^2\theta_k)}{\psi+1}
-\mu\frac{\tau_k^2}{\tau_{k+1}^2}
>\mu
\quad \forall k\ge k_0.
\]
Let $k_*:=\max\{\bar k,k_0\}$. Since $(\bar x,\bar w,\bar y)$ is a saddle point of $\mathbb L$, we have $\mathbb L(x_k,w_k,\bar y)-\Phi^\star\ge 0$ for all $k\ge 1$. Also, $(\tau_k)$ is decreasing, and with $T_k:=\tau_kT$ we have $T_k\preceq T_{k-1}$. Thus, taking $y=\bar y$ in \eqref{eq:step_ineq_extpsi_final}, we obtain for all $k\ge k_*$,
$a_{k+1}(\bar y)\le a_k(\bar y)-\Delta_k$,
where
\begin{align*}
a_k(\bar y)
&:=
\frac{\psi}{\psi-1}\|u_{k+1}-\bar x\|_S^2
+\|\bar w-w_{k-1}\|_{T_{k-1}}^2
+\frac1\beta\|\bar y-y_{k-1}\|^2, \\
\Delta_k
&:=
\mu\|x_{k+1}-x_k\|_S^2
+\|w_k-w_{k-1}\|_{T_k}^2
+\frac{1-\mu}{\beta}\|y_k-y_{k-1}\|^2.
\end{align*}
Hence $\{a_k(\bar y)\}$ is non-increasing, and by Lemma~\ref{lem:fejer-basic}, we get $\sum_{k=k_*}^\infty \|x_{k+1}-x_k\|_S^2<\infty$. It remains to prove that $\|x_k-u_k\|_S\to 0$. Using \eqref{eq:grp_u}, we have
$x_{k+1}-u_{k+1}=\frac1\psi(x_k-u_k)+(x_{k+1}-x_k)$.
Therefore, for any $\varepsilon>0$, Young's inequality yields
\[
\|x_{k+1}-u_{k+1}\|_S^2
\le
\frac{1+\varepsilon}{\psi^2}\|x_k-u_k\|_S^2
+
\left(1+\frac1\varepsilon\right)\|x_{k+1}-x_k\|_S^2.
\]
Now choose $\varepsilon>0$ such that $\frac{1+\varepsilon}{\psi^2}<1$, which is possible since $\psi>1$. As $\sum_{k=k_*}^\infty \|x_{k+1}-x_k\|_S^2<~\infty$, Lemma~\ref{lemm: sum un and vn} gives $\sum_{k=1}^\infty \|x_k-u_k\|_S^2<\infty$, and hence $\|x_k-u_k\|_S\to 0$. The remainder of the convergence argument is the same as in the final part of the proof of Theorem~\ref{thm:conv}.
\end{proof}
\section{A non-decreasing step-size strategy to solve \eqref{eq:model}}\label{sec_5}
This section is devoted to proposing a non-decreasing step-size rule for a simple modification of \ref{eq:GrpADMM} algorithm, which can efficiently solve \eqref{eq:model} without requiring prior knowledge of the norm of the operator $A$ or any complicated hyperparameter tuning. The algorithm is formally written as follows.

\begin{algorithm}[H]
\caption{A non-decreasing step-size rule for \ref{eq:GrpADMM} to solve \eqref{eq:model}}\label{alg:2}
\KwIn{Let $S\in\mathbb{S}^{q}_{++}$, $T\in\mathbb{S}^{p}_{++}$. Choose $x_{0}\in\mathbb{R}^{q}$, $w_{0}\in\mathbb{R}^{p}$, $y_{0}\in\mathbb{R}^{m}$ with $u_0=x_0$.  Let $\tau_0>0$, $\beta>0$, $\psi\in(1,\varphi)$,  $\rho\in[1,1/\psi+1/\psi^2]$, and $0<r_1<r<\frac{\rho}{2}$. Let $(\xi_k)$ be a sequence such that $\rho+\xi_k>1~\forall k$ and $\sum_{k=1}^\infty \log(\rho+\xi_k)<+\infty.$ Set $\bar\lambda = \lambda_{\min}(S)$.}

\For{$k=1,2,\ldots$}{
  \textbf{Step 1} (Compute)
  \begin{equation*}
    u_{k}=\frac{\psi-1}{\psi}\,x_{k-1}+\frac{1}{\psi}\,u_{k-1},
  \end{equation*}
    \begin{equation*}
    x_{k}=\argmin_{x}\!\left\{\mathbb{L}(x,w_{k-1},y_{k-1})+\frac{1}{2\tau_{k-1}}\|x-u_{k}\|_{S}^{2}\right\}.
  \end{equation*}

  \textbf{Step 2} (Update)
\begin{equation}\label{eq:steps_alg_2_tau_k}
      \tau_k = \begin{cases}
       \frac{r_1 \sqrt{\bar\lambda}}{\sqrt{\beta}L_k}, & \text{if}~~\tau_{k-1}L_k>\frac{r\sqrt{\bar\lambda}}{\sqrt{\beta}}\\
        (\rho+\xi_{k-1})\,\tau_{k-1}, & \text{otherwise}
      \end{cases},\qquad\sigma_k =\beta\tau_k
\end{equation}

  \textbf{Step 3} (Compute)
 \begin{equation}\label{eq:alg_2_grp_w}
    w_{k}=\argmin_{w}\!\left\{\mathbb{L}_{\sigma
    _k}(x_{k},w,y_{k-1})+\frac{1}{2\sigma_k}\|w-w_{k-1}\|_{T}^{2}\right\},
  \end{equation}
  \begin{equation*}
    y_{k}=y_{k-1}+\sigma
    _k\big(Ax_{k}+Bw_{k}-b\big).
  \end{equation*}
}
\end{algorithm}

A few comments regarding Algorithm \ref{alg:2} are in order.
\begin{remark}\label{rmk_3}
    Note that for theoretical guarantee of convergence, \eqref{psithetanminus1} should hold and for that, we need $1<\frac{1}{\psi}+\frac{1}{\psi^2}$. Hence, in practice, we can take $\psi$ closer to the golden ratio ($\varphi$) instead of taking $\psi=\varphi$, so that the above assertion holds. Furthermore, note that in the $w$–update (see \eqref{eq:alg_2_grp_w}), the proximal term is scaled by $\sigma_k^{-1}$. This scaling is deliberate as it produces the telescoping term $\|\bar w-w_{k-1}\|_{T}^{2}-\|\bar w-w_k\|_{T}^{2}$ in Lemma~\ref{lemma_2}, and hence the sequence $\{\|w_k-\bar w\|_{T}\}$ become F\'ejer monotone combined with the other energy terms.

\end{remark}
\begin{remark}\label{remark_1}
Let $\tau_{\min}:=\min\left\{\tau_0,\frac{\sqrt{\bar\lambda}\,r_1}{\sqrt{\beta}\,\|A\|}\right\}.$ Then the sequence $(\tau_k)$ generated by Algorithm~\ref{alg:2} satisfies $\tau_k\ge \tau_{\min}~~ \forall k\ge 0.$ Indeed, since $L_k\le \|A\|,$ the first branch of \eqref{eq:steps_alg_2_tau_k} yields
\[
\tau_k=\frac{\sqrt{\bar\lambda}\,r_1}{\sqrt{\beta}\,L_k}
\ge
\frac{\sqrt{\bar\lambda}\,r_1}{\sqrt{\beta}\,\|A\|}.
\]
On the other hand, in the second branch, since $\rho+\xi_{k-1}> 1$, we have
\[
\tau_k\ge \min\left\{\tau_{k-1},\frac{\sqrt{\bar\lambda}\,r_1}{\sqrt{\beta}\,\|A\|}\right\}.
\]
Therefore, the conclusion follows by induction.
\end{remark}

\begin{remark}\label{remark:xi_choice}
In Algorithm~\ref{alg:2}, the sequence $(\xi_k)$ is chosen so that the step-size sequence $(\tau_k)$ converges; see Remark~\ref{remark:tau_convergence}. At the same time, the steps are not decreasing, which allows the method to adapt more flexibly to the local behaviour of the operator and may improve stability when the iterates pass through relatively flat regions. The assumptions imposed on $(\xi_k)$ are nonempty. For instance, one may take
$\xi_k:=1-\rho+\frac{1}{(k+1)^a},~~a>1.$ Then
$\rho+\xi_k=1+\frac{1}{(k+1)^a}>1$ for all $k\ge 0$. Moreover, since
\[
0<\log\!\left(1+\frac{1}{(k+1)^a}\right)\le \frac{1}{(k+1)^a},
\]
and $\sum_{k=0}^\infty \frac{1}{(k+1)^a}<\infty$ for $a>1$, it follows that
\[
\sum_{k=0}^\infty \log(\rho+\xi_k)
=
\sum_{k=0}^\infty \log\!\left(1+\frac{1}{(k+1)^a}\right)
<+\infty.
\]
\end{remark}

\begin{remark}\label{remark:tau_convergence}
Next we show that the sequence $(\tau_k)$ generated by Algorithm~\ref{alg:2} converges to some $\tau^\star\in[\tau_{\min},+\infty)$. By Remark~\ref{remark_1}, we have $\tau_k\ge \tau_{\min}>0$ for all $k\ge 0$. Let
$$
d_k:=\log\tau_k-\log\tau_{k-1}, \quad k\ge 1.
$$
We claim that $d_k\le \log(\rho+\xi_{k-1}) ~~ \forall k\ge 1.$ Indeed, if the second branch in \eqref{eq:steps_alg_2_tau_k} is chosen, then
$$
d_k=\log\!\left(\frac{\tau_k}{\tau_{k-1}}\right)=\log(\rho+\xi_{k-1}).
$$
Furthermore, if the first branch is chosen, then by the condition $\tau_{k-1}L_k>\frac{r\sqrt{\bar\lambda}}{\sqrt{\beta}},$ we obtain
$$
\tau_k=\frac{r_1\sqrt{\bar\lambda}}{\sqrt{\beta}L_k}
<
\frac{r_1}{r}\tau_{k-1}.
$$
Since $0<r_1<r$, we have
$$
d_k=\log\!\left(\frac{\tau_k}{\tau_{k-1}}\right)
<
\log\!\left(\frac{r_1}{r}\right)<0<\log(\rho+\xi_{k-1}).
$$
Therefore, $d_k^+:=\max\{0,d_k\}\le \log(\rho+\xi_{k-1}),$ and thus $\sum_{k=1}^\infty d_k^+<\infty.$ On the other hand, for every $k\ge 1$
$$
\sum_{i=1}^k d_i=\log\tau_k-\log\tau_0
\ge \log\tau_{\min}-\log\tau_0.
$$
Thus the partial sums of $\sum_{k=1}^\infty d_k$ are bounded below.
Writing
$$
d_k=d_k^+-d_k^-,
\qquad
d_k^-:=-\min\{0,d_k\}\ge 0,
$$
we have
$$
\sum_{i=1}^k d_i^-=\sum_{i=1}^k d_i^+ - \sum_{i=1}^k d_i.
$$
Since $\sum_{k=1}^\infty d_k^+<\infty$, the sequence $\left(\sum_{i=1}^k d_i^+\right)$ is bounded above. Further,  $\sum_{i=1}^k d_i$ is bounded below, thus it follows that $\left(\sum_{i=1}^k d_i^-\right)$ is bounded above. Being nondecreasing, $\left(\sum_{i=1}^k d_i^-\right)$ converges, and therefore $\sum_{k=1}^\infty d_k^-<\infty$.
Consequently, $\sum_{k=1}^\infty d_k$ converges, and hence $(\log\tau_k)$ converges.
Therefore, together with Remark \ref{remark_1}$, (\tau_k)$ converges to some $\tau^\star\ge \tau_{\min}>0$.
\end{remark}

\begin{remark}\label{remark:first_case_finite}
One more thing to notice from Algorithm \ref{alg:2} is that the condition
\[
\tau_{k-1}L_k>\frac{r\sqrt{\bar\lambda}}{\sqrt{\beta}}
\]
can hold only for finitely many values of $k$. Suppose, on the contrary, that there exists a subsequence $(k_j)$ such that
\[
\tau_{k_j-1}L_{k_j}>\frac{r\sqrt{\bar\lambda}}{\sqrt{\beta}}
\qquad \forall j\ge 1.
\]
Then, by the first branch of the update rule \eqref{eq:steps_alg_2_tau_k}
\[
\tau_{k_j}=\frac{r_1\sqrt{\bar\lambda}}{\sqrt{\beta}L_{k_j}},~~\text{and hence}\quad
\frac{\tau_{k_j}}{\tau_{k_j-1}}
=\frac{r_1\sqrt{\bar\lambda}}{\sqrt{\beta}L_{k_j}\tau_{k_j-1}}
<\frac{r_1}{r}<1.
\]
Since $(\tau_k)$ converges to some $\tau^\star\ge \tau_{\min}>0$, both $(\tau_{k_j})$ and $(\tau_{k_j-1})$ converge to the same limit $\tau^\star$, and therefore
\[
\lim_{j\to\infty}\frac{\tau_{k_j}}{\tau_{k_j-1}}=1.
\]
Passing to the limit in the above inequality gives $1\le \frac{r_1}{r}<1,$ which is impossible. Hence, the stated condition can occur only finitely many times. Consequently, after some finite index, only the second branch of the update rule~\eqref{eq:steps_alg_2_tau_k} is active. We stress that this is an asymptotic property of the step-size rule. In the numerical experiments, this finite index may be very large, and the residuals may reach a given error bound before this eventual situation becomes visible in the step-size plots.
\end{remark}

\begin{lemma}\label{lemma_2}
    Let $\{(u_k,x_k,w_k,y_k,\tau_k)\}_{k\ge1}$ be generated by Algorithm~\ref{alg:2}. Let $(\bar x,\bar w)$ be any solution of \eqref{eq:model}. Then, there exists a natural number $k_5$ such that, for all $k\geq k_5$, and for any $y\in\mathbb{R}^m$, the following holds.
\begin{align}\label{eq:alg_2_lemma_1_eq_1}
2\tau_k\!\left(\mathbb{L}(x_k,w_k,y)-\Phi^\star\right)
&\le \frac{\psi}{\psi-1}\big(\|u_{k+1}-\bar x\|_{S}^{2}-\|u_{k+2}-\bar x\|_{S}^{2}\big) + \frac{1}{\beta}\big(\|\bar w-w_{k-1}\|_{T}^{2}-\|\bar w-w_k\|_{T}^{2}\big)\nonumber\\&
     \quad  + \frac{1}{\beta}\big(\|y-y_{k-1}\|^{2}
      -\|y-y_k\|^{2}\big) - \frac{1}{\beta}\|w_k-w_{k-1}\|_{T}^{2}\nonumber\\&
     \quad -\frac{1}{\beta}\left(1-\frac{\rho}{2}\right)\|y_k-y_{k-1}\|^{2}
       -\frac{\psi\tau_k}{\tau_{k-1}}\|x_k-u_{k+1}\|_{S}^{2}.
\end{align}
\end{lemma}
\begin{proof}
Doing  similar calculations as in Lemma \ref{lem:energy}, from \eqref{eq:after-fourpoint-y}, we obtain
    \begin{align}\label{eq:alg_2_eq_1}
\tau_k\left(\mathbb{L}(x_k,w_k,y)-\Phi^\star\right)
&\le \tau_k\langle y_{k-1}-y_k,\,A(x_{k+1}-x_k)\rangle\nonumber\\&
   \quad + \frac{1}{2\beta}\Big(\|y-y_{k-1}\|^{2}-\|y-y_k\|^{2}-\|y_k-y_{k-1}\|^{2}\Big) \nonumber\\
&\quad + \frac{\psi\tau_k}{2\tau_{k-1}}\Big(\|x_{k+1}-u_{k+1}\|_S^2-\|x_k-u_{k+1}\|_S^2-\|x_{k+1}-x_k\|_S^2\Big)\nonumber\\
&\quad + \frac{1}{2}\Big(\|\bar x-u_{k+1}\|_S^2-\|x_{k+1}-u_{k+1}\|_S^2-\|\bar x-x_{k+1}\|_S^2\Big)\nonumber\\
&\quad + \frac{1}{2\beta}\Big(\|\bar w-w_{k-1}\|_T^2-\|w_k-w_{k-1}\|_T^2-\|\bar w-w_k\|_T^2\Big).
\end{align}
We next estimate the mixed term involving the operator $A$. By Remark~\ref{remark:first_case_finite}, the first branch in the update rule
\eqref{eq:steps_alg_2_tau_k} can occur only finitely many times. Hence, there exists
$\hat k\in\mathbb N$ such that, for every $k\ge \hat k$,
\[
    \tau_k L_{k+1}
    \le
    \frac{r\sqrt{\bar\lambda}}{\sqrt{\beta}}.
\]
Therefore, using the definition of $L_{k+1}$, the Cauchy--Schwarz inequality, the fact that $r<\frac{\rho}{2}$, and
$\|v\|_S^2\ge \bar\lambda\|v\|^2$, we obtain
\begin{align}\label{Cauchy_Schwarz_step-size_update}
\tau_k\langle A(x_k-x_{k+1}),\,y_k-y_{k-1}\rangle
&\le
\tau_k\|A x_k-Ax_{k+1}\|\,\|y_k-y_{k-1}\| \nonumber\\
&\le
\frac{\rho\sqrt{\bar\lambda}}{2\sqrt{\beta}}
\|x_k-x_{k+1}\|\,\|y_k-y_{k-1}\| \nonumber\\
&\le
\frac{\rho}{2\sqrt{\beta}}
\|x_k-x_{k+1}\|_S\,\|y_k-y_{k-1}\| \nonumber\\
&\le
\frac{\rho}{4}\|x_k-x_{k+1}\|_S^2
+\frac{\rho}{4\beta}\|y_k-y_{k-1}\|^2 ~~\forall k\ge \hat k .
\end{align}
Thus, combining \eqref{Cauchy_Schwarz_step-size_update} with
\eqref{eq:alg_2_eq_1} and then using \eqref{eq:gr-id-k}, we obtain, for all
$k\ge \hat k$
\begin{align}\label{eq:alg_2_eq_2}
2\tau_k\!\left(\mathbb{L}(x_k,w_k,y)-\Phi^\star\right)
&\le \frac{\psi}{\psi-1}\big(\|u_{k+1}-\bar x\|_{S}^{2}-\|u_{k+2}-\bar x\|_{S}^{2}\big) + \frac{1}{\beta}\big(\|\bar w-w_{k-1}\|_{T}^{2}-\|\bar w-w_k\|_{T}^{2}\big)\nonumber\\&
      \quad  + \frac{1}{\beta}\big(\|y-y_{k-1}\|^{2}
      -\|y-y_k\|^{2}\big) - \left(\frac{\psi\tau_k}{\tau_{k-1}}-\frac{\rho}{2}\right)\|x_{k+1}-x_k\|_{S}^{2}\nonumber\\&
      \quad + \left(\frac{\psi\tau_k}{\tau_{k-1}}-1-\frac{1}{\psi}\right)\|x_{k+1}-u_{k+1}\|_{S}^{2} \nonumber - \frac{1}{\beta}\|w_k-w_{k-1}\|_{T}^{2}\\&
      \quad - \frac{1}{\beta}\left(1-\frac{\rho}{2}\right)\|y_k-y_{k-1}\|^{2}- \frac{\psi\tau_k}{\tau_{k-1}}\|x_k-u_{k+1}\|_{S}^{2}.
\end{align}
From Remark \ref{remark:tau_convergence} and the definition of $\psi$, notice that
\begin{equation}\label{psithetanminus1}
    \lim_{k\to\infty}\left(\frac{\psi\tau_k}{\tau_{k-1}} - 1 - \frac{1}{\psi}\right) = \psi -1 -\frac{1}{\psi}<0.
\end{equation}
Thus, there exists a natural number $k_3$ such that
\begin{equation}\label{psitheta_k-1-1bypsi}
    \frac{\psi\tau_k}{\tau_{k-1}} - 1 - \frac{1}{\psi}<0~~\forall k\geq k_3.
\end{equation}
Furthermore, there exists another natural number $k_4$ such that
\begin{equation}\label{psitheta_k-rhoby2}
    \frac{\psi\tau_k}{\tau_{k-1}}-\frac{\rho}{2}>0~~\forall k\geq k_4.
\end{equation}
Let $k_5= \max\{\hat k, k_3, k_4\}$. Then for all $k\geq k_5$, combining \eqref{eq:alg_2_eq_2}, \eqref{psithetanminus1}, \eqref{psitheta_k-1-1bypsi} and~\eqref{psitheta_k-rhoby2}, we obtain Lemma~\ref{lemma_2}.
\end{proof}

\begin{theorem}
    Let $(\bar x,\bar w,\bar y)$ be a saddle point of $\mathbb{L}$. Under Assumption~\ref{assump:basic},
let the sequence $\{(x_k,w_k,y_k)\}_{k\ge1}$ be generated by Algorithm~\ref{alg:2}. 
\begin{enumerate}[label=(\alph*)]
    \item Then the sequence $\{(x_k,w_k,y_k)\}_{k\ge1}$ converges to a saddle point of $\mathbb{L}$.\label{first_statement}  
    \item The average ergodic sequences $\{\tilde x_N\}$ and $\{\tilde w_N\}$ satisfy
    \[
    \big|\Phi(\tilde x_N,\tilde w_N)-\Phi^\star\big|= \mathcal{O}(1/N')
    \qquad \text{and} \qquad
    \|A\tilde x_N+B\tilde w_N-b\|=\mathcal{O}(1/N'),
    \]
    where\[
\tilde x_N:=\frac{1}{N'}\sum_{k=k_5}^N x_k,\qquad
\tilde w_N:=\frac{1}{N'}\sum_{k=k_5}^N w_k,\qquad N'=N-k_5+1,
\]
and $k_5$ is given in Lemma \ref{lemma_2}.
    \label{second_statement}
\end{enumerate}
\end{theorem}
\begin{proof} of \ref{first_statement}.
    Since $(\bar x, \bar w, \bar y)$ is a saddle point of $\mathbb{L}$, we have $\mathbb{L}(x_k,w_k, \bar y) - \Phi^\star\geq0~\forall k$. Now, by applying Lemma~\ref{lemma_2} with $y=\bar y$, from \eqref{eq:alg_2_lemma_1_eq_1}, we obtain
\begin{equation*}
     a_{k+1}(\bar y)\leq a_k(\bar y) - b_k~~~\text{for all}~~k\geq k_5,
\end{equation*}
where
\begin{align*}
    a_k(\bar y)&:=\frac{\psi}{\psi-1}\|u_{k+1}-\bar x\|_{S}^{2}
     +\frac{1}{\beta}\|\bar w-w_{k-1}\|_{T}^{2}
     +\frac{1}{\beta}\|\bar y-y_{k-1}\|^{2}\nonumber,\\
b_k&:=\frac{1}{\beta}\|w_k-w_{k-1}\|_T^{2}
     +\frac{\psi\tau_k}{\tau_{k-1}}\|x_k-u_{k+1}\|_{S}^{2}+\frac{1}{\beta}\left(1-\frac{\rho}{2}\right)\|y_k-y_{k-1}\|^{2}.
\end{align*}
Since $S, T\succ 0$ and $1-\frac{\rho}{2}>0$, both $\{a_k(\bar y)\}$ and $\{b_k\}$ are non-negative sequences. Thus, by running analogous arguments as in Theorem \ref{thm:conv} and keeping in mind the facts that $(\tau_k)$ is bounded below by $\tau_{\min}>0$ and $(\tau_k)$ is convergent, we obtain that $\{(x_k, w_k,y_k)\}$ converges to a saddle point of $\mathbb{L}$.
\end{proof}
\begin{proof}of \ref{second_statement}.
It follows from Remark~\ref{remark_1} and \eqref{eq:alg_2_lemma_1_eq_1} that
\[
\tau_k \ge\tau_{\min}:= \min\left\{\tau_0,\, \frac{\sqrt{\bar\lambda} r_1}{\sqrt{\beta}\|A\|}\right\}
\quad \forall k.
\]
Therefore, by proceeding analogously to the proof of Theorem~\ref{thm:ergodic-image}, and using \eqref{eq:alg_2_lemma_1_eq_1}, we obtain the claimed sublinear convergence rates for the objective gap and the feasibility residual. This completes the proof.
\end{proof}

\begin{remark}
Define
\[
r_k^p:=Ax_k+Bw_k-b,\quad
\eta_k^x:=-\frac{1}{\tau_{k-1}}S(x_k-u_k)+A^\top(y_k-y_{k-1}),
\]
and
\[
\eta_k^w:=-\frac{1}{\sigma_k}T(w_k-w_{k-1}).
\]
Then $\eta_k^x\in \partial g(x_k)+A^\top y_k$ and
$\eta_k^w\in \partial f(w_k)+B^\top y_k$. Therefore, using Lemma~\ref{lemma_2} and the bound $\sigma_k=\beta\tau_k\ge \beta\tau_{\min}>0$, one may proceed along the same lines as in the proof of Theorem \ref{thm:pointwise_alg1_clean} to obtain the pointwise $O(1/\sqrt{N})$ convergence rate for Algorithm~\ref{alg:2}. We omit the details.
\end{remark}

\begin{remark}
It is worth noting that, in the case where $B=-I$, $b=0$, $S=I$, and $T=0$, Algorithm~\ref{alg:2} is different from the algorithm studied in \cite{soe2026golden}. In particular, Algorithm~\ref{alg:2} may be regarded as a new method for the problem considered in \cite{soe2026golden}, with the additional feature that it allows for a non-decreasing step-size rule.
\end{remark}

\section{Numerical results}\label{numerical_section}
We now present numerical experiments to assess the practical performance of the proposed strategies. We compare four algorithms: Algorithm~\ref{alg:extened psi}, Algorithm~\ref{alg:2}, GrpADMM \cite{Chen2023GRPADMM}, and PADMM \cite{Eckstein1994}. We use Algorithm \ref{alg:extened psi} in numerical experiments as the parameters are more relaxed in comparison to Algorithm~\ref{alg:1}, and may lead to better performance as we will see later.  Before proceeding, note that both Algorithms~\ref{alg:extened psi} and Algorithm~\ref{alg:2} require the $w$–block weight matrix $T$ to be positive definite ($T\succ0$) in order to guarantee iterate convergence of the sequence $\{w_k\}$. However, when the goal is only to plot objective gap and function-value residuals, it is sufficient to take $T$ to be a positive semidefinite ($T\succeq0$) matrix. Unless stated otherwise, in all the experiments for Algorithm \ref{alg:2}, we set $\xi_k = 1-\rho + \frac{1}{k^{1.001}}$, $\psi=1.50$, $\rho=\frac{1}{\psi}+\frac{1}{\psi^2}$, $r=\frac{0.99\rho}{2}$ and $r_1 = 0.99r$. Given an iterate $(x_k,w_k,y_k)$, we report Relative objective gap and the Feasibility gap
defined as
\[
\mathrm{Rel\_gap}_{k}
:=
\frac{|\Phi(x_k,w_k)-\Phi^{\star}|}{|\Phi^{\star}|},
\qquad
\mathrm{Fes\_gap}_{k}
:=
\|Ax_k+Bw_k-b\|_{2},
\]
where $\Phi^{\star}:=\Phi(x^{\star},w^{\star})$ is computed as the best objective value obtained across all methods after sufficiently long runs. Furthermore, the KKT system associated with \eqref{eq:model} is
\[
0 \in \partial g(x) + A^\top y, 
\qquad
0 \in \partial f(w) + B^\top y, 
\qquad
Ax + Bw - b = 0.
\]
The first two conditions are equivalent to $x = \operatorname{prox}_g(x - A^\top y),~~
w = \operatorname{prox}_f(w - B^\top y),$ respectively. Therefore, for each iterate $(x_k,w_k,y_k)$, we can measure the combined KKT residual
\begin{equation}\label{eq:kkt-residual}
\rm{KKT\_res}_k
:=
\sqrt{(r_k^x)^2 + (r_k^w)^2 + (r_k^p)^2},
\end{equation}
where
\[
r_k^x := \bigl\|x_k - \operatorname{prox}_g(x_k - A^\top y_k)\bigr\|,
\quad
r_k^w := \bigl\|w_k - \operatorname{prox}_f(w_k - B^\top y_k)\bigr\|,
\quad
r_k^p := \|Ax_k + Bw_k - b\|.
\]

All methods were implemented in Python~3.11 and executed in a Google Colab environment with 12.7~GB RAM.
\subsection{Sparse signal recovery via LASSO}
To assess the practical performance of the proposed algorithms on a sparse signal recovery task, we consider the LASSO model \cite{tibshirani1996regression}
\begin{equation}\label{eq:lasso_signal_recovery}
    \min_{x\in\mathbb{R}^{n}} \lambda \|x\|_{1}
        + \frac12 \|Ax-b\|^{2},
\end{equation}
where $A\in\mathbb{R}^{m\times n}$ is the sensing matrix, $b\in\mathbb{R}^{m}$ is the observation vector, and $\lambda>0$ is the regularization parameter. 

\begin{figure}[htbp]
    \centering
    \subfloat[Algorithm \ref{alg:2}]{%
        \includegraphics[width=0.47\textwidth]{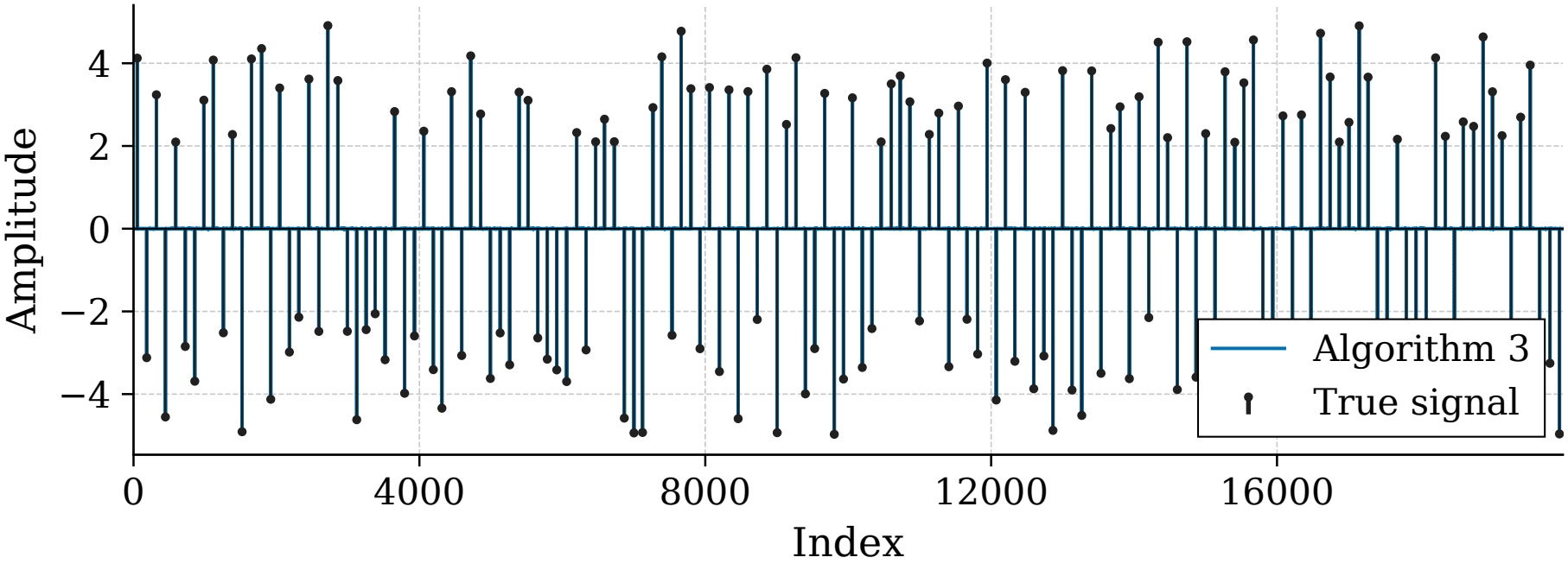}    
    }
    \hfill
    \subfloat[Algorithm \ref{alg:extened psi}]{%
        \includegraphics[width=0.47\textwidth]{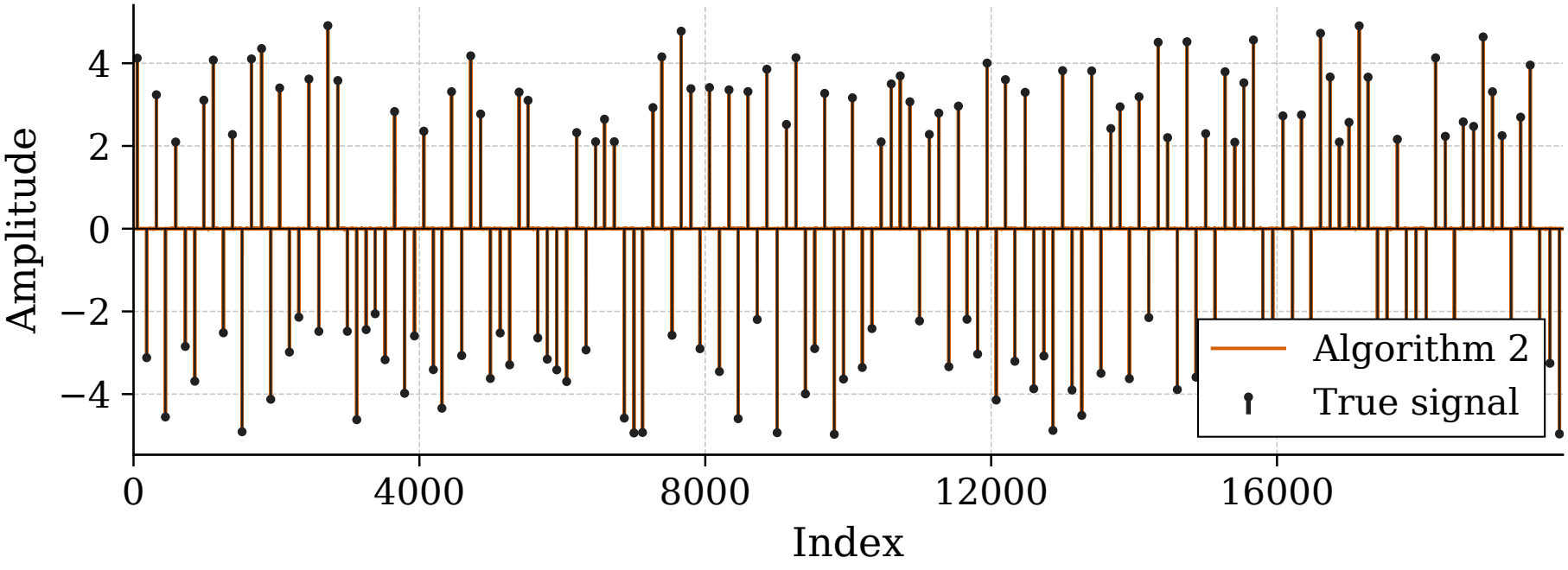}
    }

    \vspace{0.4cm}

    \subfloat[GrpADMM]{%
        \includegraphics[width=0.47\textwidth]{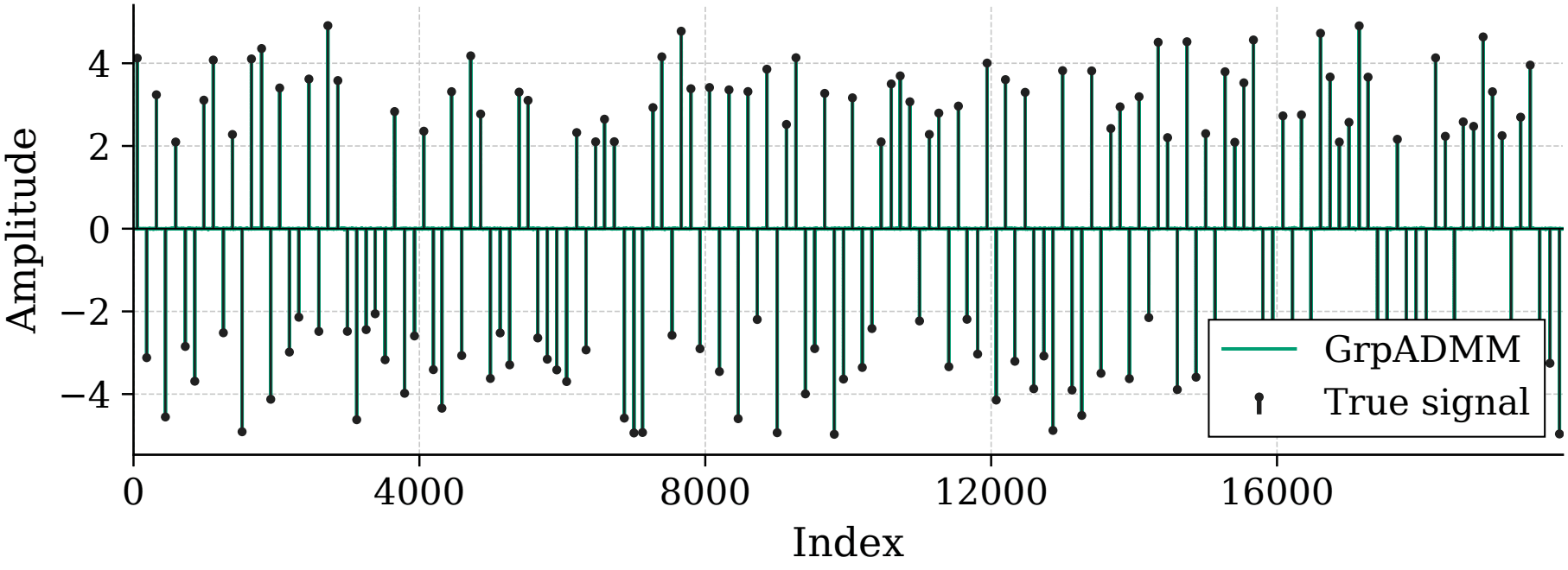}
    }
    \hfill
    \subfloat[PADMM]{%
        \includegraphics[width=0.47\textwidth]{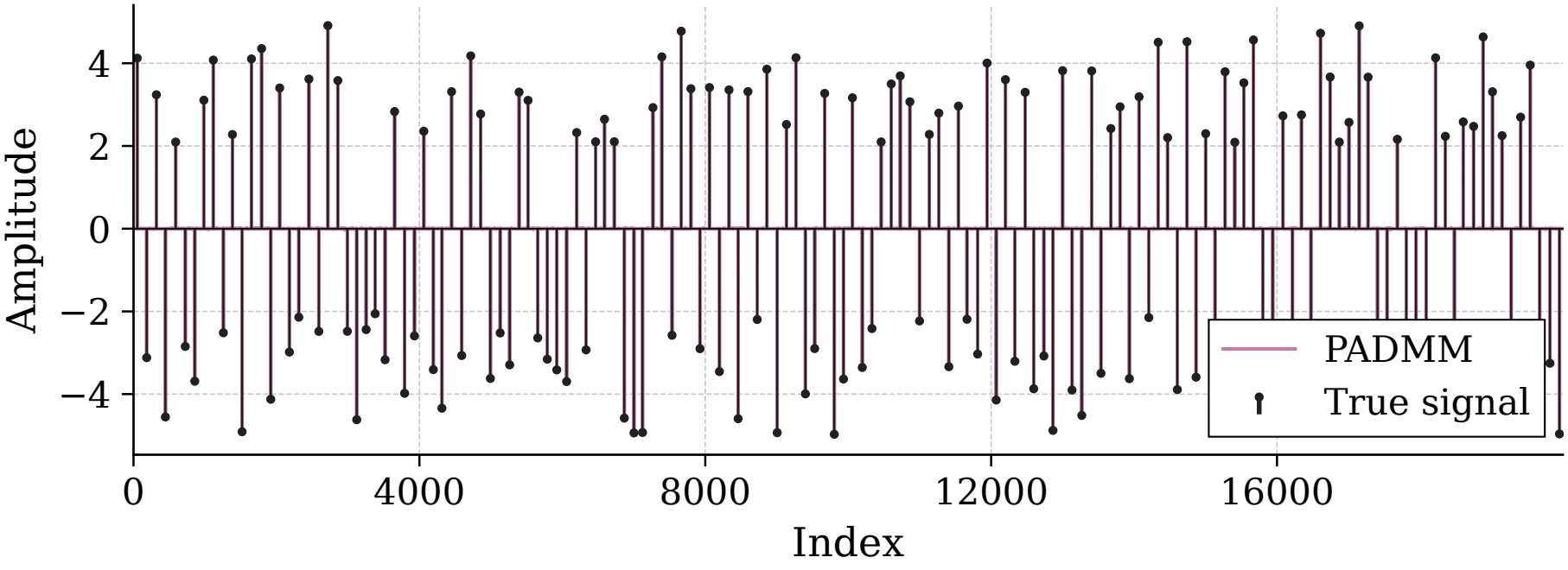}
    }

\caption{Sparse-signal recovery for the LASSO problem with $(n,m,s)=(20000,10000,150)$. The true signal is compared with the reconstructions obtained by (a) Algorithm~\ref{alg:2}, (b) Algorithm~\ref{alg:extened psi}, (c) \eqref{eq:GrpADMM}, and (d) \eqref{eq:PADMM}.}
\label{fig:lasso_signal_support_recovery_n_20000}
\end{figure}

In this experiment, our goal is to recover the sparse signal $x$. In order to place \eqref{eq:lasso_signal_recovery} into the linearly constrained separable framework studied in this paper, we introduce an auxiliary variable $w\in\mathbb{R}^{m}$ and rewrite \eqref{eq:lasso_signal_recovery} as
\begin{equation}\label{eq:split_lasso_signal_recovery}
    \min_{x\in\mathbb{R}^{n},\,w\in\mathbb{R}^{m}}
    \ \Phi(x,w)
    := \lambda \|x\|_{1} + \frac12 \|w\|^{2}
    \quad\text{subject to}\quad
    Ax-w=b.
\end{equation}

\begin{figure}[htbp]
\centering

\subfloat[Relative objective gap]{
  \includegraphics[width=0.30\linewidth]{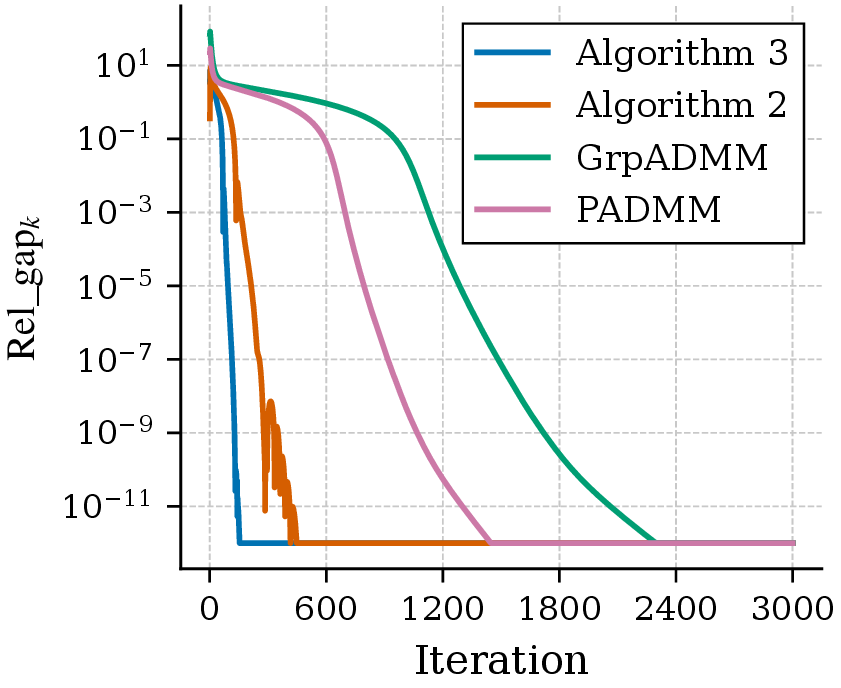}
}\hfill
\subfloat[Feasibility residual]{
  \includegraphics[width=0.30\linewidth]{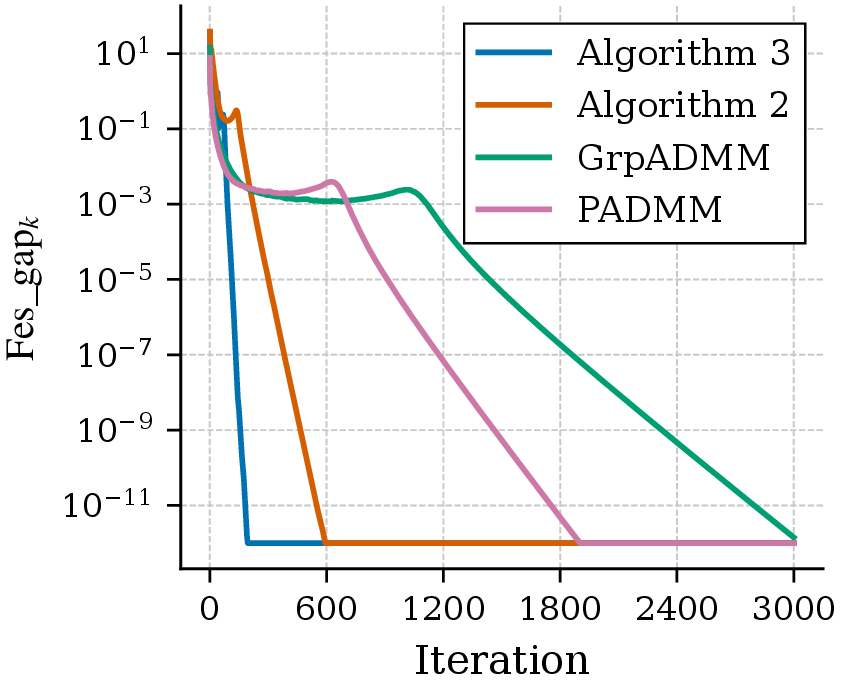}
}

\vspace{0.6em}

\subfloat[KKT residual \eqref{combined_KKT_residual}]{
  \includegraphics[width=0.30\linewidth]{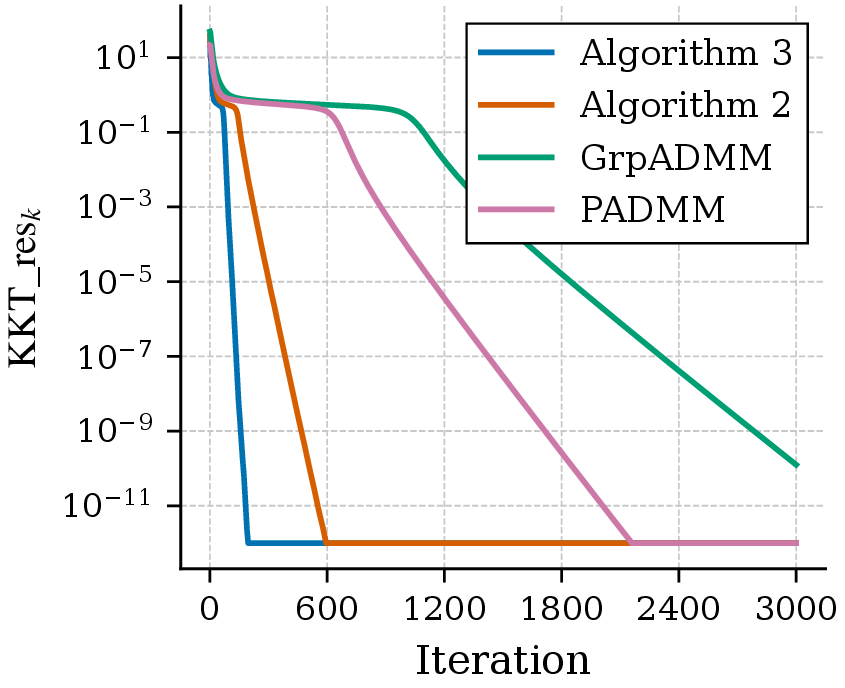}
}\hfill
\subfloat[Primal step-size ($\tau_k$)]{
  \includegraphics[width=0.30\linewidth]{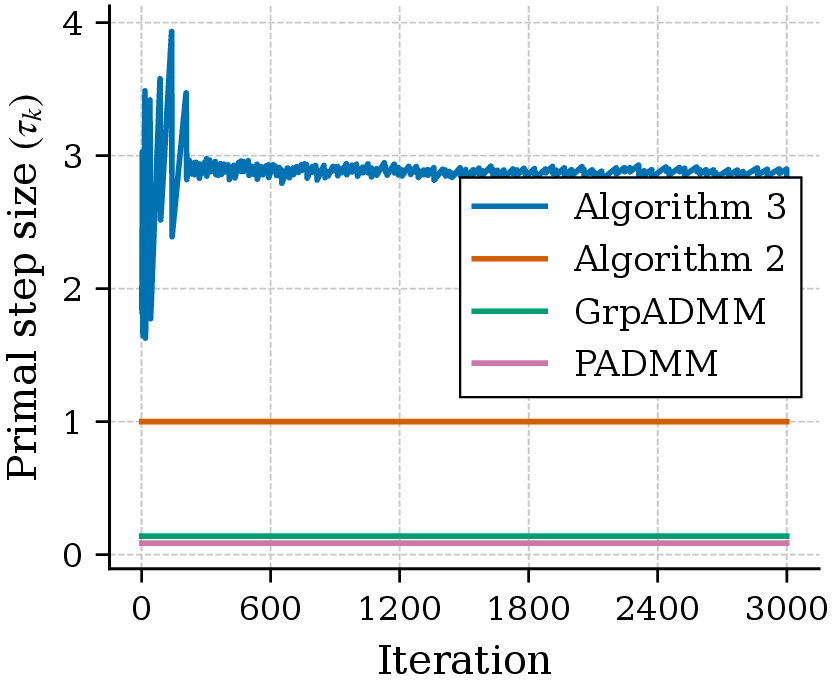}
}

\caption{Convergence plots of the LASSO problem for four different algorithms on the problem size $(n,m,s)=(20000, 10000, 150)$.}
\label{fig:feas_gap_both_sizes_lasso_n_20000}
\end{figure}
Thus, comparing \eqref{eq:split_lasso_signal_recovery} with \eqref{eq:model}, we have $g(x)=\lambda\|x\|_{1},~
f(w)=\frac12\|w\|^{2},$ and $
B=-I.$ In this experiment, we set $(n,m)=(20000,10000),$ and $(n,m)=(10000,3000),$ and $\lambda=10^{-2}.$ The entries of the matrix $A\in\mathbb{R}^{m\times n}$ are sampled independently from the Gaussian distribution $\mathcal{N}(0,1/m)$. The ground-truth signal $x_{\mathrm{true}}\in\mathbb{R}^{n}$ is chosen to be sparse with exactly $s$ nonzero entries, and its nonzero elements are distributed uniformly over $\{1,\dots,n\}$, with a small random perturbation to avoid an artificially regular pattern. The measurement vector is then defined by
\[
b = A x_{\mathrm{true}} + \eta,
\]
where $\eta\sim\mathcal{N}(0,\sigma_{\mathrm{noise}}^{2}I_{m})$, and $\sigma_{\mathrm{noise}}=10^{-2}$. Hence, the data are mildly contaminated by Gaussian noise, which makes the recovery task nontrivial. In addition to the relative objective residual and feasibility gap, we monitor the combined KKT residual \eqref{eq:kkt-residual}, which in this case is
\begin{equation}\label{combined_KKT_residual}
    \mathrm{KKT}\_{\rm res_k}
:=
\Bigl(
    \|Ax_k-w_k-b\|^{2}
    +\frac{1}{4}\|w_k-y_k\|^{2}+ \bigl\|
        x_k-\operatorname{prox}_{\lambda\|\cdot\|_{1}}(x_k-A^{\top}y_k)
    \bigr\|^{2}
\Bigr)^{1/2}.
\end{equation}

In Algorithm~\ref{alg:extened psi}, Algorithm~\ref{alg:2} and \ref{eq:GrpADMM}, we choose $S=I$ and $T=0$ for both practical and structural reasons. By taking $S=I$, the proximal regularisation, in the $x$-subproblem becomes a standard Euclidean quadratic term, which makes the $x$-update explicit and inexpensive. Indeed, the $x$-subproblem of Algorithms~\ref{alg:extened psi} and \ref{alg:2}  admits a closed-form soft-thresholding formula. On the other hand, choosing $T=0$ avoids adding an unnecessary proximal correction in the $w$-block, since the term $\frac12\|w\|_2^2$ together with the augmented Lagrangian contribution already makes the $w$-subproblem strongly convex and explicitly solvable. In fact, since $B=-I$, the $w$-update becomes
\[
w_k
=
\arg\min_w
\left\{
\frac12\|w\|^2
-\langle y_{k-1},w\rangle
+\frac{\sigma_k}{2}\|Ax_k-w-b\|^2
\right\},
\]
which yields the closed-form expression
\[
w_k=\frac{y_{k-1}+\sigma_k(Ax_k-b)}{1+\sigma_k}.
\]
This keeps each iteration computationally light and makes the comparison focus on the step-size strategies rather than on the cost of solving inner subproblems. We compare four methods, and the following parameters are selected for each one.

\begin{itemize}
    \item \textbf{Algorithm~\ref{alg:2}:} $ S=I,~T=0,~\beta=0.1,~\tau_{0}=1$, and other parameters are selected as mentioned above.

    \item \textbf{Algorithm~\ref{alg:extened psi}:} $S=I,~T=0,~ \psi=1.7,~ \beta=0.1,~\mu=0.79,~\tau_{0}=1$.

    \item \textbf{\ref{eq:GrpADMM}:} $S=I,~ T=0,~ \psi=\varphi,~ \sigma_{\rm grp}=2,$ and 
    $\tau_{\rm grp} = \frac{\psi}{\sigma_{\rm grp}\|A\|^{2}}.$

    \item \textbf{\ref{eq:PADMM}:} $S=\frac{1}{\tau_{\rm pad}}I-\sigma_{\rm pad}A^{\top}A, ~T=0,~\sigma_{\rm pad}=2
    $, and $\tau_{\rm pad}=\frac{0.99}{\sigma_{\rm pad}\|A\|^{2}}.$

\end{itemize}

This experiment is designed to examine two complementary aspects of the methods. The first is optimisation performance, which is evaluated through the relative objective residual, the feasibility residual, and the combined KKT residual. The second is recovery quality, which is assessed by comparing the reconstructed signal with the true sparse spike signal. We consider two problem instances corresponding to different dimensions of the underlying sparse signal. Figures~\ref{fig:feas_gap_both_sizes_lasso_n_20000} and~\ref{fig:feas_gap_both_sizes_lasso_n_10000} show that Algorithm~\ref{alg:2} consistently outperforms the other three methods with respect to the residual measures, while Algorithm~\ref{alg:extened psi} provides the second-best performance. The full signal recovery plots in Figures~\ref{fig:lasso_signal_support_recovery_n_20000} and~\ref{fig:lasso_signal_support_recovery_n_10000} further illustrate how accurately each method recovers the spike amplitudes, and these observations are consistent with the residual curves. Overall, this example provides a representative large-scale benchmark for assessing the effectiveness of the proposed step-size strategies.

\begin{figure}[htbp]
\centering

\subfloat[Relative objective gap]{
  \includegraphics[width=0.30\linewidth]{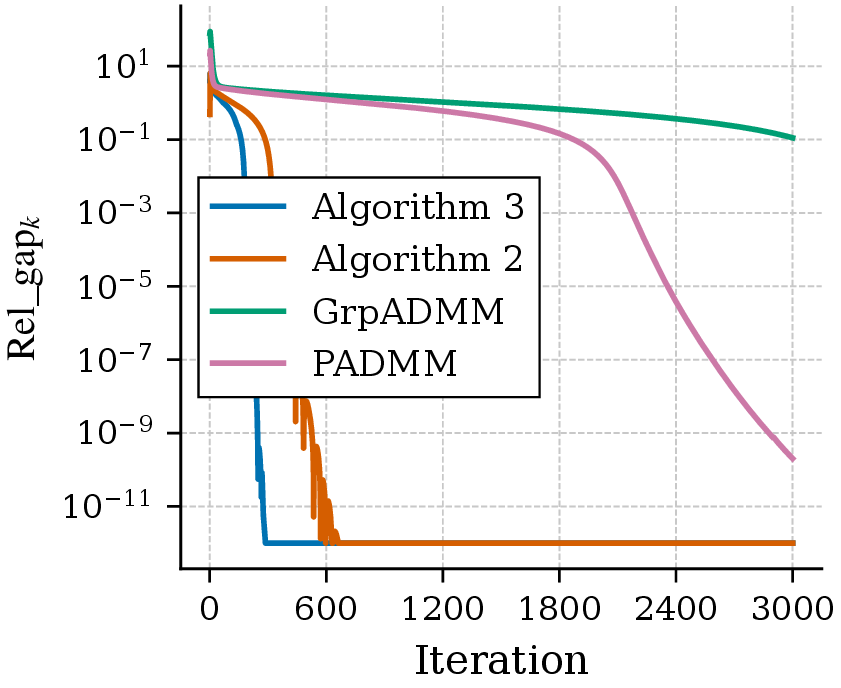}
}\hfill
\subfloat[Feasibility residual]{
  \includegraphics[width=0.30\linewidth]{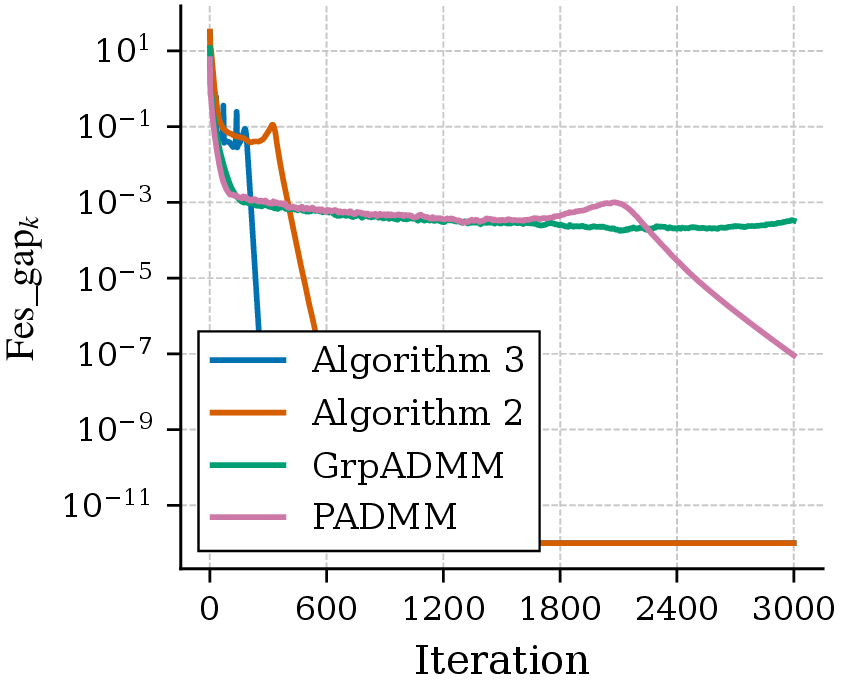}
}

\vspace{0.6em}

\subfloat[KKT residual \eqref{combined_KKT_residual}]{
  \includegraphics[width=0.30\linewidth]{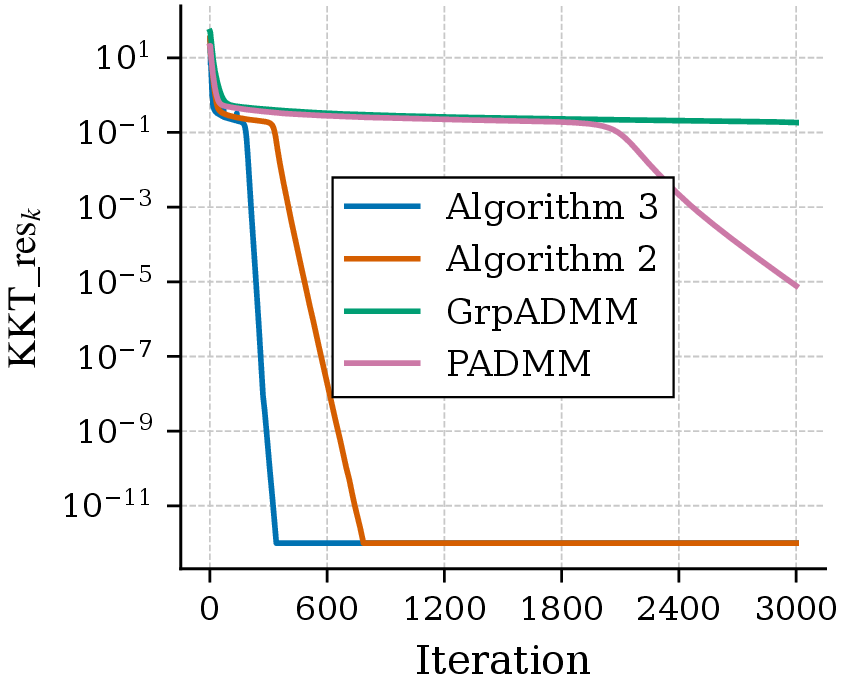}
}\hfill
\subfloat[Primal step-size ($\tau_k$)]{
  \includegraphics[width=0.30\linewidth]{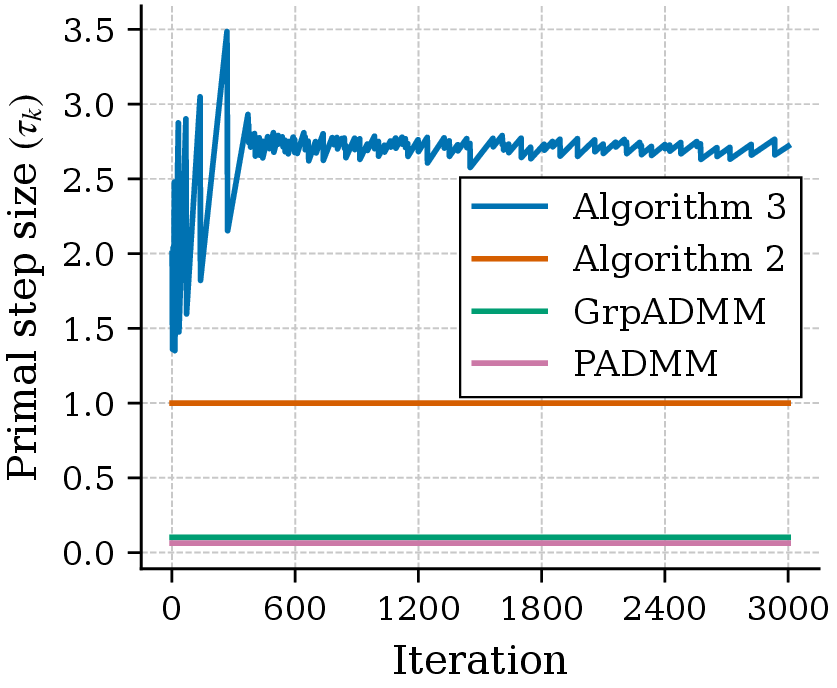}
}

\caption{Convergence plots of the LASSO problem for four different algorithms on the problem size $(n,m,s)=(10000, 3000, 100)$.}
\label{fig:feas_gap_both_sizes_lasso_n_10000}
\end{figure}

\begin{figure}[htbp]
    \centering
    \subfloat[Algorithm 2]{%
        \includegraphics[width=0.48\textwidth]{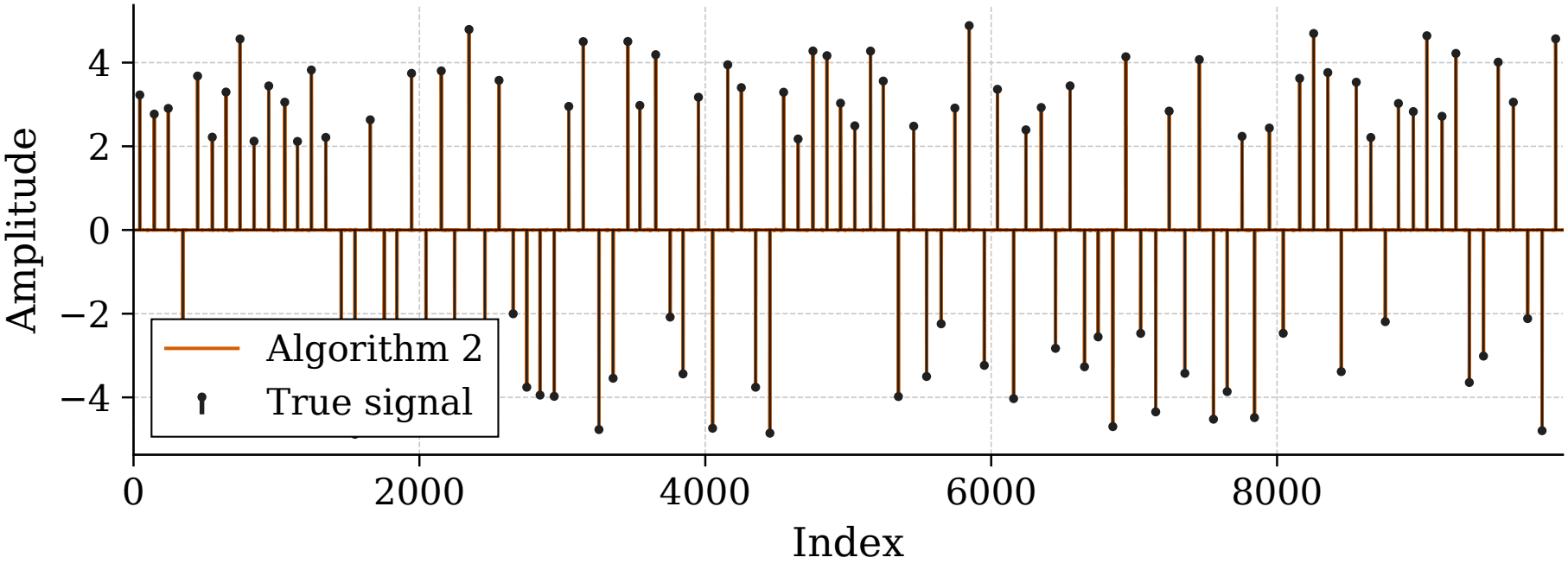}
        
    }
    \hfill
    \subfloat[Algorithm 3]{%
        \includegraphics[width=0.48\textwidth]{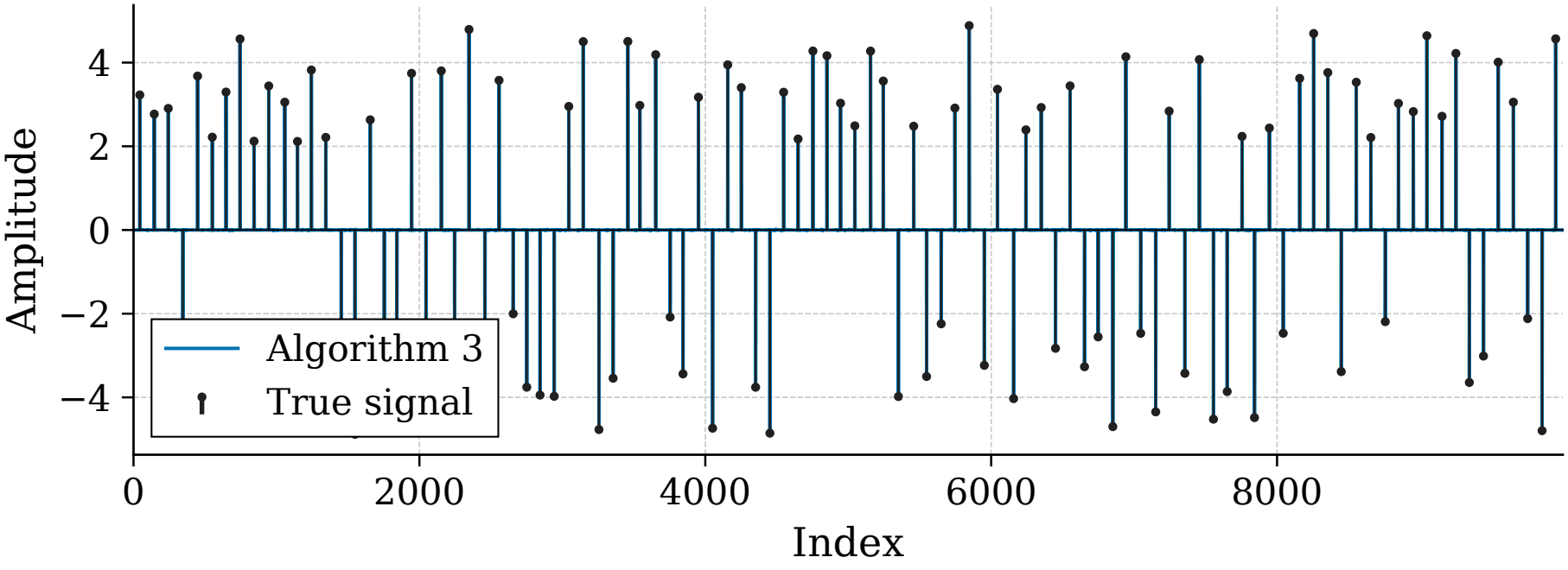}

    }

    \vspace{0.4cm}

    \subfloat[GrpADMM]{%
        \includegraphics[width=0.48\textwidth]{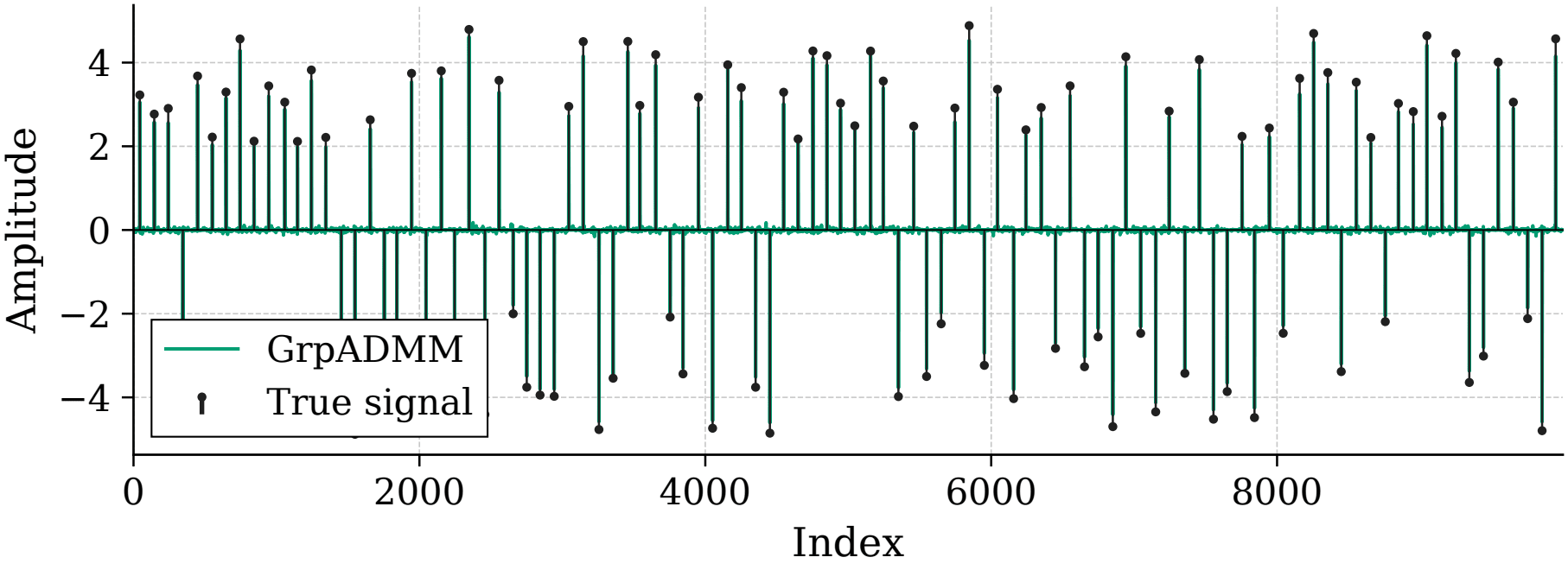}
       
    }
    \hfill
    \subfloat[PADMM]{%
        \includegraphics[width=0.48\textwidth]{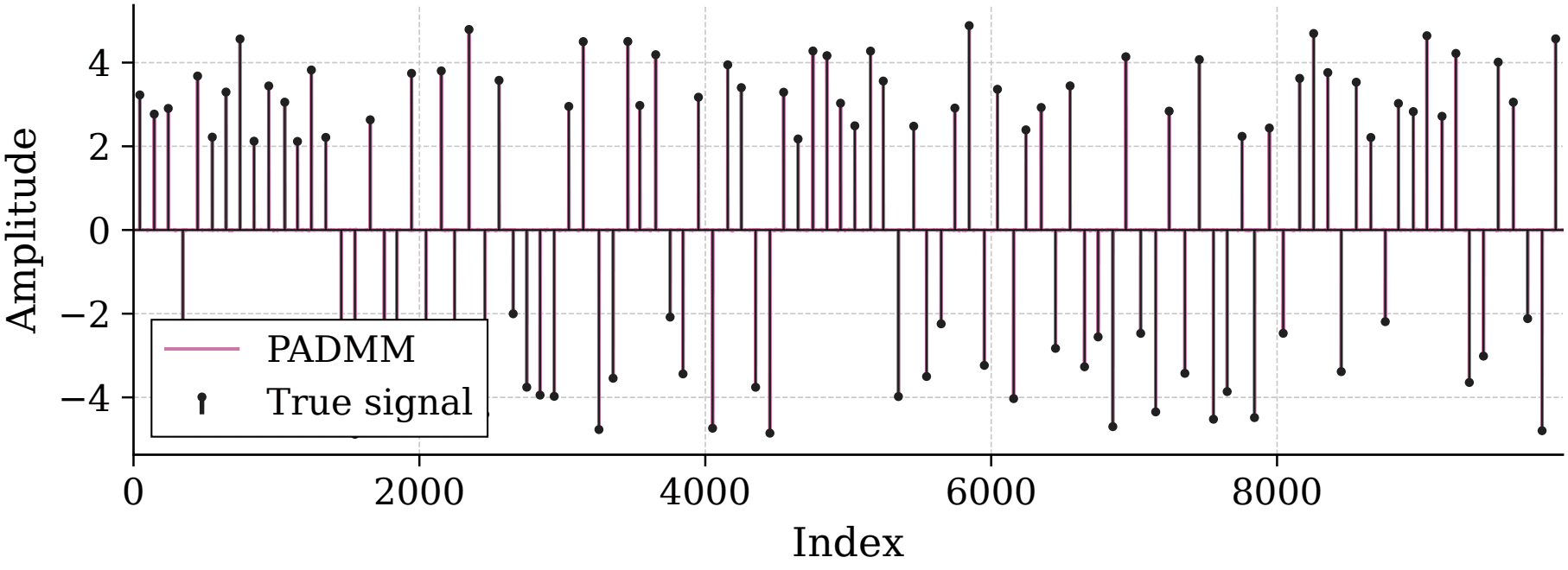}
        
    }

\caption{Recovery of a sparse signal for the LASSO problem with $(n,m,s)=(10000,3000,100)$. The ground-truth signal is compared with the reconstructions generated by (a) Algorithm~\ref{alg:extened psi}, (b) Algorithm~\ref{alg:2}, (c) \eqref{eq:GrpADMM}, and (d) \eqref{eq:PADMM}.}
\label{fig:lasso_signal_support_recovery_n_10000}
\end{figure}

\begin{figure}[t]
\centering

\subfloat[Relative objective gap]{
  \includegraphics[width=0.30\linewidth]{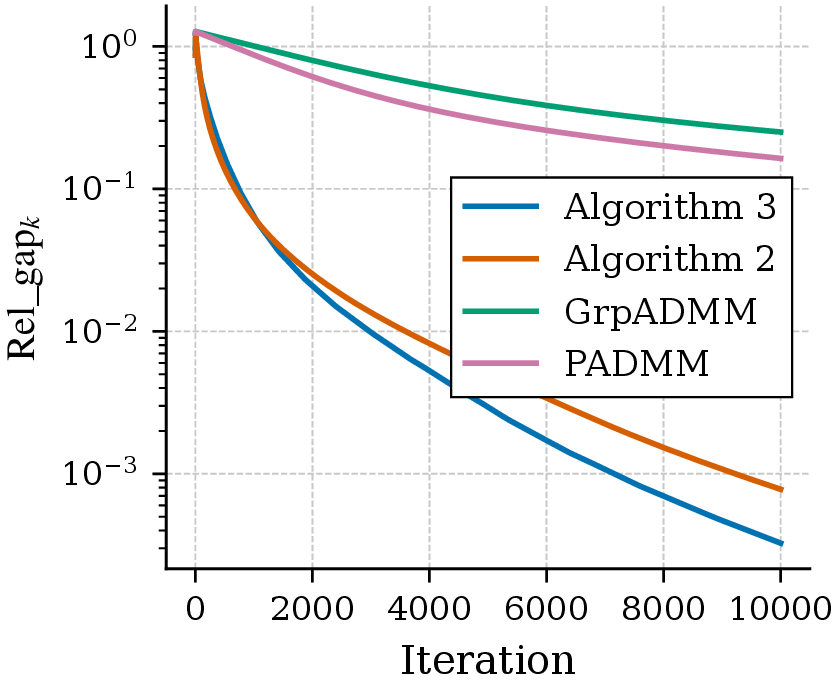}
}\hfill
\subfloat[Feasibility residual]{
  \includegraphics[width=0.30\linewidth]{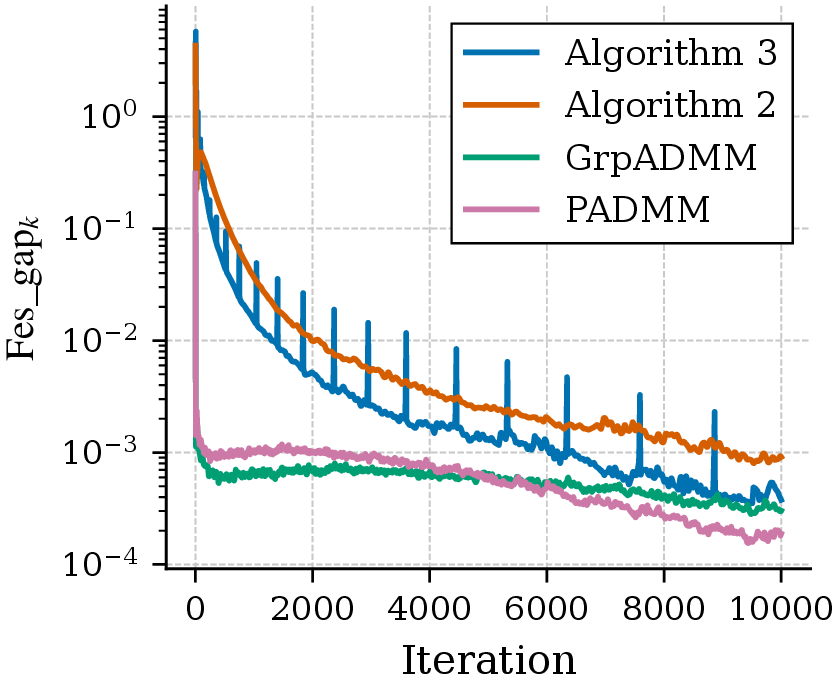}
}

\vspace{0.6em}

\subfloat[KKT residual \eqref{combined_KKT_residual}]{
  \includegraphics[width=0.30\linewidth]{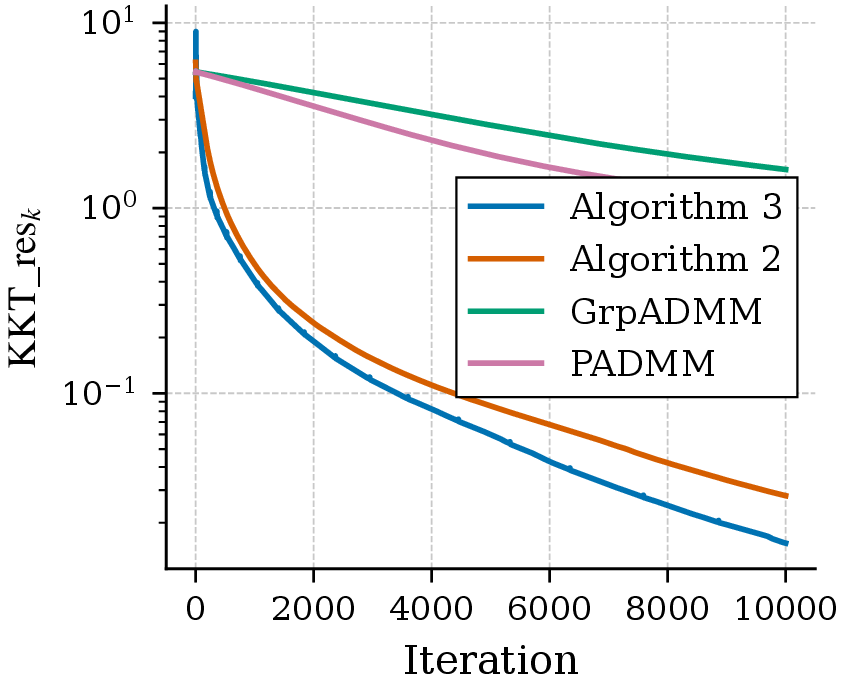}
}\hfill
\subfloat[Primal step-size ($\tau_k$)]{
  \includegraphics[width=0.30\linewidth]{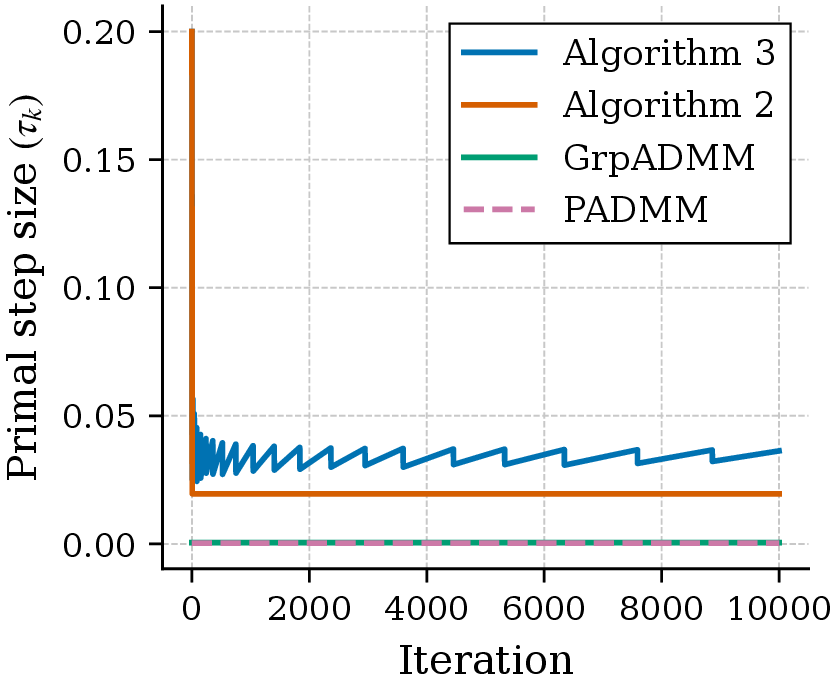}
}

\caption{Convergence plots of the image deblurring problem for four different algorithms.}
\label{fig:image_deblur}
\end{figure}

\subsection{TV-based image deblurring}\label{subsec:tv_deblur_largeA}

Our next experiment is to test on the TV-regularized image deblurring problem \cite{Goldstein}
\begin{equation}\label{eq:img-prim}
    \min_{x\in\mathbb{R}^{N\times N}}
    \;
    \lambda \|\nabla x\|_{2,1}
    + \frac{\mu}{2}\|Hx-c\|_2^2
    + \iota_{[0,1]^{N\times N}}(x),
\end{equation}
where $x$ is the unknown image, $c$ is the blurred and noisy observation, $H$ is a linear blur operator, $\nabla$ is the discrete gradient with periodic boundary conditions, and $\iota_{[0,1]^{N\times N}}$ is the indicator function of the box constraint $[0,1]^{N\times N}$. The isotropic TV seminorm is given by
\[
\|\nabla x\|_{2,1}
=
\sum_{i,j}\sqrt{(\nabla_1x)_{ij}^2+(\nabla_2x)_{ij}^2}.
\]
\begin{figure}[htbp]
\centering

\subfloat[True image]{
  \includegraphics[width=0.25\linewidth]{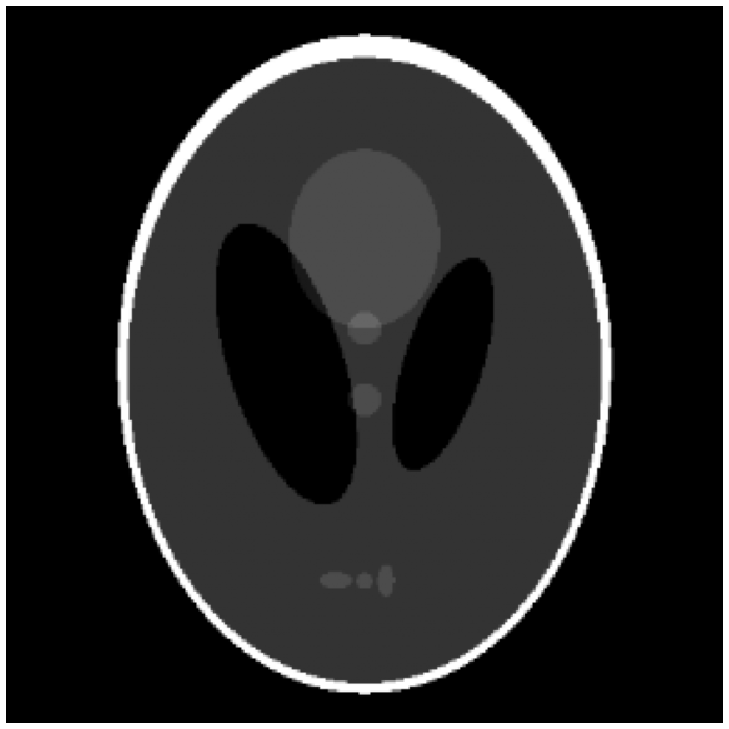}
}\hfill
\subfloat[Blurry and noisy]{
  \includegraphics[width=0.25\linewidth]{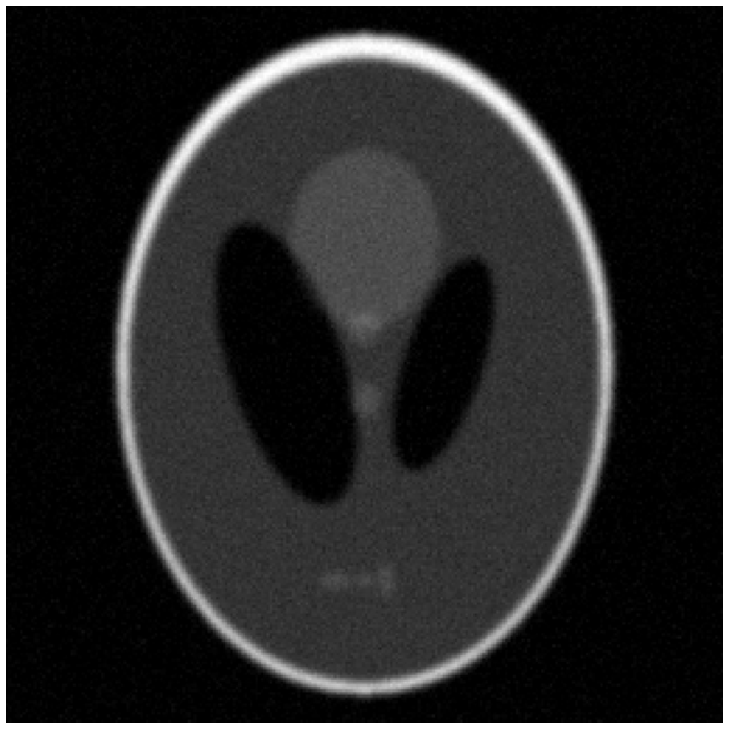}
}\hfill
\subfloat[PADMM (PSNR $=24.97$)]{
  \includegraphics[width=0.25\linewidth]{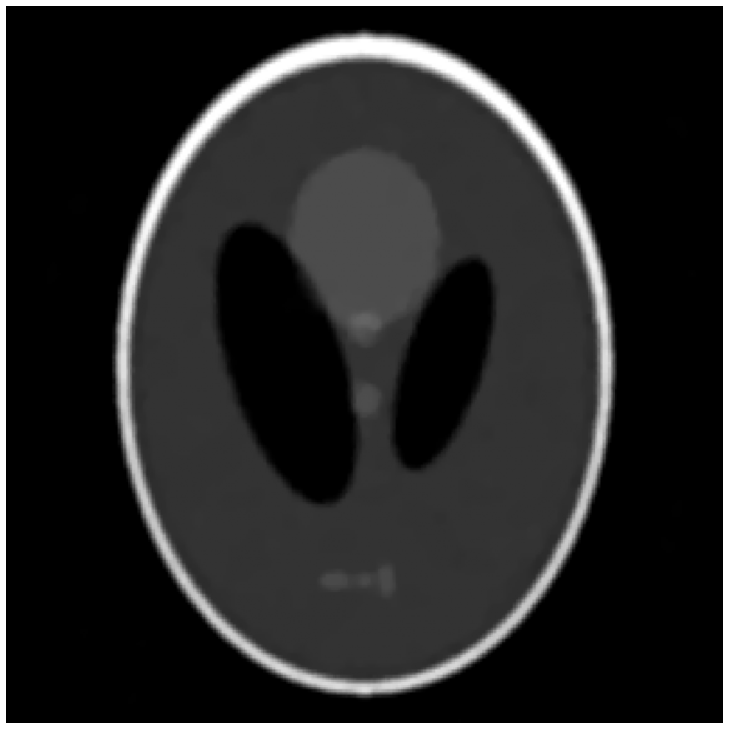}
}

\vspace{0.6em}

\subfloat[Algorithm~\ref{alg:2} (PSNR $=29.65$)]{
  \includegraphics[width=0.25\linewidth]{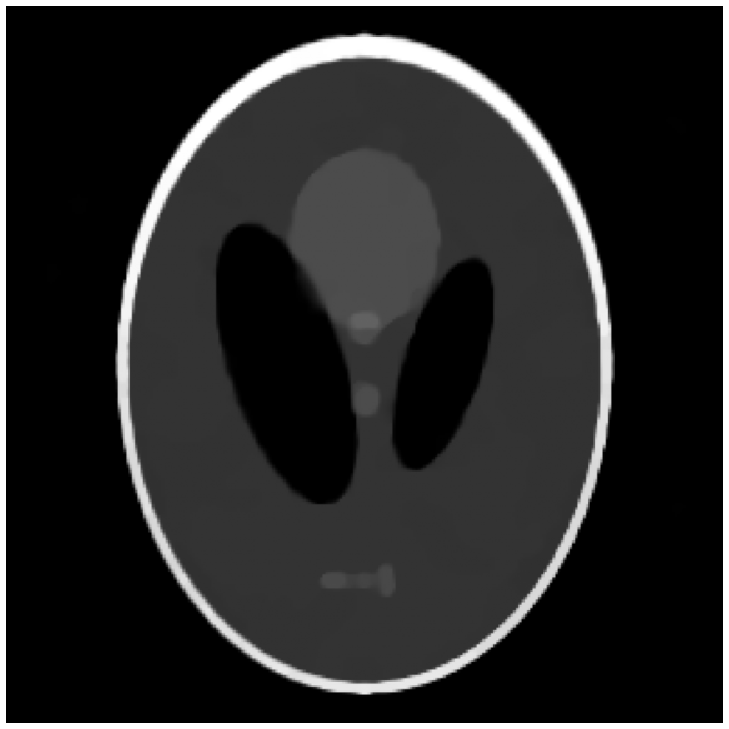}
}\hfill
\subfloat[Algorithm~\ref{alg:extened psi} (PSNR $=29.62$)]{
  \includegraphics[width=0.25\linewidth]{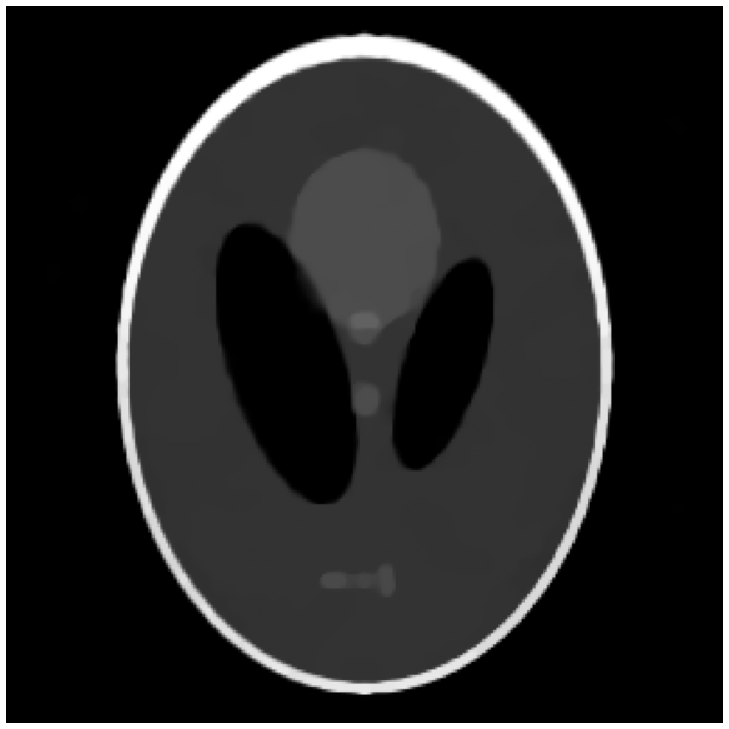}
}\hfill
\subfloat[GrpADMM (PSNR $=24.44$)\label{subfig:deblur-grpadmm}]{
  \includegraphics[width=0.25\linewidth]{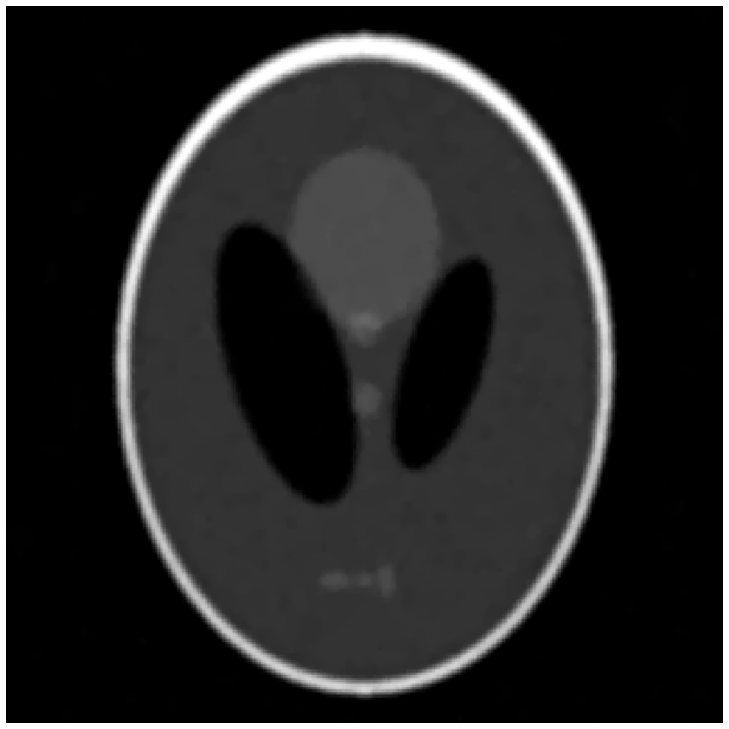}
}

\caption{TV-regularized image deblurring. Top row: (a) ground truth, (b) blurry and noisy observation, and (c) reconstructed by PADMM. Bottom row: (d) reconstructed by Algorithm~\ref{alg:2}, (e) reconstructed by Algorithm~\ref{alg:extened psi}, and (f) reconstructed by GrpADMM.}
\label{fig:deblur-recon}
\end{figure}
\begin{figure}[htbp]
\centering

\subfloat[$\beta=0.1$]{
  \includegraphics[width=0.25\linewidth]{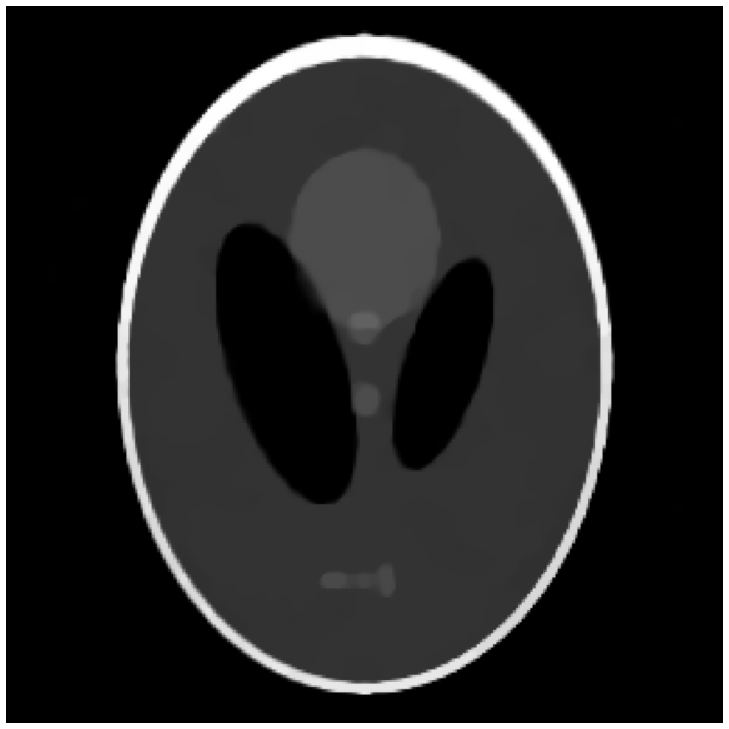}
}\hfill
\subfloat[$\beta=0.5$]{
  \includegraphics[width=0.25\linewidth]{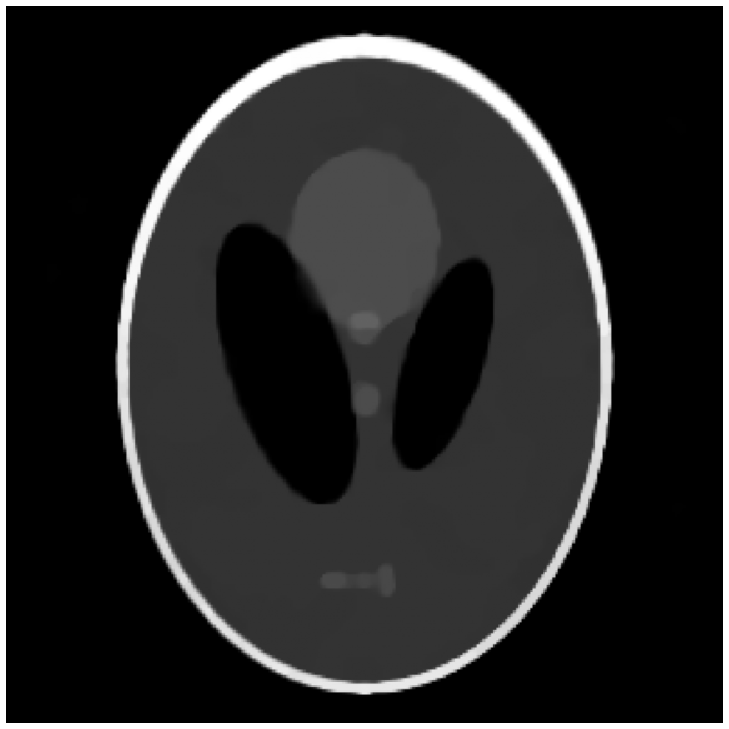}
}\hfill
\subfloat[$\beta=2$]{
  \includegraphics[width=0.25\linewidth]{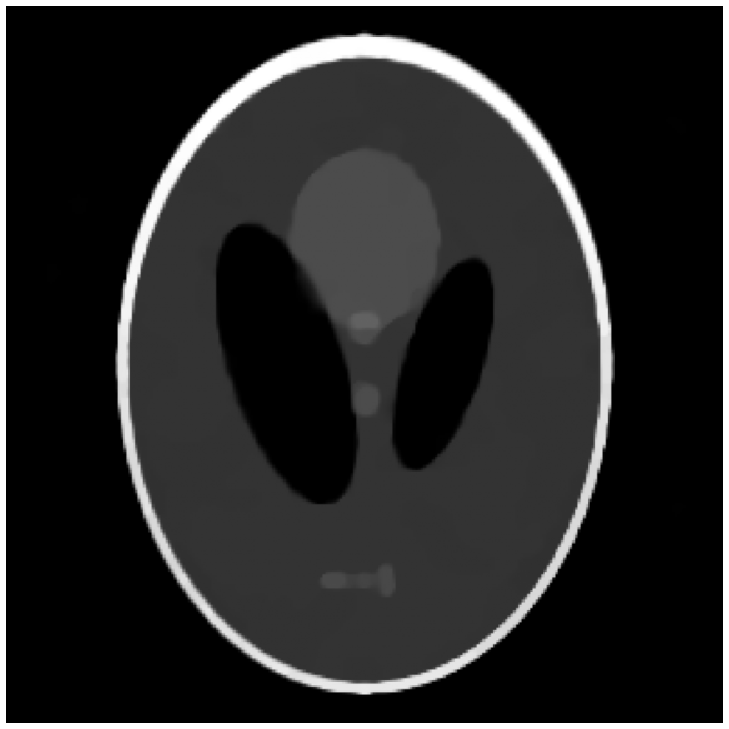}
}

\vspace{0.6em}

\subfloat[$\beta=3$]{
  \includegraphics[width=0.25\linewidth]{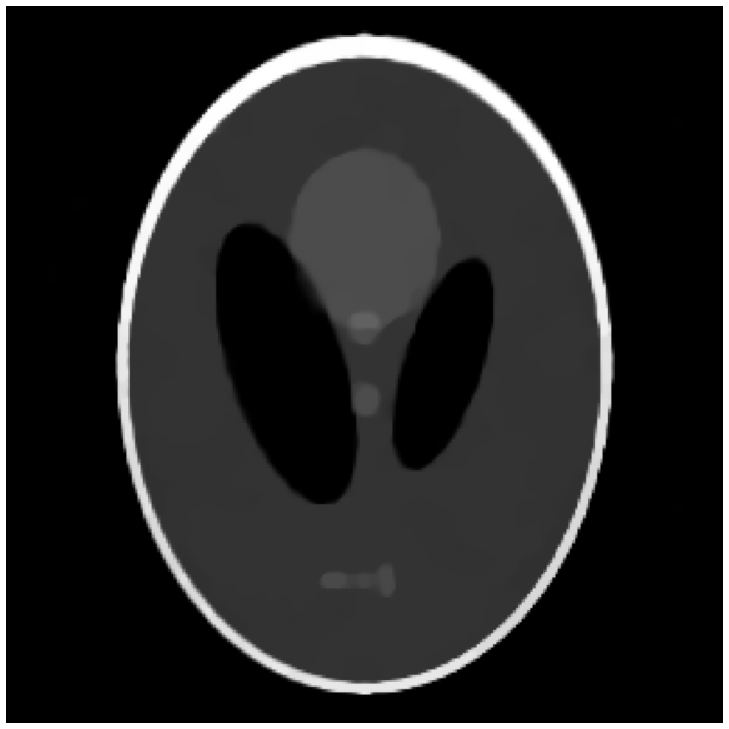}
}\hfill
\subfloat[$\beta=7$]{
  \includegraphics[width=0.25\linewidth]{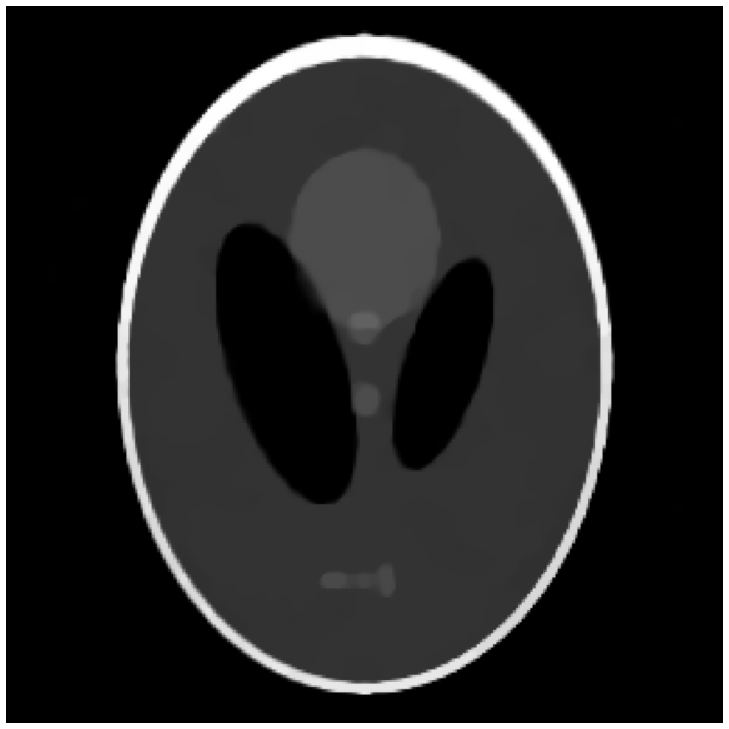}
}\hfill
\subfloat[$\beta=10$\label{subfig:deblur-grpadmm}]{
  \includegraphics[width=0.25\linewidth]{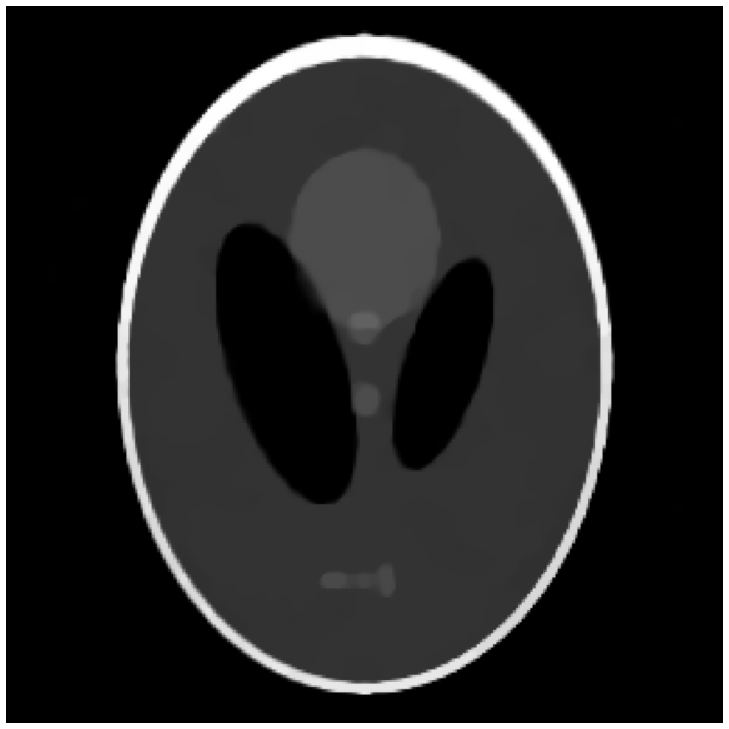}
}

\caption{Influence of $\beta$ on recovered images obtained by Algorithm~\ref{alg:2}. The associated PSNR values are: (a) 29.52 dB, (b) 29.69 dB, (c) 29.68 dB, (d) 29.67 dB, (e) 29.62 dB, and (f) 29.68 dB.}
\label{fig:diferent beta}
\end{figure}
To fit the model \eqref{eq:model}, we introduce two auxiliary variables, namely $w_1=\nabla x$ and $w_2=Hx-c$. Then~\eqref{eq:img-prim} can be rewritten as
\begin{equation*}
    \min_{x,w_1,w_2}
    \;
    \iota_{[0,1]^{N\times N}}(x)
    + \lambda \|w_1\|_{2,1}
    + \frac{\mu}{2}\|w_2\|_2^2
    \quad
    \text{s.t.}
    \quad
    \nabla x-w_1=0,
    \qquad
    Hx-w_2=c.
\end{equation*}
Equivalently, by writing $w=(w_1,w_2)$, we obtain
\[
\min_{x,w}\; g(x)+f(w)
\quad \text{s.t.} \quad
Ax+Bw=b,
\]
where $g(x)=\iota_{[0,1]^{N\times N}}(x),~ f(w_1,w_2)=\lambda\|w_1\|_{2,1}+\frac{\mu}{2}\|w_2\|_2^2,$ and
\[
A=
\begin{bmatrix}
\nabla\\ H
\end{bmatrix},
\quad
B=-I,
\quad
b=
\begin{bmatrix}
0\\ c
\end{bmatrix}.
\]
In all experiments, we take $N=256$ and use a resized Shepp--Logan phantom $x^\star\in[0,1]^{256\times 256}$ as the ground truth. The blur operator is a periodic Gaussian convolution. To create an instance with a relatively large operator norm, we scale the blur operator as $H=\alpha_H\widetilde H$ with $\alpha_H=20$, where $\widetilde H$ denotes the normalised Gaussian blur. Then the observation is generated by taking
\[
c=\alpha_H(\widetilde Hx^\star+\varepsilon),
\quad
\varepsilon\sim\mathcal N(0,\bar\sigma^2I),
\quad
\bar\sigma=0.01.
\]
In this way, the operator $A$ becomes significantly larger than in the unscaled case. After a moderate tuning, we set $\lambda=10^{-2}$ and $\mu=2.5\times10^{-3}$. For Algorithm~\ref{alg:1}, Algorithm~\ref{alg:2}, and \eqref{eq:GrpADMM}, we choose
\[
S=s_xI,
\quad
T=
\begin{bmatrix}
t_{\rm tv}I & 0\\
0 & t_{\rm data}I
\end{bmatrix},
\quad
t_{\rm tv}=0.10,\quad t_{\rm data}=0.15.
\]
This choice is convenient both theoretically and computationally. Since $g(x)$ is just the box constraint, the $x$-subproblem for Algorithm~\ref{alg:1}, Algorithm~\ref{alg:2}, and \eqref{eq:GrpADMM} becomes an explicit projection
\[
x_k
=
P_{[0,1]^{N\times N}}
\!\left(
u_k-\frac{\tau}{s_x}A^\top y
\right),
\]
with the obvious iteration-dependent values of $\tau$ and $y$. Moreover, here
\[
A^\top y=\nabla^\top y^{(1)}+H^\top y^{(2)}
       =-\mathrm{div}(y^{(1)})+H^\top y^{(2)}.
\]
The block-diagonal choice of $T$ also preserves separability in the $w$-update. Here, $w_1$ corresponds to the TV term, while $w_2$ corresponds to the quadratic data-fidelity term, and these two blocks have rather different numerical behaviour. In practice, taking $t_{\rm data} \ge t_{\rm tv}$ yields slightly better damping in the data block and leads to more stable behaviour on this scaled instance. For the above algorithms with fixed steps, the $w$-subproblem has a closed-form solution. Writing $y=(y^{(1)},y^{(2)})$ and setting
$a_1=\nabla x+\frac{y^{(1)}}{\sigma}$ and
$a_2=Hx-c+\frac{y^{(2)}}{\sigma}$,
we obtain
\[
w_1
=
\operatorname{prox}_{(\lambda/d_1)\|\cdot\|_{2,1}}
\!\left(
\frac{\sigma a_1+t_{\rm tv}w_1^{\rm old}}{d_1}
\right),
\]
where
\[
d_1=\sigma+t_{\rm tv}~~\text{and}~~
w_2
=
\frac{\sigma a_2+t_{\rm data}w_2^{\rm old}}
{\mu+\sigma+t_{\rm data}}.
\]
Thus, the $w_1$-update is a pointwise isotropic soft-thresholding step. For Algorithm~\ref{alg:2}, the proximal term in the $w$-subproblem is scaled by $\frac{1}{\sigma_k}$, so the formulas change slightly. With
$a_{1,k}=\nabla x_k+\frac{y^{(1)}_{k-1}}{\sigma_k}$ and
$a_{2,k}=Hx_k-c+\frac{y^{(2)}_{k-1}}{\sigma_k}$,
the updates become
\[
w_{1,k}
=
\operatorname{prox}_{(\lambda/d_{1,k})\|\cdot\|_{2,1}}
\!\left(
\frac{\sigma_ka_{1,k}+\frac{t_{\rm tv}}{\sigma_k}w_{1,k-1}}{d_{1,k}}
\right),
\quad
d_{1,k}=\sigma_k+\frac{t_{\rm tv}}{\sigma_k},
\]
and
\[
w_{2,k}
=
\frac{
\sigma_ka_{2,k}+\frac{t_{\rm data}}{\sigma_k}w_{2,k-1}
}{
\mu+\sigma_k+\frac{t_{\rm data}}{\sigma_k}
}.
\]
These are exactly the formulas used in the implementation. For PADMM, we take the classical linearised choice
\[
S_{\rm pad}=\frac{1}{\tau_{\rm pad}}I-\sigma_{\rm pad}A^\top A.
\]
With this choice, the $x$-subproblem reduces to the projected gradient-type step
\[
x_{k+1}
=
P_{[0,1]^{N\times N}}
\Big(
x_k-\tau_{\rm pad}A^\top\big(y_k+\sigma_{\rm pad}(Ax_k-w_k-b)\big)
\Big),
\]
which is straightforward to implement and requires only the application of $A$ and $A^\top$. We compare four methods with the following parameters.
\begin{itemize}
    \item \textbf{Algorithm~\ref{alg:2}:}
     $\beta=5$, $\tau_0=0.2$, $S=s_xI$ with $s_x=1$, and $T=\mathrm{blkdiag}(0.10I,0.15I)$.

    \item \textbf{Algorithm~\ref{alg:extened psi}:}
    $\psi=1.7$, $\beta=5$, $\mu_{\rm step}=0.79$, $\tau_0=0.2$, $S=s_xI$ with $s_x=1$, and $T=\mathrm{blkdiag}(0.10I,0.15I)$.

    \item \textbf{\ref{eq:GrpADMM}:}
    $\psi=\varphi$, $\sigma_{\rm grp}=8$, $S=s_xI$ with $s_x=1$, $T=\mathrm{blkdiag}(0.10I,0.15I)$, and
    $\tau_{\rm grp}=\frac{\varphi}{\sigma_{\rm grp}\|A\|^2}$.

    \item \textbf{\ref{eq:PADMM}:}
    $\sigma_{\rm pad}=15$, $\tau_{\rm pad}=\frac{0.99}{\sigma_{\rm pad}\|A\|^2}$,
    $S_{\rm pad}=\frac{1}{\tau_{\rm pad}}I-\sigma_{\rm pad}A^\top A$, and
    $T=\mathrm{blkdiag}(0.10I,0.15I)$.
\end{itemize}
We initialize all methods with $x_0=P_{[0,1]^{N\times N}}(c/\alpha_H),~u_0=x_0,~w_{1,0}=\nabla x_0,~w_{2,0}=Hx_0-c,~y_0=0.$ For each iterate $(x_k,w_{1,k},w_{2,k},y_k^{(1)},y_k^{(2)})$, the combined KKT residual \eqref{eq:kkt-residual} is given by
\[
\operatorname{KKTres}_k
:=
\sqrt{r_{x,k}^2+r_{w,k}^2+r_{p,k}^2},
\]
where
\[
r_{x,k}
:=
\left\|
x_k-
P_{[0,1]^{N\times N}}
\bigl(
x_k-\nabla^\top y_k^{(1)}-H^\top y_k^{(2)}
\bigr)
\right\|,
\]
\[
r_{w,k}^2
:=
\left\|
w_{1,k}-
\operatorname{prox}_{\lambda\|\cdot\|_{2,1}}
\bigl(
w_{1,k}+y_k^{(1)}
\bigr)
\right\|^2
+
\left\|
\mu w_{2,k}-y_k^{(2)}
\right\|^2,
\]
and
\[
r_{p,k}^2
:=
\|\nabla x_k-w_{1,k}\|^2
+
\|Hx_k-w_{2,k}-c\|^2.
\]
The numerical results in Figures~\ref{fig:image_deblur} and~\ref{fig:deblur-recon} show a clear advantage of the two proposed algorithms over the fixed-step cases on this image deblurring instance. From Figure~\ref{fig:image_deblur}, both Algorithms~\ref{alg:extened psi} and \ref{alg:2} decrease the relative objective gap substantially faster than GrpADMM and PADMM, with Algorithm~\ref{alg:2} giving the best overall performance and Algorithm~\ref{alg:extened psi} following very closely. A similar trend is observed in the combined KKT residual graph. Although Figure~\ref{fig:image_deblur} shows that the fixed-step methods can produce slightly smaller raw feasibility residuals in part of the run, this advantage is not reflected in either the objective decrease or the overall KKT residual. From Figure~\ref{fig:deblur-recon}, we can observe that both Algorithms~\ref{alg:extened psi} and \ref{alg:2} recover sharper boundaries and finer structures than the other two algorithms. In particular, the images produced by GrpADMM and PADMM remain visibly more blurred, whereas the proposed algorithms recover the phantom's main anatomical features much more accurately. Furthermore, one can observe from Figure~\ref{fig:diferent beta} that, when $\beta\in[1,5]$, all reconstructions are of high quality. Hence, for this deblurring problem, Algorithm~\ref{alg:2} appears quite robust to $\beta$, and a broad range of values yields good reconstruction performance.
\subsection{Unbalanced optimal transport problem}
\label{sec:num-ot-split}

We consider the Kantorovich optimal transport (OT) problem \cite{Villani2003,PeyreCuturi2019,Chizat2018Unbalanced}
\begin{equation*}
  \min_{X\in\mathbb{R}_+^{n_s\times n_t}} \ \langle C,X\rangle
  \quad \text{subject to}\quad
  X\mathbf{1}=a,\qquad X^\top\mathbf{1}=b,
\end{equation*}
where $C\in\mathbb{R}^{n_s\times n_t}$ is the transport cost matrix, and $a\in\Delta^{n_s}$, $b\in\Delta^{n_t}$ are prescribed source and target histograms on the probability simplex $\Delta^n:=\{u\in\mathbb{R}^n_+:\ \mathbf{1}^\top u=1\}.$ To relax the marginal equalities, we adopt the
\emph{unbalanced} OT model with a squared-$\ell_2$ penalty on marginal violations (see \cite[Page 13]{Chapel2021UOTregpath} for more details):
\begin{equation*}
  \min_{X\ge 0}\ \langle C,X\rangle
  + \frac{\gamma}{2}\big(\|X\mathbf{1}-a\|_2^2 + \|X^\top\mathbf{1}-b\|_2^2\big),
  \qquad \gamma>0,
\end{equation*}
which is a classical quadratic-penalty relaxation of the equality constraints; see, e.g., \cite{PeyreCuturi2019} for quadratic penalties in constrained convex optimisation and for linear OT constraints. This model penalises deviations from the marginal constraints and is particularly convenient for this type of separable convex optimisation problem. We take uniform grids $(s_i)_{i=1}^{n_s}\subset[0,1]$ and $(t_j)_{j=1}^{n_t}\subset[0,1]$, with
\[
s_i=\frac{i-1}{n_s-1},\qquad t_j=\frac{j-1}{n_t-1},
\]
and define the quadratic ground cost
\[
C_{ij}=(s_i-t_j)^2,\qquad i=1,\dots,n_s,\ \ j=1,\dots,n_t.
\]

\begin{figure}[htbp]
\centering

\subfloat[Relative objective gap \label{subfig:ot-feas-100}]{
  \includegraphics[width=0.30\linewidth]{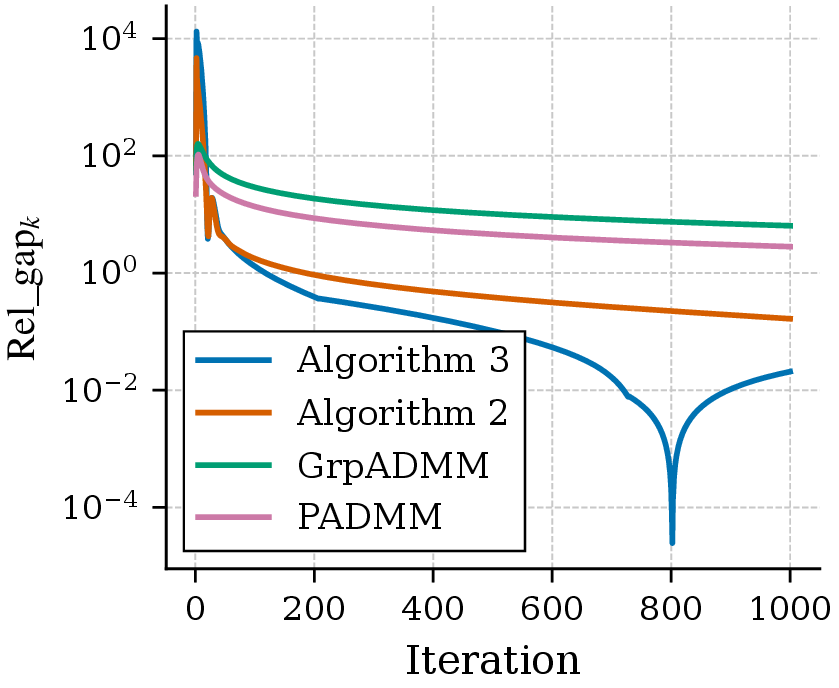}
}\hfill
\subfloat[Feasibility residual\label{subfig:ot-gap-100}]{
  \includegraphics[width=0.30\linewidth]{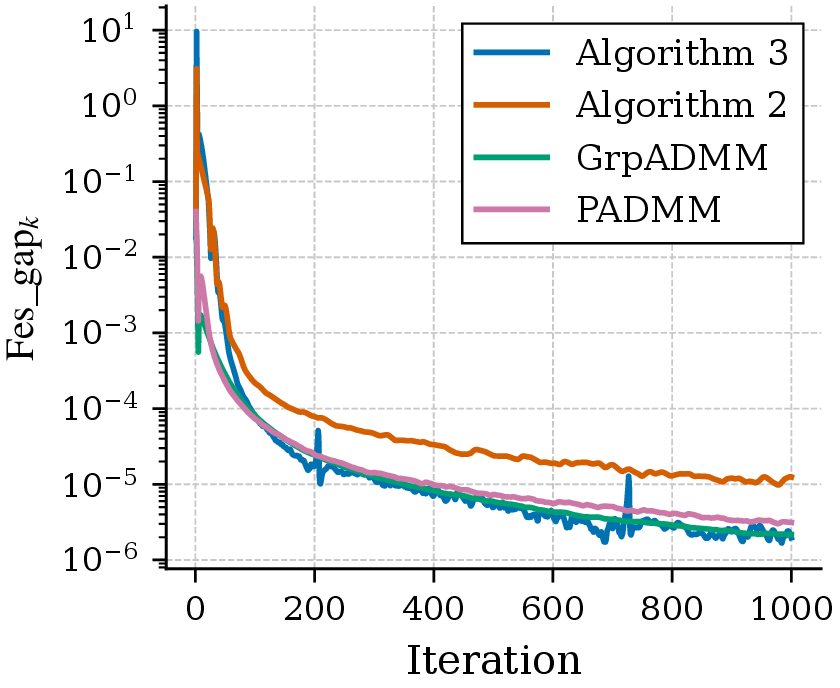}
}

\vspace{0.6em}

\subfloat[KKT residual \eqref{eq:uot-kkt-residual}\label{subfig:ot-feas-30}]{
  \includegraphics[width=0.30\linewidth]{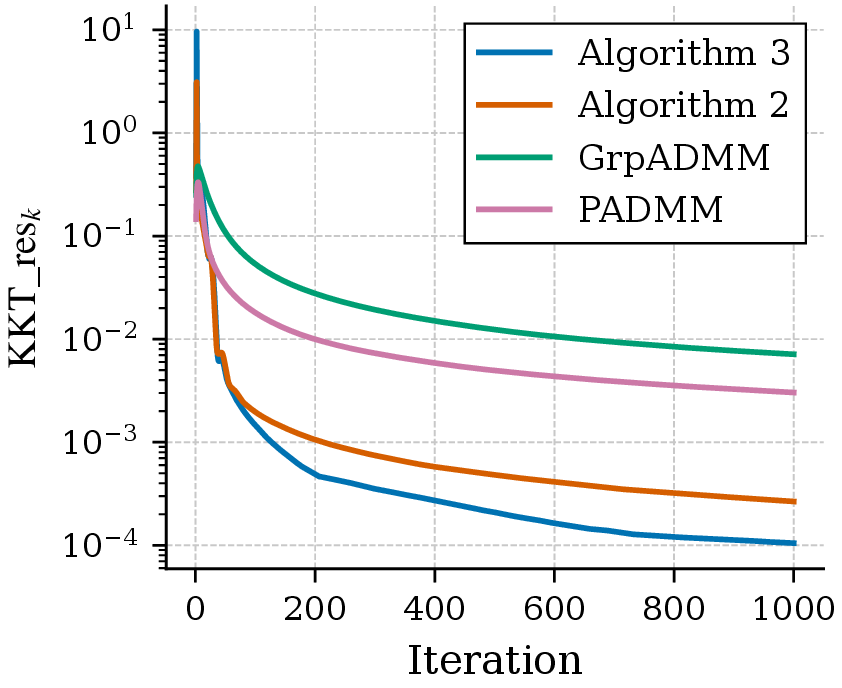}
}\hfill
\subfloat[Primal step-size ($\tau_k$)\label{subfig:ot-gap-30}]{
  \includegraphics[width=0.30\linewidth]{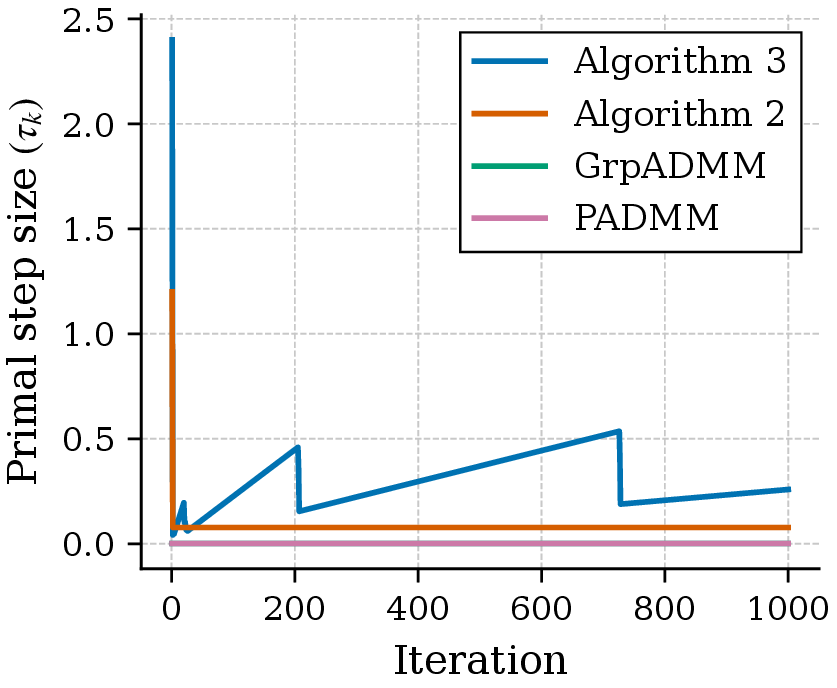}
}

\caption{Convergence plot summaries for the unbalanced optimal transport experiment with dimension $(n_s, n_t)=(1000,500)$.}
\label{fig:transport_plan_residuals_ns_1000}
\end{figure}

The entries of $a\in\mathbb{R}^{n_s}$ and $b\in\mathbb{R}^{n_t}$ are sampled independently from the uniform distribution on $(0,1)$, and then normalized so that $a\in\Delta^{n_s}$ and $b\in\Delta^{n_t}$. Let $x=\mathrm{vec}(X)\in\mathbb{R}^{n_sn_t}$ denote the row-major vectorization of $X$, namely
\[
\mathrm{vec}(X)
=
[X_{1,1},\ldots,X_{1,n_t},\ X_{2,1},\ldots,X_{2,n_t},\ \ldots,\ X_{n_s,1},\ldots,X_{n_s,n_t}]^\top,
\]
and set $c=\mathrm{vec}(C)$. Define the linear operator $A:\mathbb{R}^{n_sn_t}\to\mathbb{R}^{n_s+n_t}$ by
\[
(Ax)_i=\sum_{j=1}^{n_t}X_{ij},\qquad i=1,\ldots,n_s,
\]
and
\[
(Ax)_{n_s+j}=\sum_{i=1}^{n_s}X_{ij},\qquad j=1,\ldots,n_t.
\]
Thus, $Ax$ stacks the row sums and column sums of the transport plan. If we define $\hat b:=\begin{bmatrix} a \\ b \end{bmatrix}\in\mathbb{R}^{n_s+n_t},$ and introduce an auxiliary variable $w\in\mathbb{R}^{n_s+n_t}$, then the problem can be written in the split form
\begin{equation*}
  \min_{x,w}\ \Phi(x,w)
  :=
  \underbrace{\langle c,x\rangle+\iota_{\mathbb{R}_+^{n_sn_t}}(x)}_{g(x)}
  +
  \underbrace{\frac{\gamma}{2}\|w\|_2^2}_{f(w)}
  \quad\text{subject to}\quad
  Ax+w=\hat b.
\end{equation*}
Here $\iota_{\mathbb{R}_+^{n_sn_t}}$ denotes the indicator function of the nonnegative orthant. In the implementation, the operator $A$ is applied implicitly through row and column summations rather than formed as a dense matrix, which is especially important for the larger instances reported below. In our experiments, we consider two problem sizes
\[
(n_s,n_t)=(1000,500)
\qquad\text{and}\qquad
(n_s,n_t)=(2000,1000).
\]
Hence $A\in\mathbb{R}^{1500 \times 500000}$ and $A\in\mathbb{R}^{3000 \times 2000000}$, respectively.
We set $\gamma=1$, and initialize all methods with $x^0=0\in\mathbb{R}^{n_sn_t},~
w^0=0\in\mathbb{R}^{n_s+n_t},~
y^0=0\in\mathbb{R}^{n_s+n_t},$ and run each method for $1000$ iterations. The algorithmic parameters are chosen as follows:
\begin{itemize}
  \item \textbf{Algorithm \ref{alg:2}}: $S=I,~T=\eta_w I, \tau_0=1.2,~\eta_w=10^{-2}$.

  \item \textbf{Algorithm \ref{alg:extened psi}}: $S=I,~T=0,~\psi=1.75,~\beta=1,~\mu=0.68,~\tau_0=1.2.$
  \item \textbf{GrpADMM}: $S=I,~T=0,~\sigma=1,~\tau=\frac{\varphi}{\sigma\|A\|^2},~ \varphi=\frac{1+\sqrt5}{2}.$

  \item \textbf{PADMM}: $S=I,~T=0,~\sigma=1,~ \tau=\frac{0.99}{\sigma\|A\|^2}.$

\end{itemize}
\begin{figure}[htbp]
\centering

\subfloat[Algorithm~\ref{alg:2}\label{fig:plan-alg2}]{
  \includegraphics[width=0.44\textwidth]{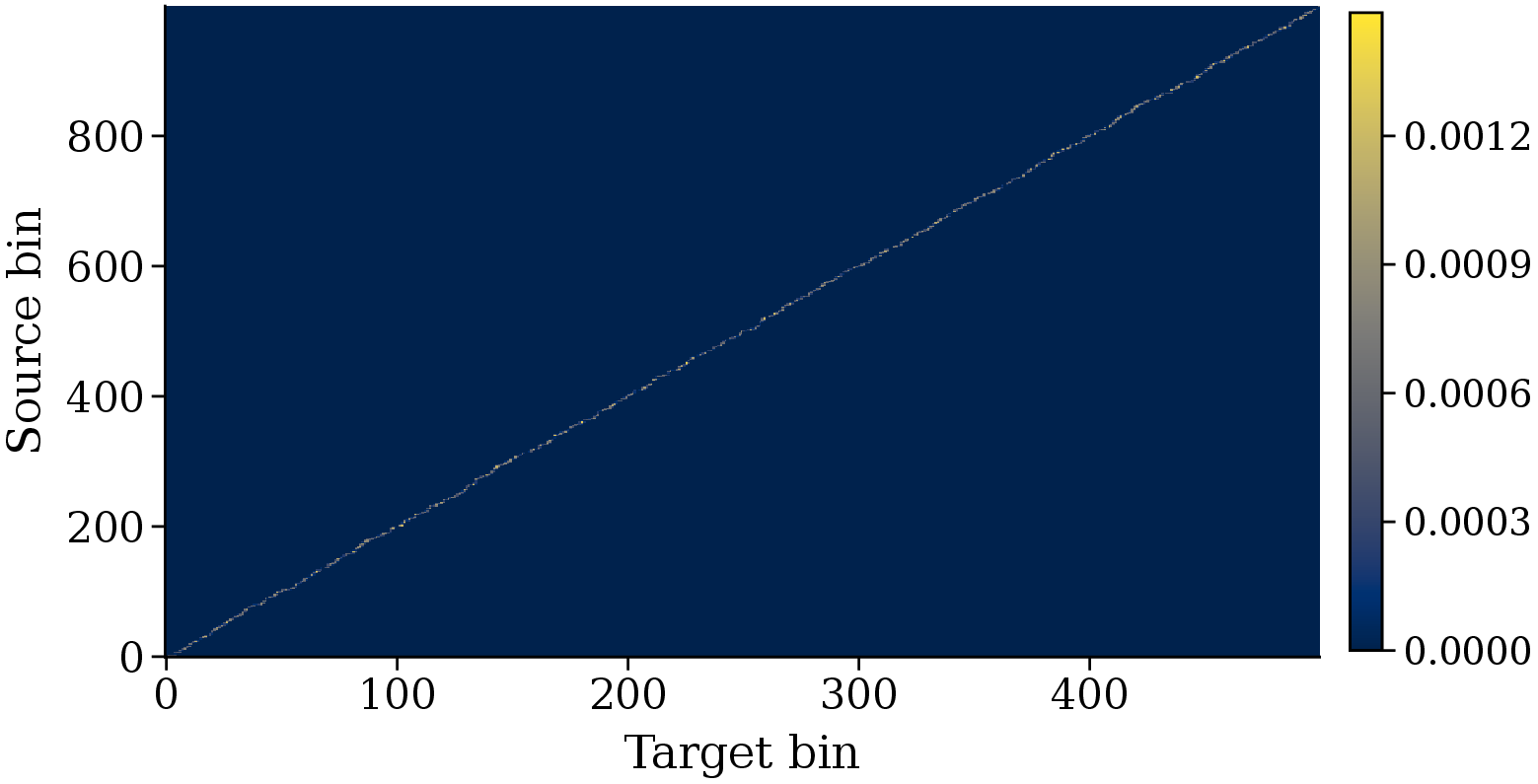}
}\hfill
\subfloat[Algorithm~\ref{alg:extened psi}\label{fig:plan-alg1}]{
  \includegraphics[width=0.44\textwidth]{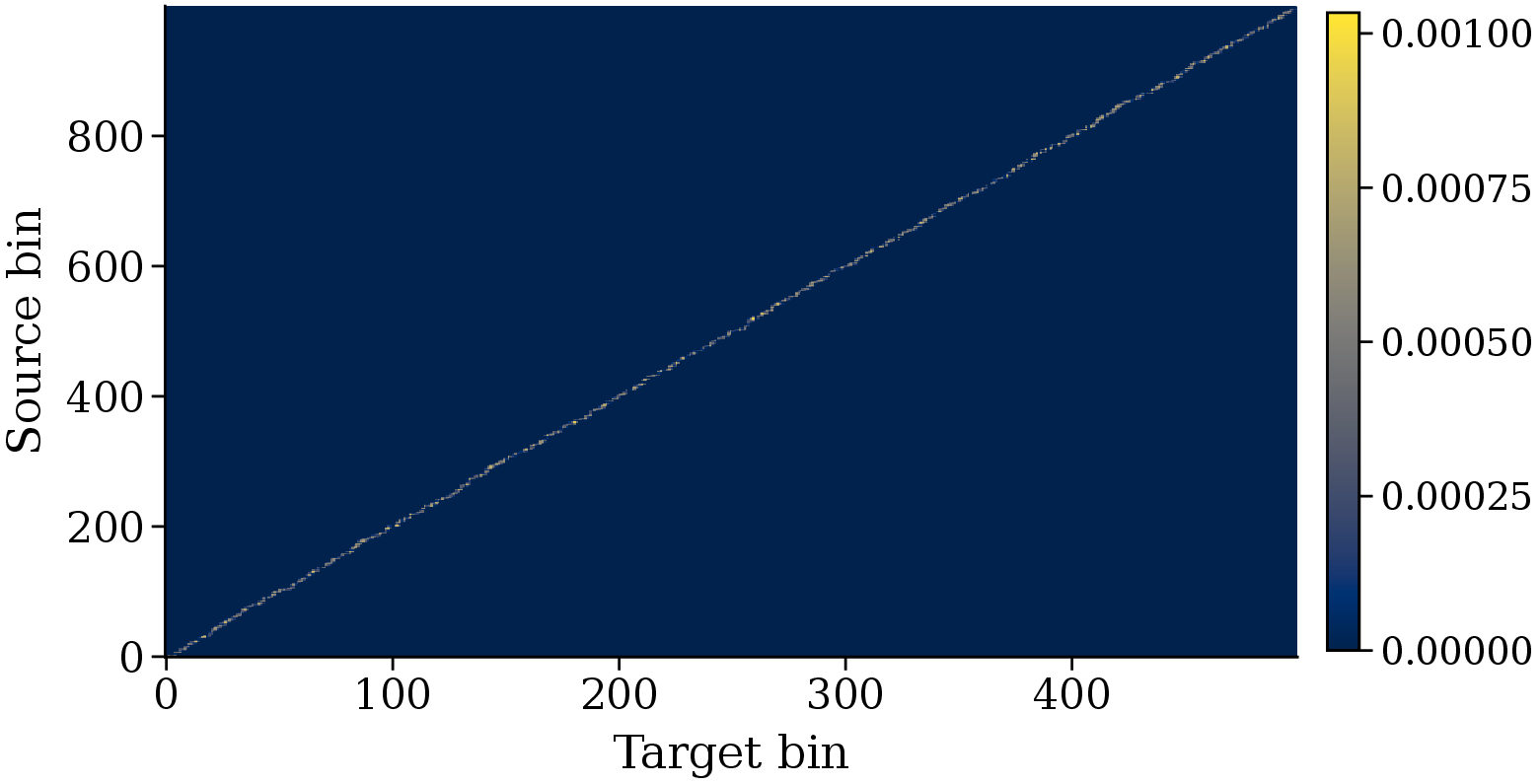}
}

\vspace{0.75em}

\subfloat[GrpADMM\label{fig:plan-grpadmm}]{
  \includegraphics[width=0.44\textwidth]{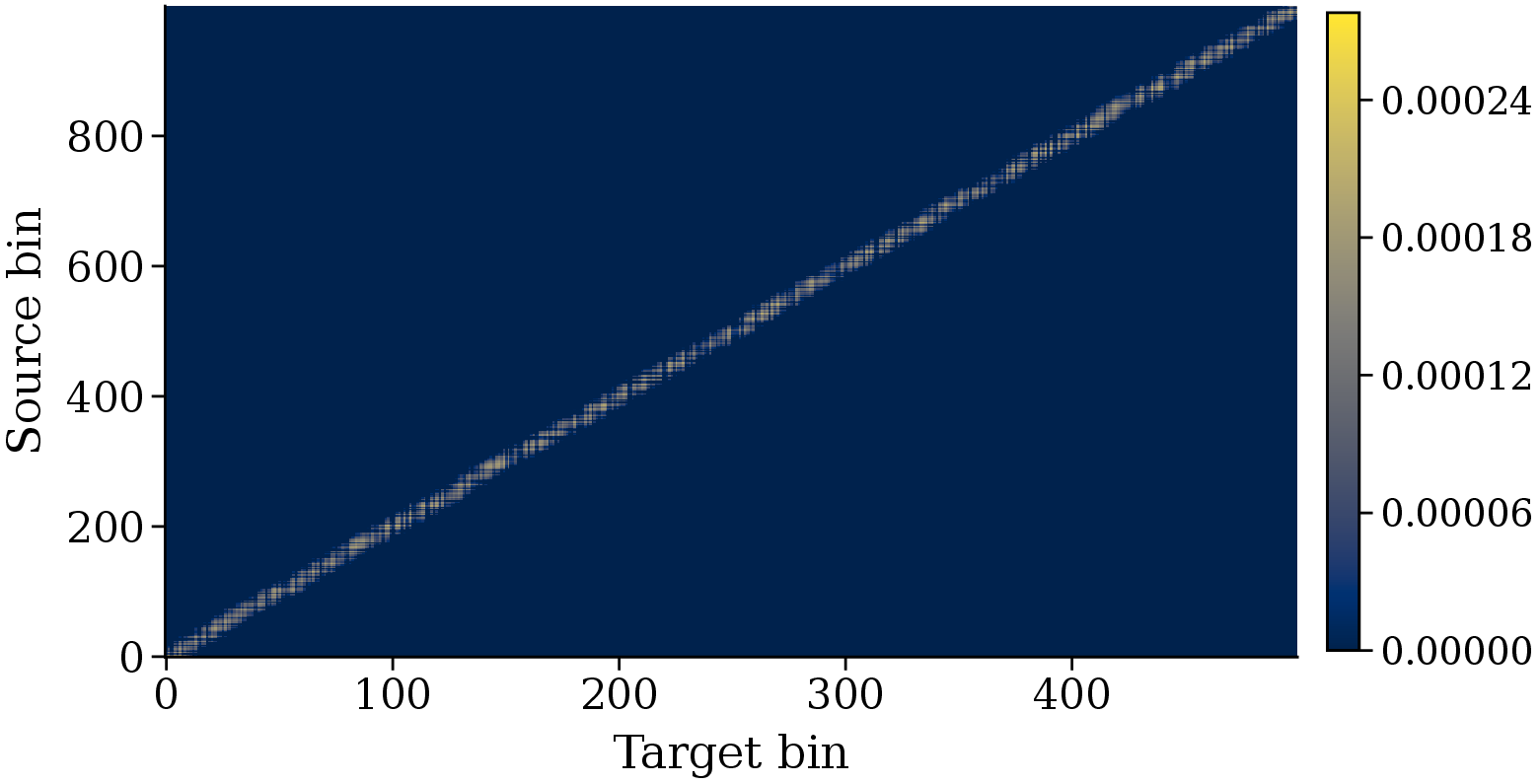}
}\hfill
\subfloat[PADMM\label{fig:plan-padmm}]{
  \includegraphics[width=0.44\textwidth]{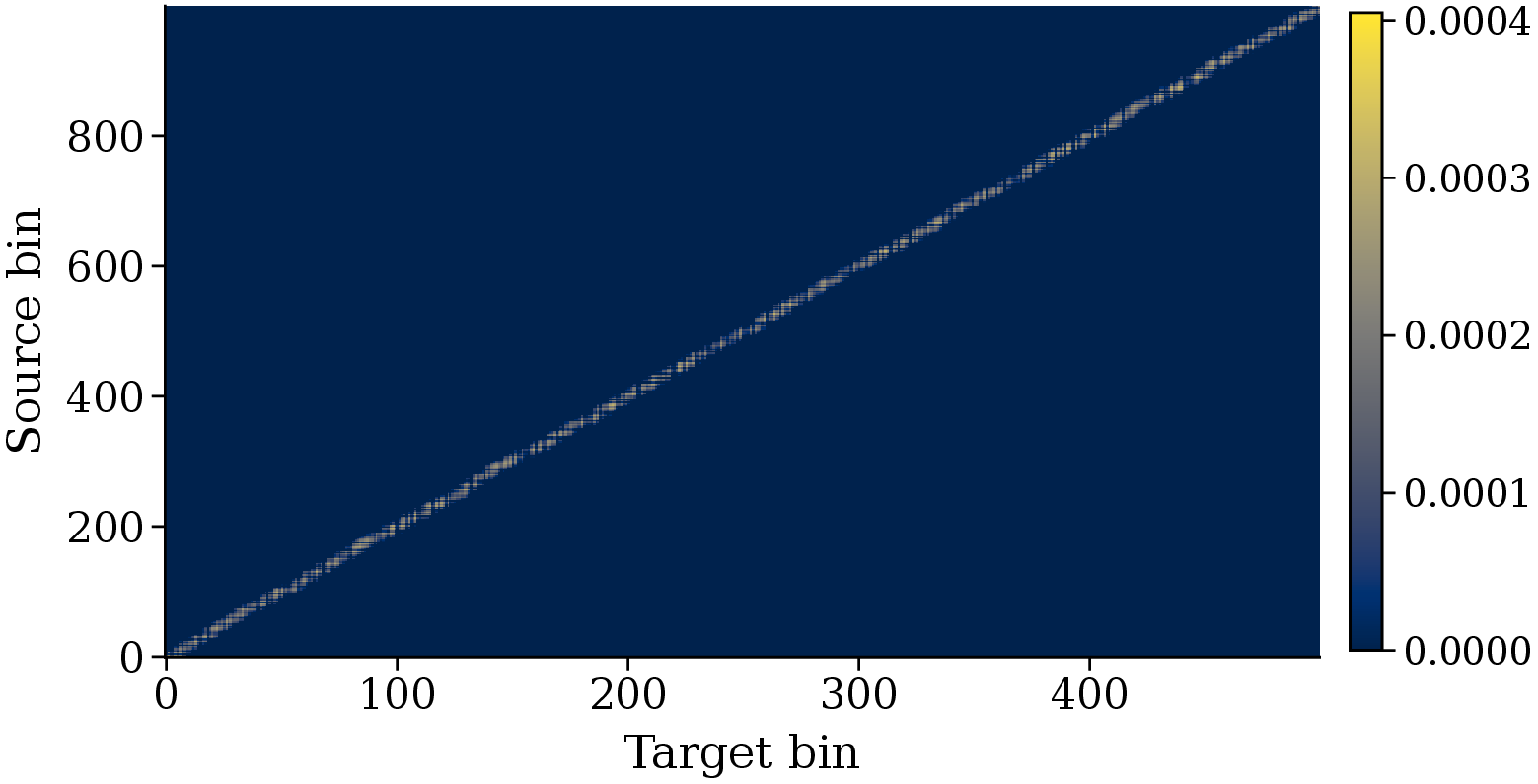}
}

\caption{Transport plans as heatmaps (rows: sources, columns: targets). 
Brighter cells mean more mass sent from $i$ to $j$.
All panels share a common colour scale; the near-diagonal structure is essentially identical across methods, indicating comparable solution quality. Here the problem size is $(n_s, n_t)=(1000,500)$.}
\label{fig:transport-plans_bin_ns_1000}
\end{figure}
The $x$-updates in  Algorithm \ref{alg:extened psi} reduce to shifted Euclidean projections onto the nonnegative orthant
\[
x^{k+1}
=
\Pi_{\mathbb{R}_+^{\,n_sn_t}}
\bigl(u^k-\tau_k(A^\top y^k+c)\bigr).
\]
For PADMM, the $x$-update is the usual linearised step
\[
x^{k+1}
=
\Pi_{\mathbb{R}_+^{\,n_sn_t}}
\!\left(
x^k-\tau\Big[c+A^\top\!\big(y^k+\sigma(Ax^k+w^k-\hat b)\big)\Big]
\right).
\]
The $w$-updates also admit closed forms. Let $z^k:=\hat b-Ax^{k+1}-\frac{1}{\sigma_k}y^k.$ Then with the nonzero proximal weight $T=\eta_w I$, the  $w$-subproblem of Algorithm \ref{alg:2} becomes
\begin{equation}\label{eq:uot-w-update-alg3}
w^{k+1}
=
\frac{\sigma_k z^k + (\eta_w/\sigma_k) w^k}
{\gamma+\sigma_k+\eta_w/\sigma_k}.
\end{equation}
Thus, compared with the standard shrinkage step, Algorithm \ref{alg:2} incorporates a mild memory term toward the previous iterate $w^k$. This slight damping effect stabilises the marginal-correction variable and leads to visibly improved practical behaviour, especially in the feasibility and KKT residuals. For Algorithm \ref{alg:extened psi}, GrpADMM, and PADMM, we set $\eta_w=0$, so the $w$-update reduces to the simpler shrinkage formula
\begin{equation*}
w^{k+1}
=
\frac{\sigma_k}{\sigma_k+\gamma}
\left(
\hat b-Ax^{k+1}-\frac{1}{\sigma_k}y^k
\right),
\end{equation*}
with $\sigma_k\equiv \sigma~\forall k$ in the fixed step-size counterparts. For this experiment, the combined KKT residual \eqref{eq:kkt-residual} simplifies to
\begin{equation}\label{eq:uot-kkt-residual}
\mathrm{KKT\_res}_k
:=
\sqrt{(r_k^x)^2 + (r_k^w)^2 + (r_k^p)^2},
\end{equation}
where
\[
r_k^x
=
\left\|
x^k-\Pi_{\mathbb{R}_+^{n_sn_t}}(x^k-A^\top y^k-c)
\right\|,~~
r_k^w
=
\left\|
w^k-\frac{1}{1+\gamma}(w^k-y^k)
\right\|,~~r_k^p=\|Ax^k+w^k-\hat b\|.
\]

\begin{figure}[htbp]
\centering

\subfloat[Relative objective gap \label{subfig:ot-feas-100}]{
  \includegraphics[width=0.30\linewidth]{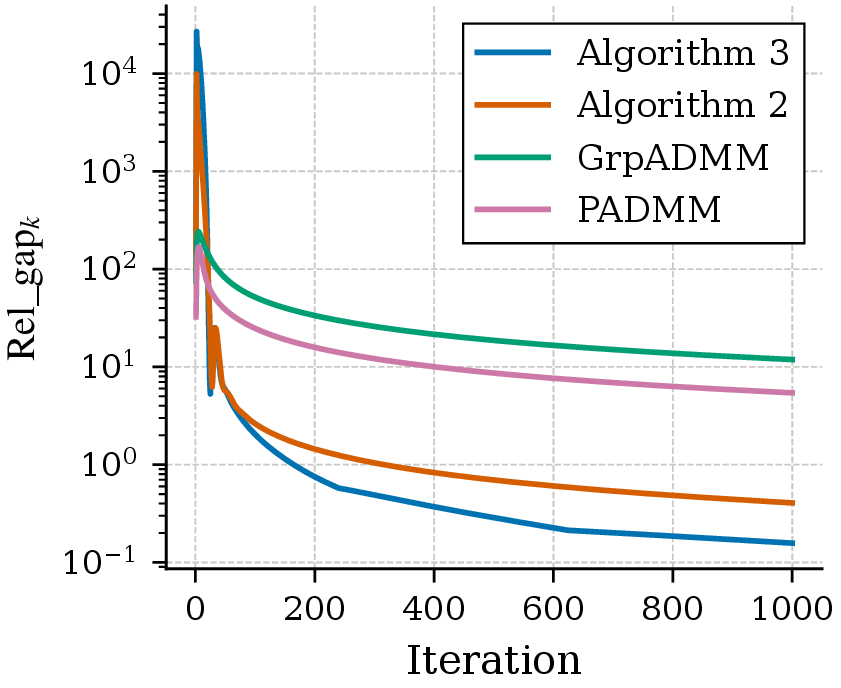}
}\hfill
\subfloat[Feasibility residual\label{subfig:ot-gap-100}]{
  \includegraphics[width=0.30\linewidth]{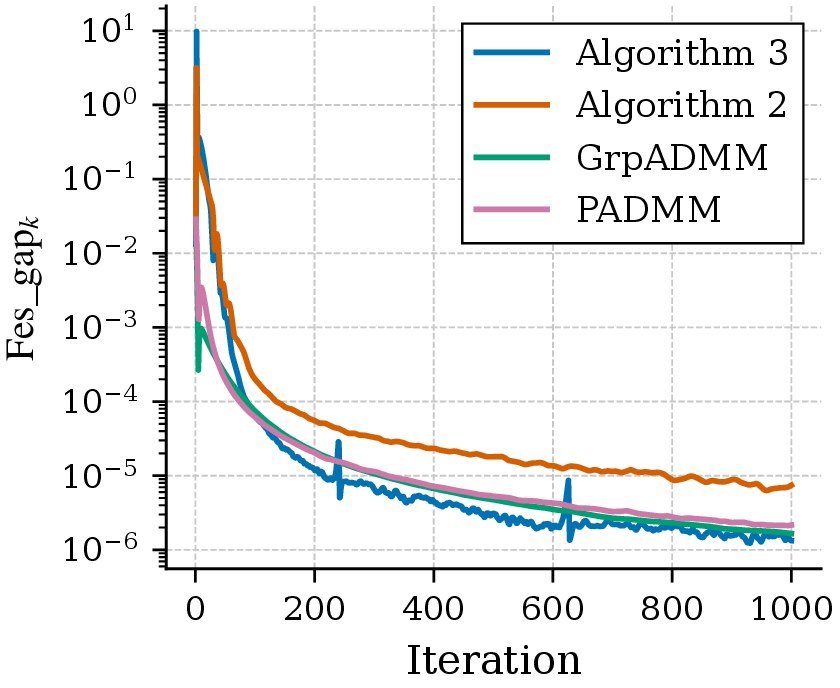}
}

\vspace{0.6em}

\subfloat[KKT residual \eqref{eq:uot-kkt-residual}\label{subfig:ot-feas-30}]{
  \includegraphics[width=0.30\linewidth]{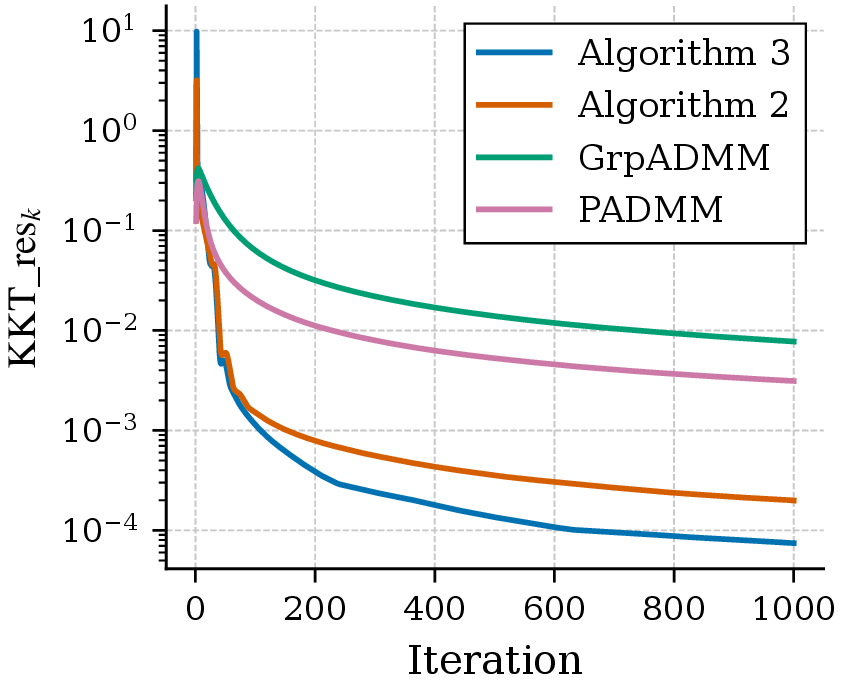}
}\hfill
\subfloat[Primal step-size ($\tau_k$)\label{subfig:ot-gap-30}]{
  \includegraphics[width=0.30\linewidth]{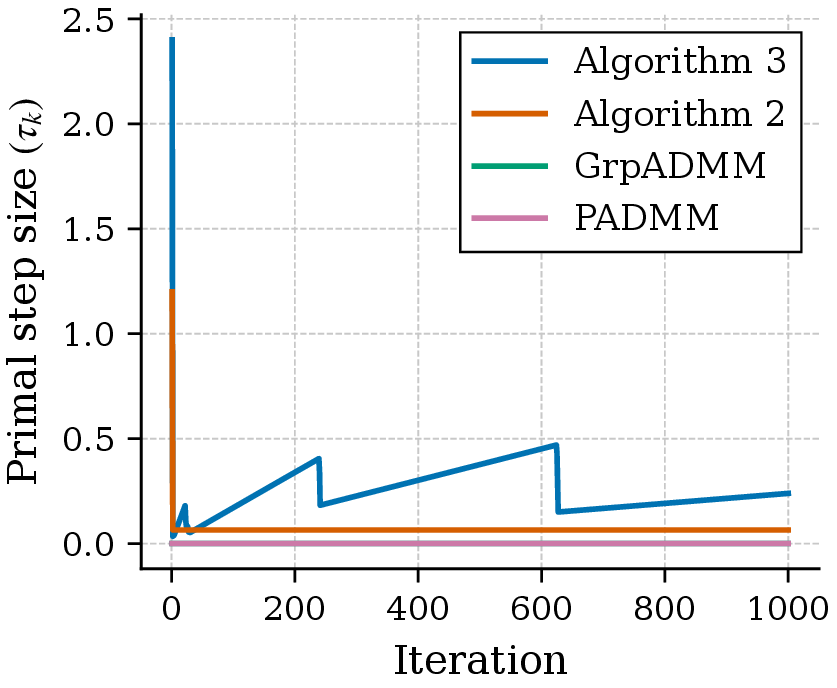}
}

\caption{Convergence plot summaries for the unbalanced optimal transport experiment with dimension $(n_s, n_t)=(2000,1000)$.}
\label{fig:transport_plan_residual_ns_2000}
\end{figure}

\begin{figure}[htbp]
\centering

\subfloat[Algorithm~\ref{alg:2}\label{fig:plan-alg2}]{
  \includegraphics[width=0.44\textwidth]{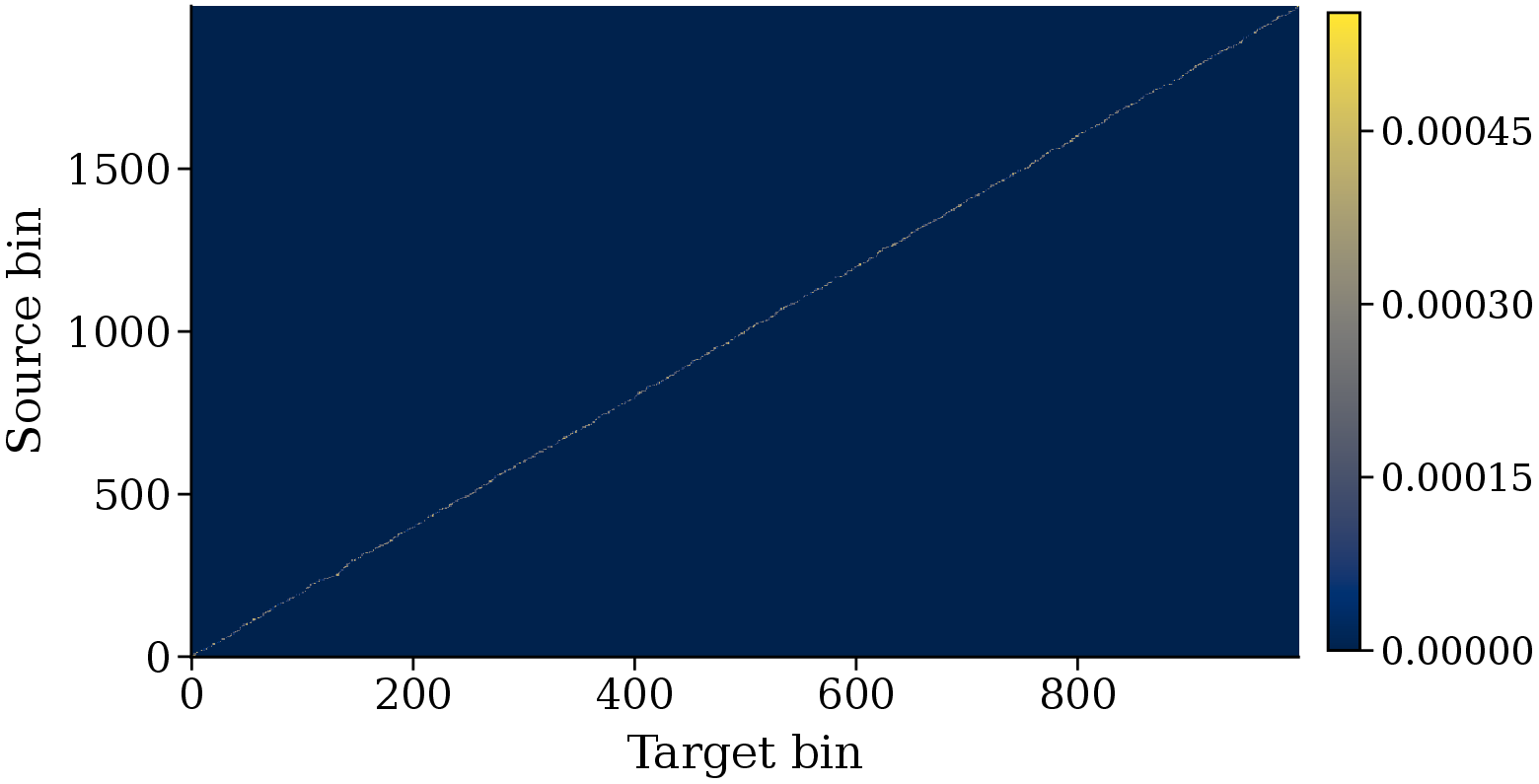}
}\hfill
\subfloat[Algorithm~\ref{alg:extened psi}\label{fig:plan-alg1}]{
  \includegraphics[width=0.44\textwidth]{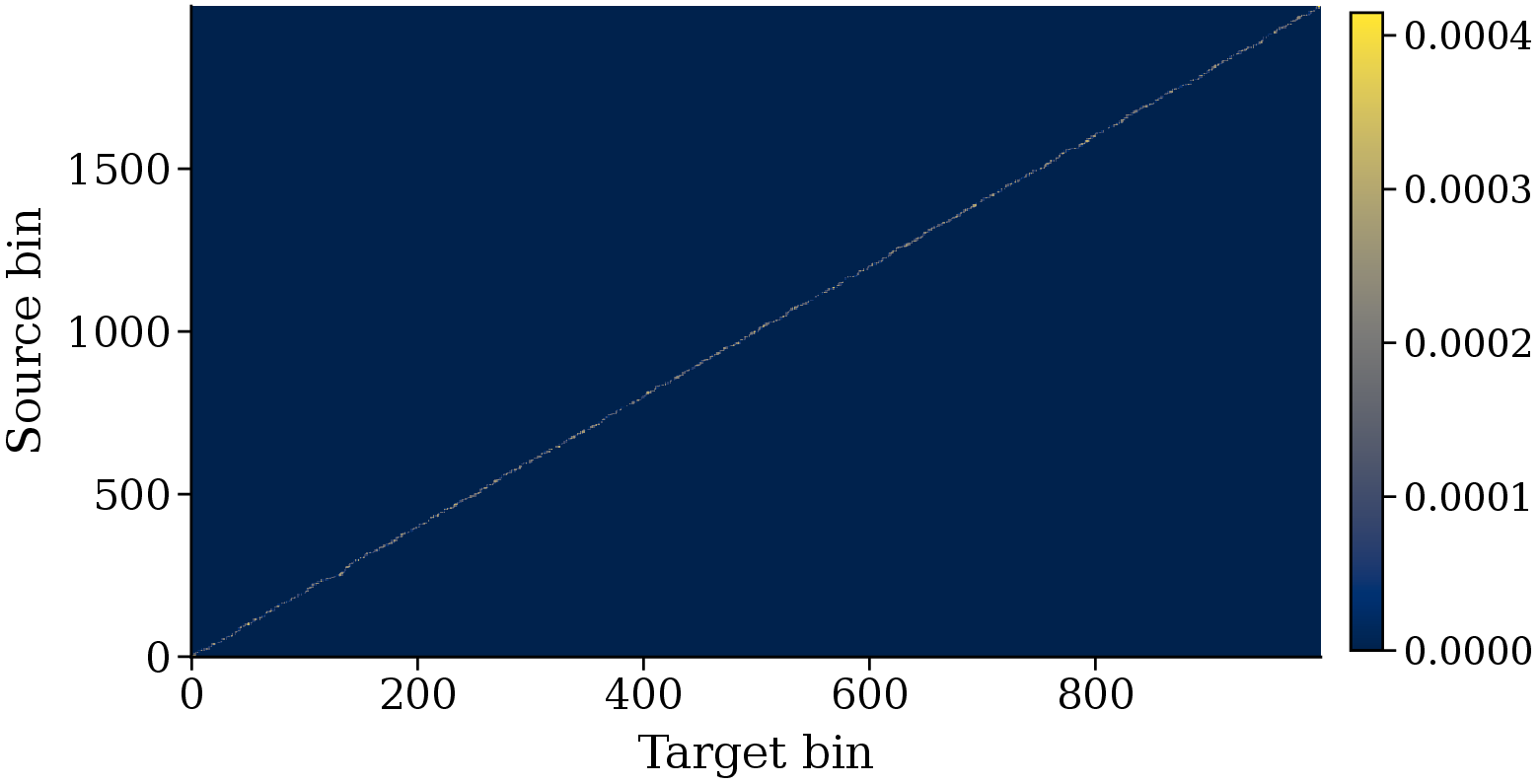}
}

\vspace{0.75em}

\subfloat[GrpADMM\label{fig:plan-grpadmm}]{
  \includegraphics[width=0.44\textwidth]{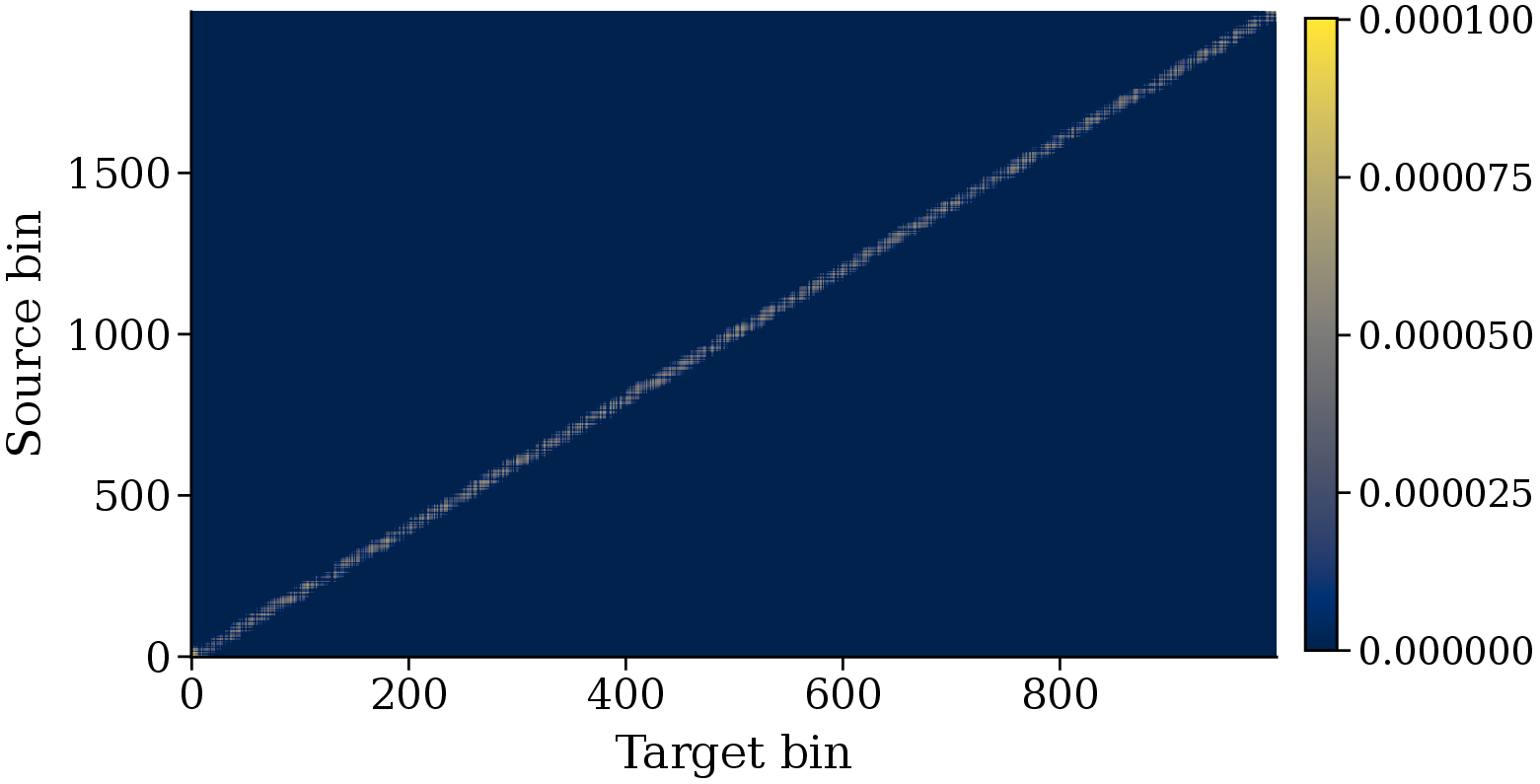}
}\hfill
\subfloat[PADMM\label{fig:plan-padmm}]{
  \includegraphics[width=0.44\textwidth]{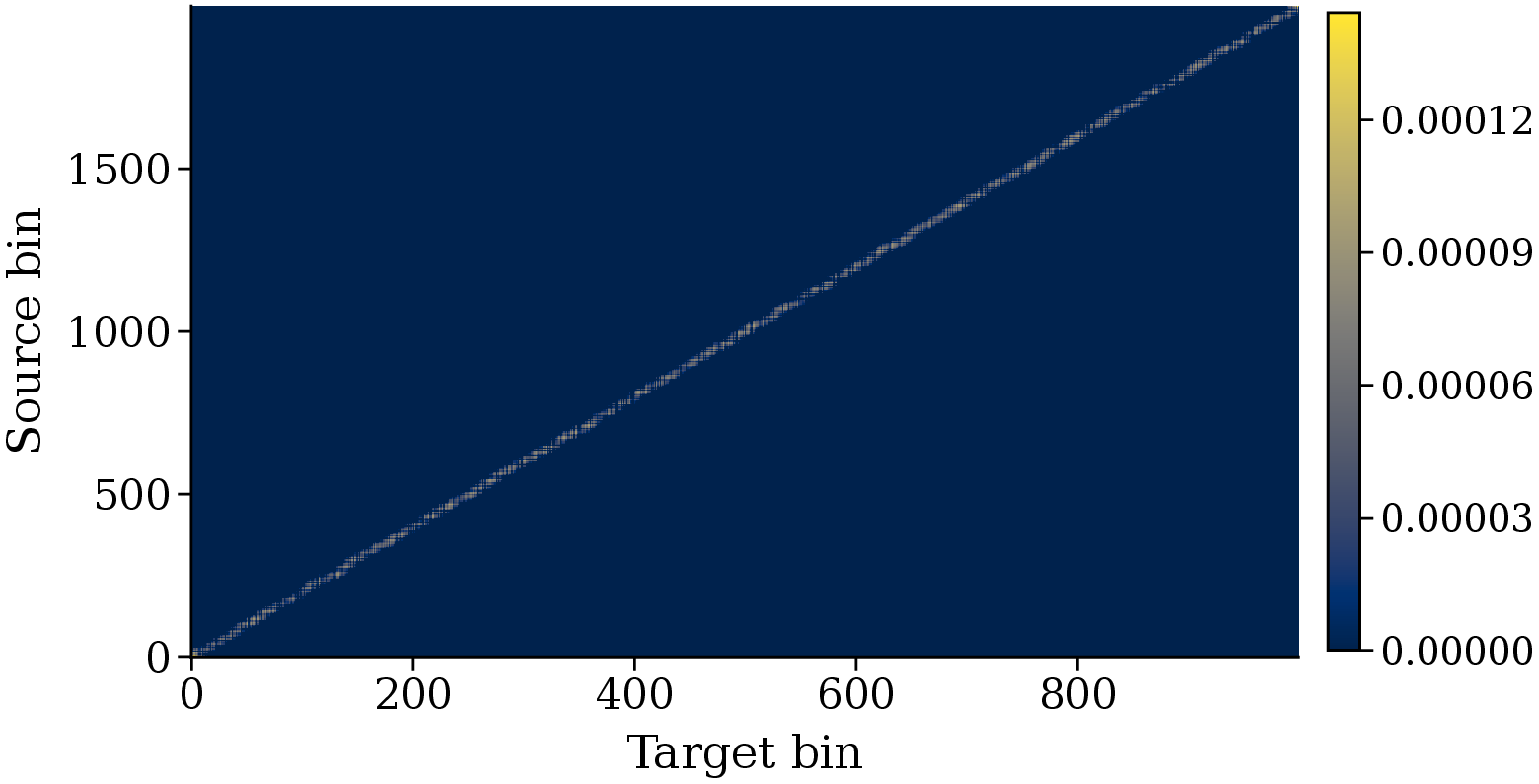}
}

\caption{Transport plans as heatmaps (rows: sources, columns: targets). 
Brighter cells mean more mass sent from $i$ to $j$.
All panels share a common colour scale; the near-diagonal structure is essentially identical across methods, indicating comparable solution quality. Here the problem size is $(n_s, n_t)=(2000,1000)$.}
\label{fig:transport-plans_bin_ns_2000}
\end{figure}
We tested two instances of different problem dimensions. Figure~\ref{fig:transport_plan_residuals_ns_1000} and \ref{fig:transport_plan_residual_ns_2000}  display the computed transport plans $X\in\mathbb{R}_+^{n_s\times n_t}$ as heatmaps, where rows correspond to source bins, columns to target bins, and colour intensity represents transported mass. Because the ground cost is quadratic, mass transport over short distances is favoured, and therefore, the dominant mass is concentrated near the diagonal. From Figure \ref{fig:transport-plans_bin_ns_1000} and \ref{fig:transport-plans_bin_ns_2000}, it can be seen that all four methods recover transport plans with the same near-diagonal structure, which is consistent with the geometry of the problem. The essential difference lies not in the final pattern of the plan, but in the speed at which that plan is reached. In particular, Figures~\ref{fig:transport_plan_residuals_ns_1000} and \ref{fig:transport_plan_residual_ns_2000}  show that Algorithm \ref{alg:2} exhibits the most favourable practical performance among the tested methods, which we attribute in part to its non-decreasing step-size rule and in part to the additional damping induced by \eqref{eq:uot-w-update-alg3}. This damping appears to suppress oscillations in the auxiliary variable $w$, thereby improving the decay of the feasibility and combined KKT residuals while preserving the same final transport structure as the baseline methods.

\subsection{Graph-fused regression problems}
\label{sec:real_data_graph_regression}

To further assess the practical performance of the proposed methods, we consider graph-fused regression models on two real datasets. The objective is to estimate a collection of local linear models while simultaneously encouraging neighbouring samples in a data-dependent graph to share similar regression coefficients. This leads to a network-lasso type formulation \cite{hallac2015network}, which fits naturally into the linearly constrained separable convex optimisation framework studied in this paper.

Let $N$ denote the number of samples, and let $a_i\in\mathbb{R}^d$ and $c_i\in\mathbb{R}$ be the feature vector and response associated with sample $i$, respectively. We construct an undirected graph $G=(V,E)$, where $V=\{1,\dots,N\}$ and $E$ is obtained from a $k$-nearest neighbour graph built from suitable contextual variables of the dataset. For each node $i$, we associate a local coefficient vector $x_i\in\mathbb{R}^d$, and for each edge $e=(i,j)\in E$, we introduce an auxiliary variable $w_e\in\mathbb{R}^d$. We then consider the following problem
\begin{align}\label{eq:real_data_graph_fused_model_rewrite}
\min_{\{x_i\}_{i=1}^N,\{w_e\}_{e\in E}}&
\left\{
\sum_{i=1}^N
\left(
\frac12(a_i^\top x_i-c_i)^2+\frac{\mu_x}{2}\|x_i\|_2^2
\right)
+\lambda\sum_{e=(i,j)\in E}\gamma_e\|w_e\|_2
\right\}\nonumber\\
&\text{subject to}\quad
x_i-x_j-w_e=0,\qquad \forall e=(i,j)\in E.
\end{align}
Here $\mu_x>0$ is a ridge parameter, $\lambda>0$ is the fusion regularisation parameter, and $\gamma_e>0$ is an edge weight. By defining
\[
x=(x_1,\dots,x_N)\in\mathbb{R}^{Nd}
~~\text{and}~~
w=(w_e)_{e\in E}\in\mathbb{R}^{|E|d},
\]
problem \eqref{eq:real_data_graph_fused_model_rewrite} can be written in the form \eqref{eq:model} as
\[
\min_{x,w}\; g(x)+f(w)
\quad\text{subject to}\quad
Ax-w=0,
\]
where
\[
g(x)
=
\sum_{i=1}^N
\left(
\frac12(a_i^\top x_i-c_i)^2+\frac{\mu_x}{2}\|x_i\|_2^2
\right),
\quad
f(w)=\lambda\sum_{e\in E}\gamma_e\|w_e\|_2,
\quad
B=-I.
\]
\begin{figure}[htbp]
\centering

\subfloat[Relative objective gap\label{fig:plan-alg2}]{
  \includegraphics[width=0.30\textwidth]{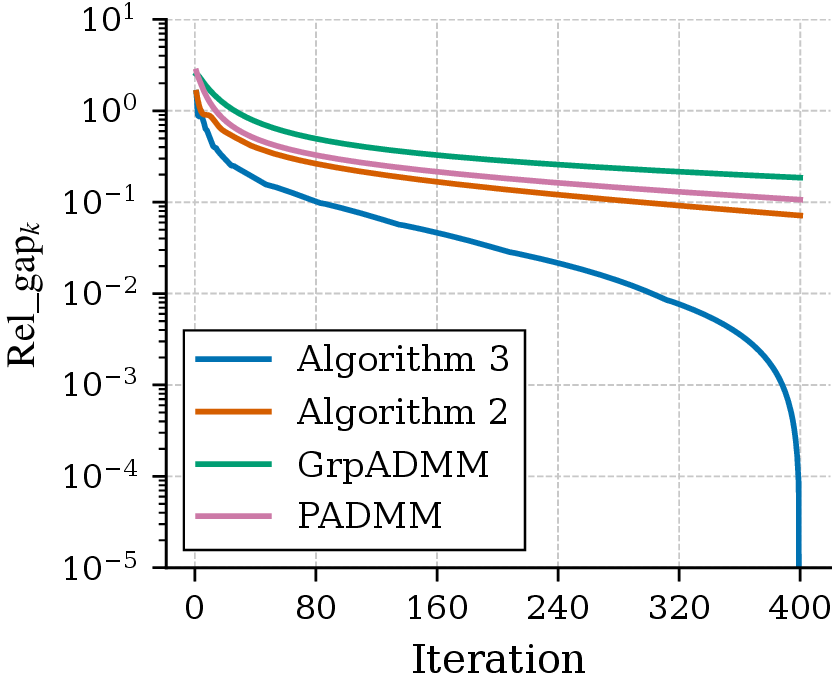}
}\hfill
\subfloat[Feasibility residual\label{fig:plan-alg1}]{
  \includegraphics[width=0.30\textwidth]{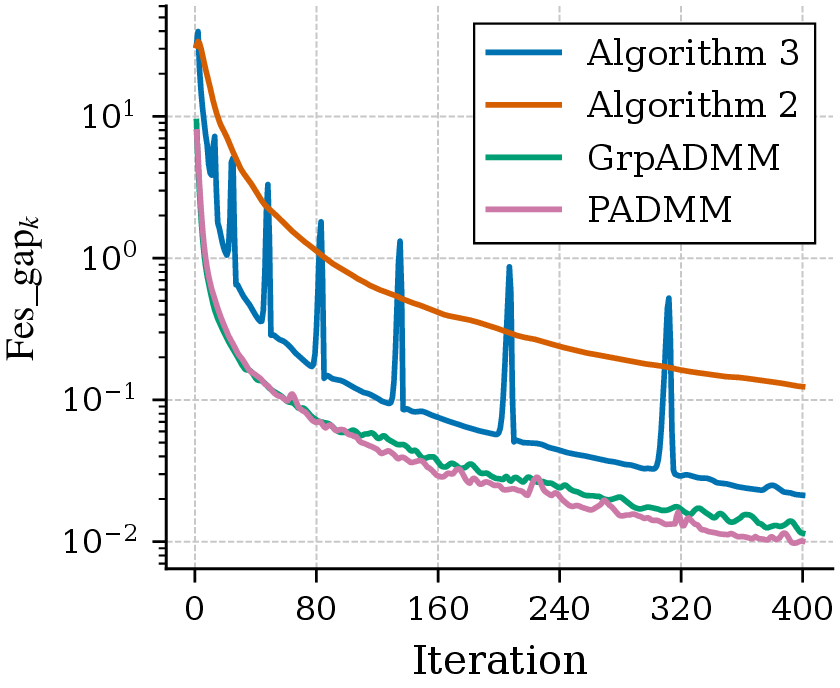}
}

\vspace{0.75em}

\subfloat[KKT residual\label{fig:plan-grpadmm}]{
  \includegraphics[width=0.30\textwidth]{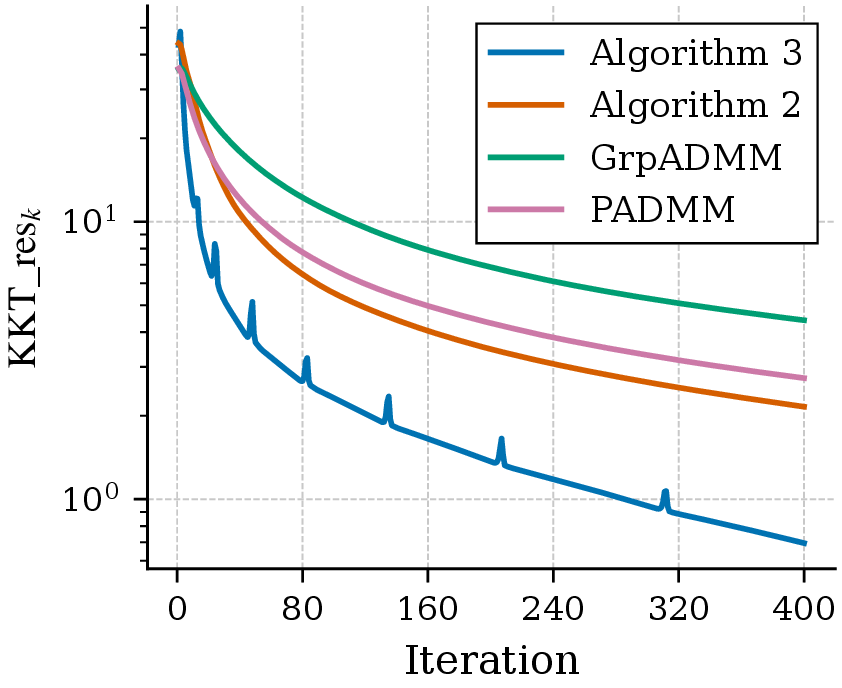}
}\hfill
\subfloat[Primal step-sizes\label{fig:plan-padmm}]{
  \includegraphics[width=0.30\textwidth]{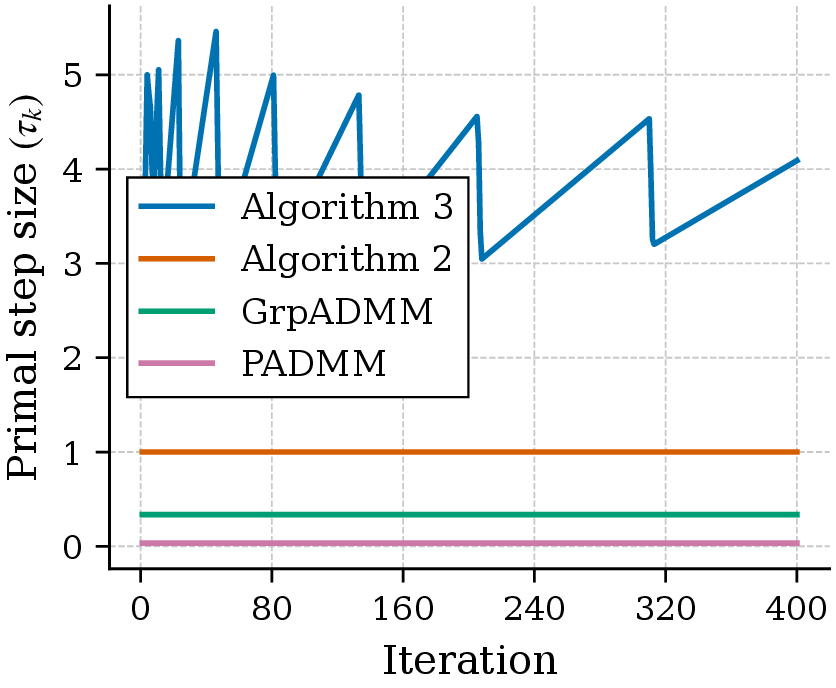}
}

\caption{Performance comparison of Algorithm~\ref{alg:2} , Algorithm~\ref{alg:extened psi}, \ref{eq:GrpADMM}, and \ref{eq:PADMM} on the California Housing dataset.}
\label{fig:california_datset}
\end{figure}
If $D$ denotes the oriented incidence matrix of the graph $G$, then we define $A := D \otimes I_d,$ where $\otimes$ denotes the Kronecker product and $I_d$ is the $d \times d$ identity matrix. Equivalently, $A$ is the block matrix whose $(i,j)$-th block is $D_{ij} I_d$. Thus, the operator $A$ couples neighbouring local models through the graph, while the nonsmooth term $f$ promotes agreement of the coefficients along adjacent edges. For Algorithms~\ref{alg:extened psi}, \ref{alg:2}, and \eqref{eq:GrpADMM}, we use the structured proximal matrices
\begin{equation}\label{eq:real_data_ST_choice_rewrite}
S=\operatorname{blkdiag}\bigl((\deg(i)+\delta)I_d\bigr)_{i=1}^N,
\qquad
T=\operatorname{blkdiag}\bigl(t_0(1+\gamma_e)I_d\bigr)_{e\in E},
\end{equation}
with $\delta>0$ and $t_0>0$. This choice is well-suited to graph-induced problems as $A^\top A$ inherits the degree pattern of the graph Laplacian, so the diagonal matrix $S$ provides a simple preconditioner for the $x$-block, while a positive matrix $T$ stabilises the $w$-block on heterogeneous graphs. With this choice, the $x$-subproblem remains nodewise separable, and each node update can be evaluated efficiently using a rank-one inverse update of Sherman--Morrison type \cite{ShermanMorrison1950,HornJohnson2013}. The $w$-subproblem is also separable over the edges and reduces to a vector soft-thresholding step. In the present graph-fused setting, the KKT residual \eqref{eq:kkt-residual} simplifies to 
\begin{equation*}
\mathrm{KKT\_res}_k
=
\left(
\|Ax_k-w_k\|_2^2
+
\bigl\|x_k-\operatorname{prox}_g(x_k-A^\top y_k)\bigr\|_2^2
+
\bigl\|w_k-\operatorname{prox}_f(w_k+y_k)\bigr\|_2^2
\right)^{1/2}.
\end{equation*}

\begin{table}[htbp]
\centering
\caption{Runtime and graph-consensus statistics for the California Housing graph-fused regression problem.}
\label{tab:california_results_short}
\begin{tabular}{lccc}
\toprule
Method & Runtime (s) & Fused-edge fraction & Consensus components \\
\midrule
Algorithm~\ref{alg:2}            & 82.893 & $3.820265\times 10^{-1}$ & 10462 \\
Algorithm~\ref{alg:extened psi}  & 65.444 & $3.522974\times 10^{-1}$ & 10931 \\
GrpADMM                          & 58.393 & $3.271104\times 10^{-1}$ & 11647 \\
PADMM                            & 62.242 & $3.542092\times 10^{-1}$ & 11123 \\
\bottomrule
\end{tabular}
\end{table}
In addition, we record the fraction of fused edges and the number of connected components obtained after thresholding the edge variables by $\|w_e\|\le 10^{-3}$. These measures provide a useful summary of how strongly each method promotes graph consensus. We set $\delta=t_0=0.25$ and other parameters selected in this experiment are as follows.

\begin{itemize} 
\item \textbf{Algorithm~\ref{alg:2}:} $S$ and $T$ are chosen as in \eqref{eq:real_data_ST_choice_rewrite}, $\beta=0.1$, $\tau_0=1$. \item \textbf{Algorithm~\ref{alg:extened psi}:} $S$ and $T$ are again chosen as in \eqref{eq:real_data_ST_choice_rewrite}, $\psi=1.65$, $\beta=0.1$, $\mu=0.8$, and $\tau_0=1$. 
\item \textbf{\ref{eq:GrpADMM}:} $S$ and $T$ are again chosen as in \eqref{eq:real_data_ST_choice_rewrite}, $\psi=\varphi$, $\sigma_{\rm grp}=1.4$, and $\tau_{\rm grp} = 0.95\, \frac{\psi\,\lambda_{\min}(S)}{\sigma_{\rm grp}\|A\|^2}$. 
\item \textbf{\ref{eq:PADMM}:} Set $S_{\rm pad} = \frac{1}{\tau_{\rm pad}}I-\sigma_{\rm pad}A^\top A,$ $ T \text{ chosen as in }\eqref{eq:real_data_ST_choice_rewrite},$ and $\sigma_{\rm pad}=1.4$ and $\tau_{\rm pad} = \frac{0.95}{\sigma_{\rm pad}\|A\|^2}$. \end{itemize}

\paragraph{\textbf{California Housing dataset:}} We first consider the California Housing dataset \cite{pace1997sparse}, which contains $N=20640$ observations describing demographic and housing characteristics of California districts. To construct a graph-fused regression model with a meaningful spatial interpretation, we use the six nonspatial covariates
\[
\texttt{MedInc},\ \texttt{HouseAge},\ \texttt{AveRooms},\ \texttt{AveBedrms},\ \texttt{Population},\ \texttt{AveOccup},
\]
as local predictors, while the response is taken to be $\texttt{MedHouseVal}$. The geographical variables $\texttt{Latitude}$ and $\texttt{Longitude}$ are used only to build the graph. After standardisation, we construct a $6$-nearest-neighbour graph from the spatial coordinates, encouraging nearby districts to share similar local regression coefficients. Accordingly, each node variable satisfies $x_i\in\mathbb{R}^{6}$. The edge weights are chosen as
\[
\gamma_e=\exp\!\left(-\frac{d_e}{\operatorname{median}\{d_{e'}:e'\in E\}}\right),
\]
where $d_e$ denotes the Euclidean distance associated with edge $e$. We set $\mu_x=10^{-2}$ and $\lambda=0.18$. The convergence behavior is displayed in Figures~\ref{fig:california_datset}--\ref{fig:california_cluster}. Figure~\ref{fig:california_datset} shows that Algorithm~\ref{alg:2} attains the smallest relative objective residual and the smallest combined KKT residual. Although GrpADMM and PADMM achieve slightly better feasibility residuals, this advantage is not reflected in the objective residual or in the overall KKT residual. Finally, Figure~\ref{fig:california_cluster} illustrates that the resulting consensus pattern is spatially meaningful over the geographical coordinates. Further, Table~\ref{tab:california_results_short} reports runtime and two graph-structural summaries. Among all methods, Algorithm~\ref{alg:2} achieves the largest fused-edge fraction and the fewest connected components, indicating the strongest overall consensus. Algorithm~\ref{alg:extened psi} remains competitive in runtime, but it induces a weaker clustering effect. GrpADMM is the fastest method on this dataset, though at the price of a visibly weaker performance in the convergence plots.

\begin{figure}[htbp]
\centering

\subfloat[Target values\label{fig:plan-target}]{
  \includegraphics[width=0.3\textwidth]{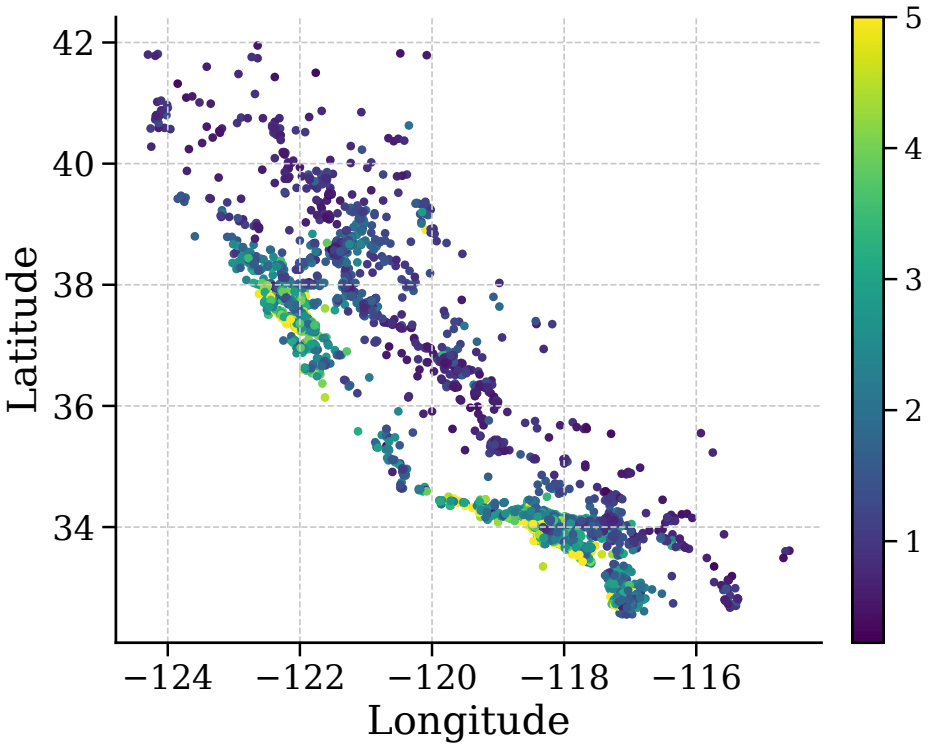}
}
\hfill
\subfloat[Algorithm~\ref{alg:2}\label{fig:plan-alg2}]{
  \includegraphics[width=0.3\textwidth]{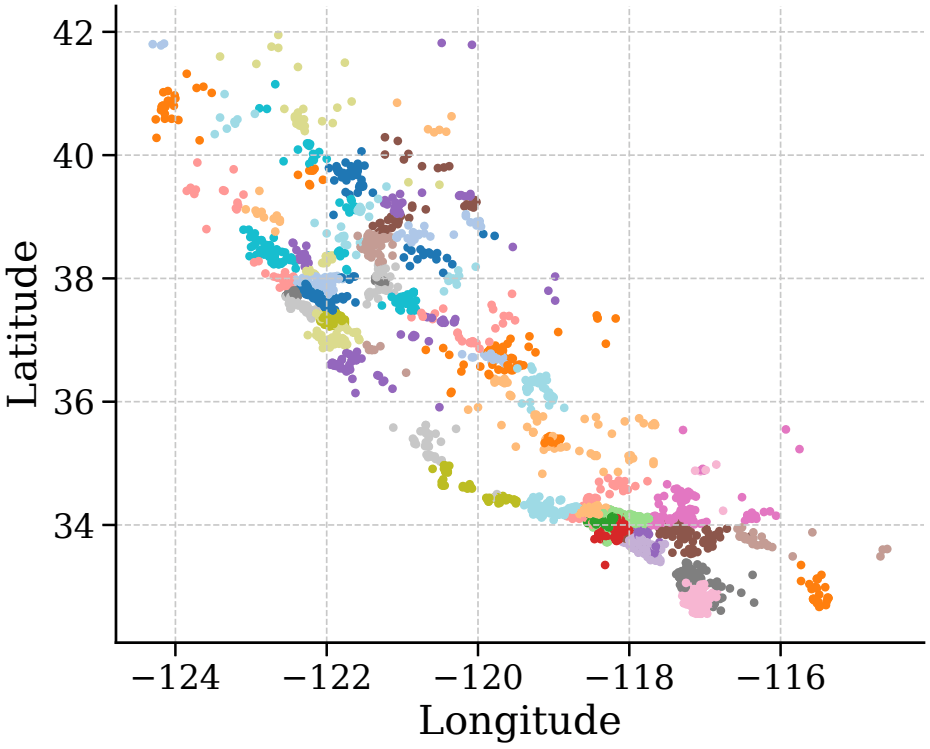}
}

\vspace{0.8em}

\subfloat[Algorithm~\ref{alg:extened psi}\label{fig:plan-alg1}]{
  \includegraphics[width=0.29\textwidth]{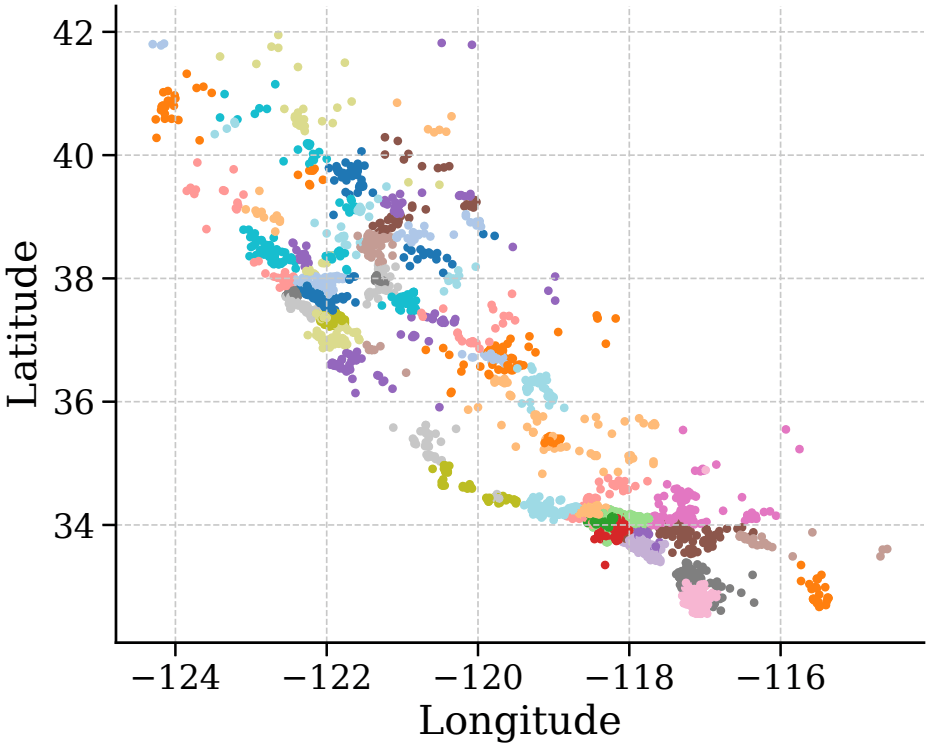}
}
\hfill
\subfloat[PADMM\label{fig:plan-padmm}]{
  \includegraphics[width=0.29\textwidth]{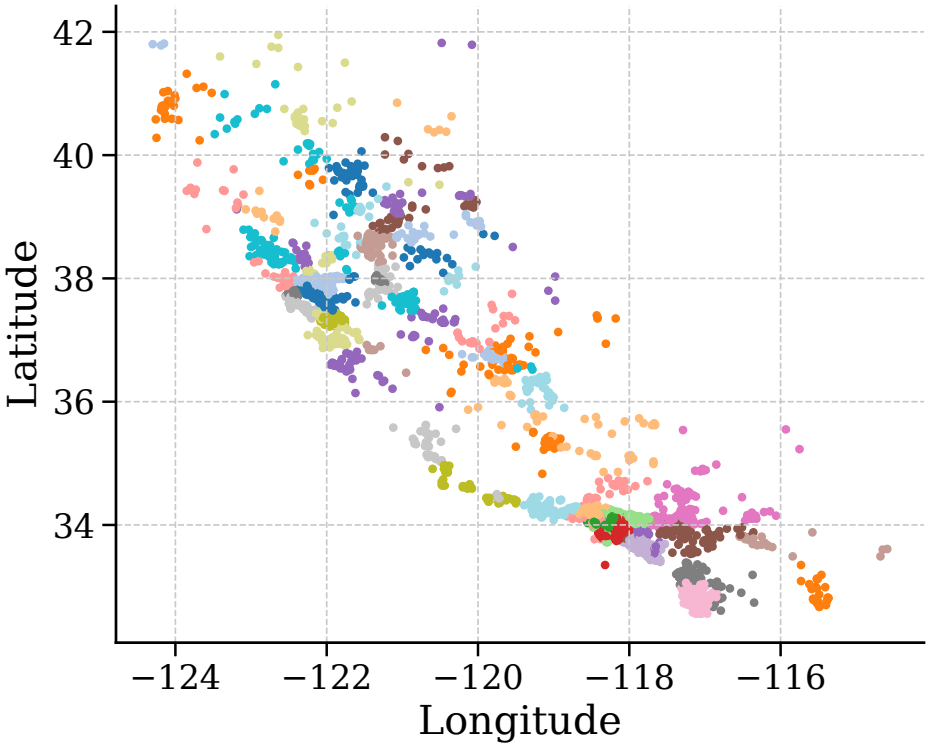}
}
\hfill
\subfloat[GrpADMM\label{fig:plan-grpadmm}]{
  \includegraphics[width=0.29\textwidth]{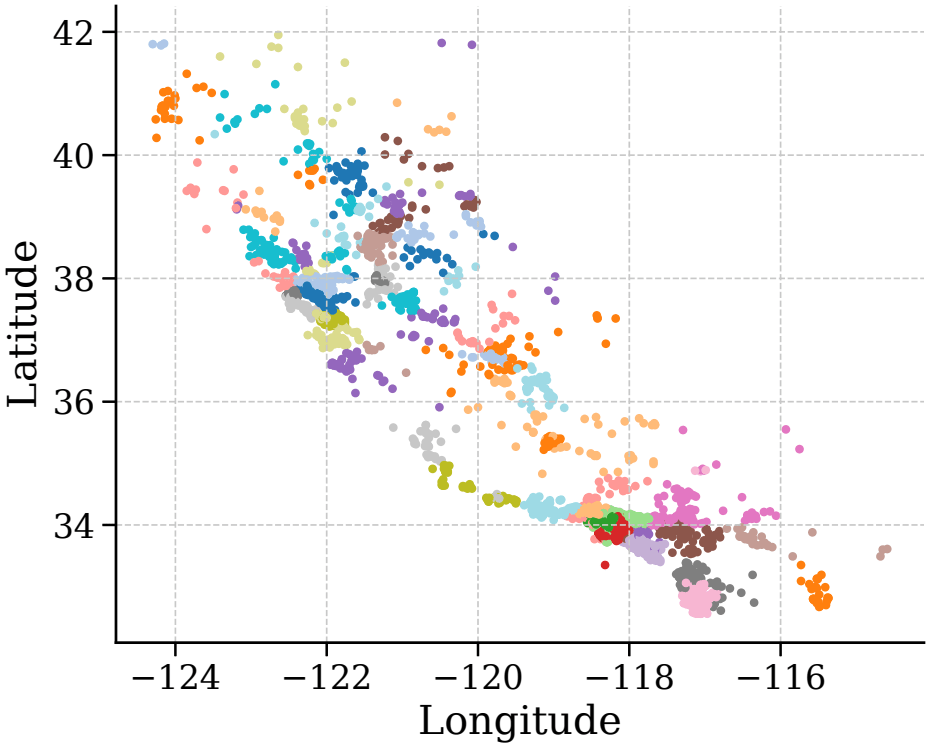}
}

\caption{Visualisation of the California Housing graph-fused regression experiment over the geographical coordinates. Panel (a) shows the response variable \texttt{MedHouseVal}, while panels (b)--(e) display the consensus clusters produced by the four Algorithms. Each colour represents one connected component of the graph obtained after thresholding the edge variables by $\|w_e\|\le 10^{-3}$.}
\label{fig:california_cluster}
\end{figure}

\paragraph{\textbf{Forest CoverType dataset:}} As a substantially larger benchmark, we next consider the Forest CoverType dataset \cite{blackard1999comparative}, which contains $N=581012$ observations extracted from $30\times 30$ meter forest cells in the Roosevelt National Forest of northern Colorado. The original dataset is intended for classification, but here we repurpose it into a large-scale graph-fused regression problem. Specifically, we take \emph{Elevation} as the response variable and use the remaining nine continuous cartographic variables as local predictors.
\[
\begin{aligned}
&\texttt{Aspect},\ \texttt{Slope},\ \texttt{Horizontal\_Distance\_To\_Hydrology}, \texttt{Vertical\_Distance\_To\_Hydrology},\\
&\texttt{Horizontal\_Distance\_To\_Roadways},\ \texttt{Hillshade\_9am},\ \texttt{Hillshade\_Noon},\mathbb{}\texttt{Hillshade\_3pm},\\
&\texttt{Horizontal\_Distance\_To\_Fire\_Points}.
\end{aligned}
\]
Thus, each node variable satisfies $x_i\in\mathbb{R}^{9}$. To define the graph, we use a subset of terrain descriptors that capture local topographic similarity and construct a $4$-nearest neighbor graph after standardisation. Hence, two forest cells are connected when their terrain profiles are similar, and the graph regularisation promotes similarity of the associated local regression models.

\begin{figure}[htbp]
\centering

\subfloat[Relative objective gap\label{fig:plan-alg2}]{
  \includegraphics[width=0.30\textwidth]{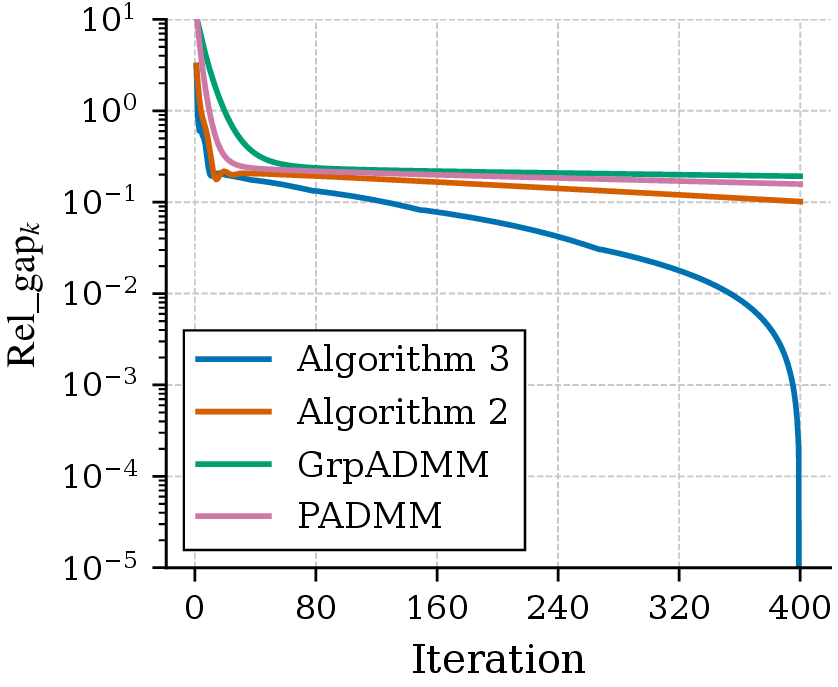}
}\hfill
\subfloat[Feasibility residual\label{fig:plan-alg1}]{
  \includegraphics[width=0.30\textwidth]{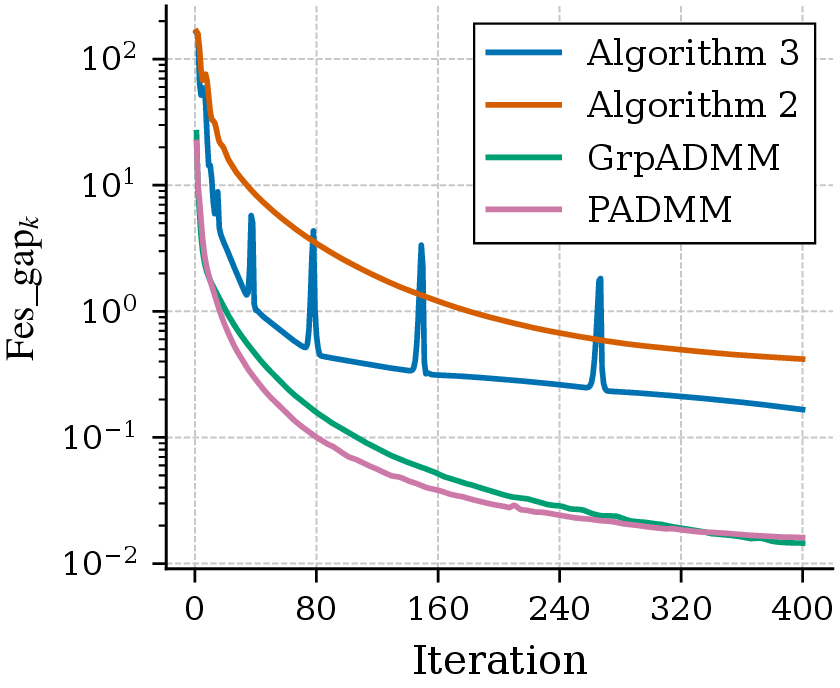}
}

\vspace{0.75em}

\subfloat[KKT residual\label{fig:plan-grpadmm}]{
  \includegraphics[width=0.30\textwidth]{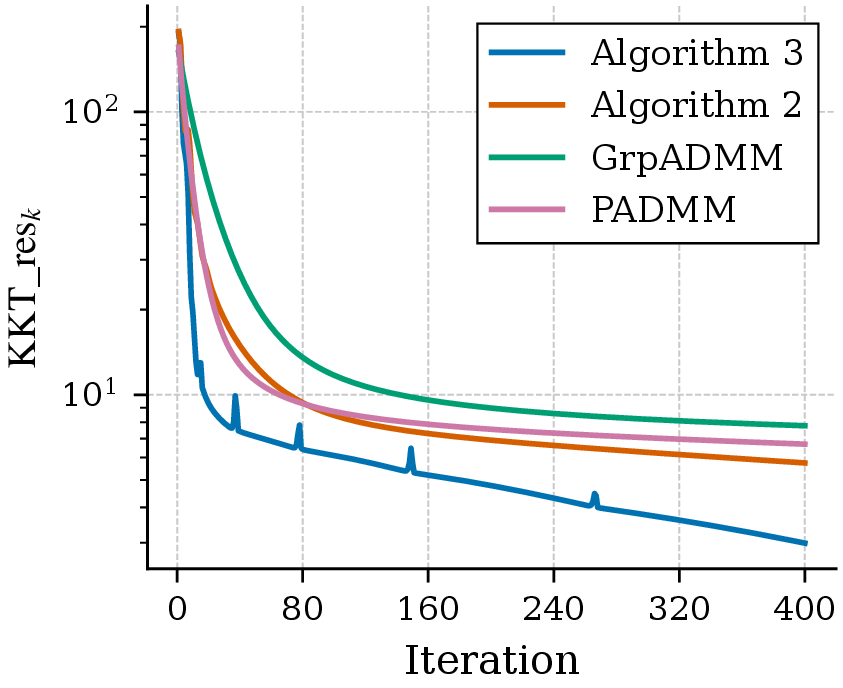}
}\hfill
\subfloat[Primal step-sizes\label{fig:plan-padmm}]{
  \includegraphics[width=0.30\textwidth]{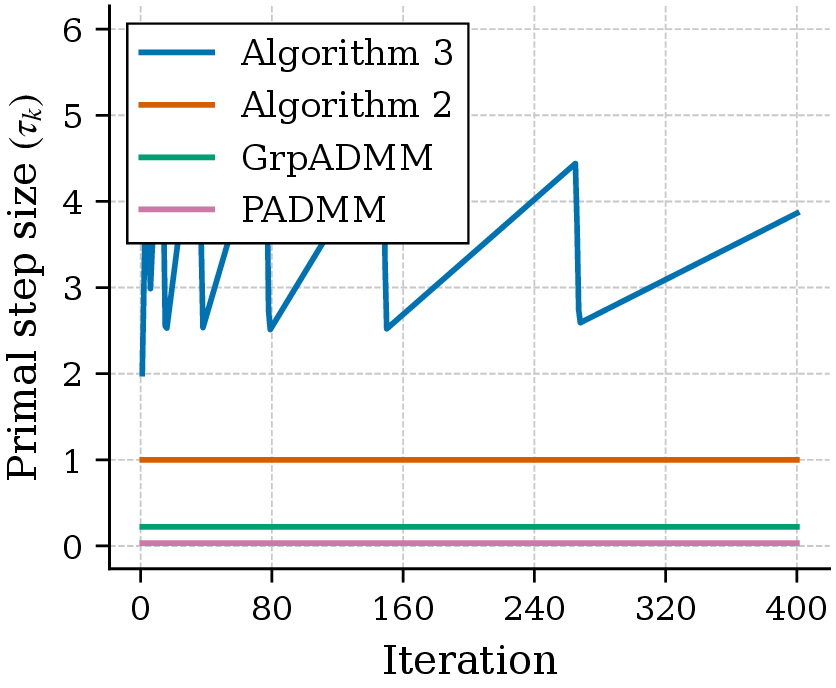}
}

\caption{Performance comparison of Algorithm~\ref{alg:2} , Algorithm~\ref{alg:1}, \ref{eq:GrpADMM}, and \ref{eq:PADMM} on the Forest CoverType dataset.}
\label{fig:Forest_cover_dataset}
\end{figure}

\begin{table}[htbp]
\centering
\caption{Runtime and graph-consensus statistics for the Forest CoverType graph-fused regression problem.}
\label{tab:covtype_results_short}
\begin{tabular}{lccc}
\toprule
Method & Runtime (s) & Fused-edge fraction & Consensus components \\
\midrule
Algorithm~\ref{alg:2}            & 1688.570 & $3.983600\times 10^{-1}$ & 250531 \\
Algorithm~\ref{alg:extened psi}  & 1682.611 & $3.685597\times 10^{-1}$ & 264183 \\
GrpADMM                          & 1533.710 & $3.331452\times 10^{-1}$ & 292807 \\
PADMM                            & 1620.829 & $3.484381\times 10^{-1}$ & 282361 \\
\bottomrule
\end{tabular}
\end{table}

\begin{figure}[htbp]
\centering

\subfloat[Target values\label{fig:plan-target}]{
  \includegraphics[width=0.3\textwidth]{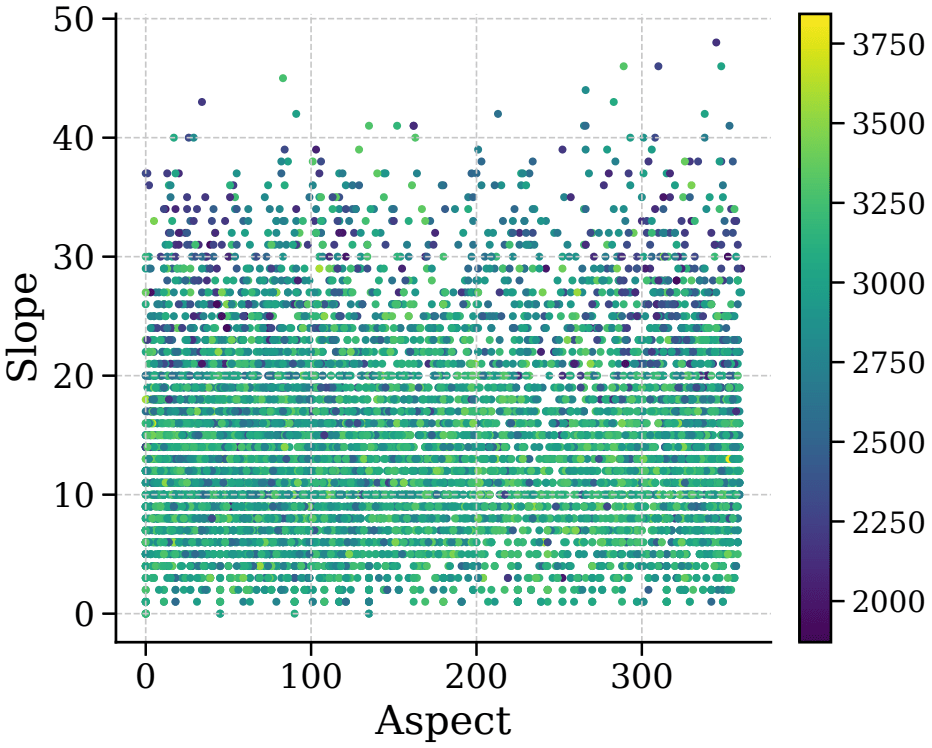}
}
\hfill
\subfloat[Algorithm~\ref{alg:2}\label{fig:plan-alg2}]{
  \includegraphics[width=0.3\textwidth]{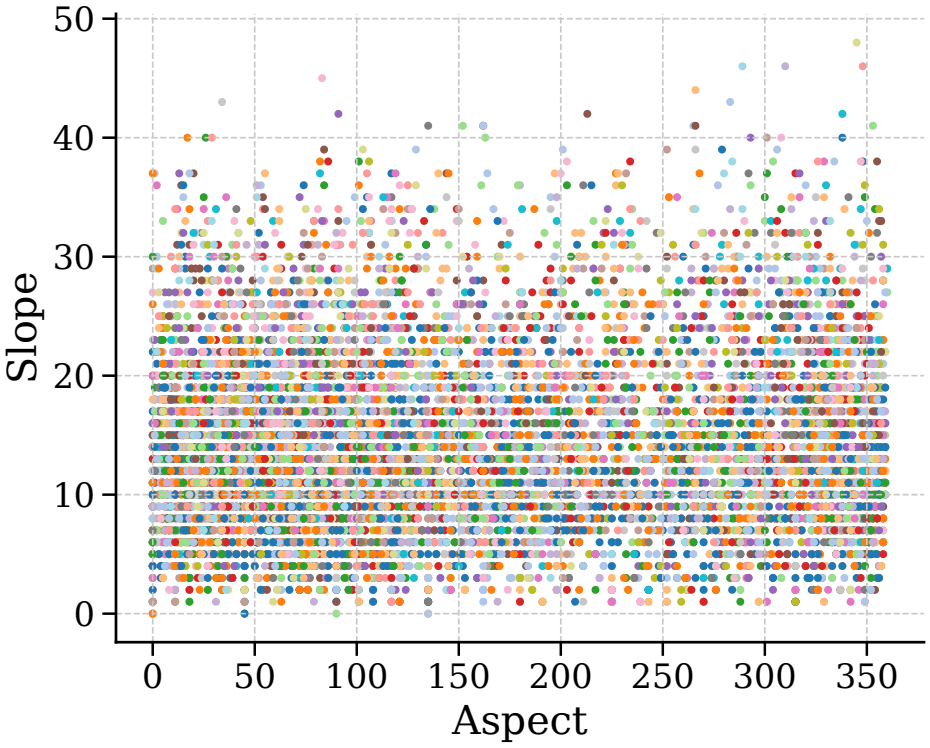}
}

\vspace{0.8em}

\subfloat[Algorithm~\ref{alg:extened psi}\label{fig:plan-alg1}]{
  \includegraphics[width=0.29\textwidth]{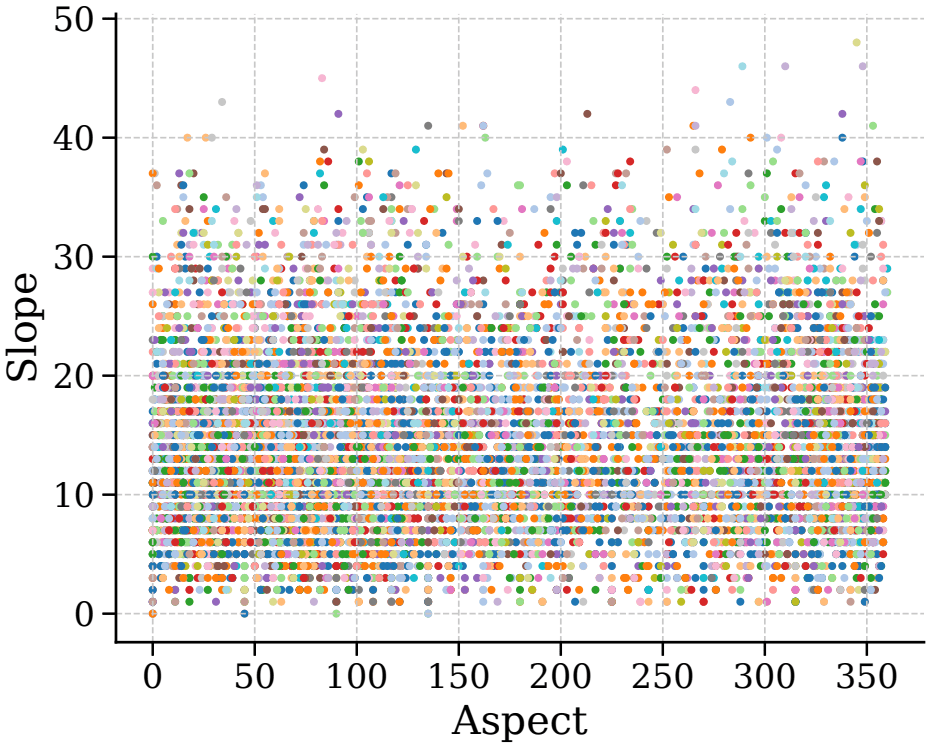}
}
\hfill
\subfloat[PADMM\label{fig:plan-padmm}]{
  \includegraphics[width=0.29\textwidth]{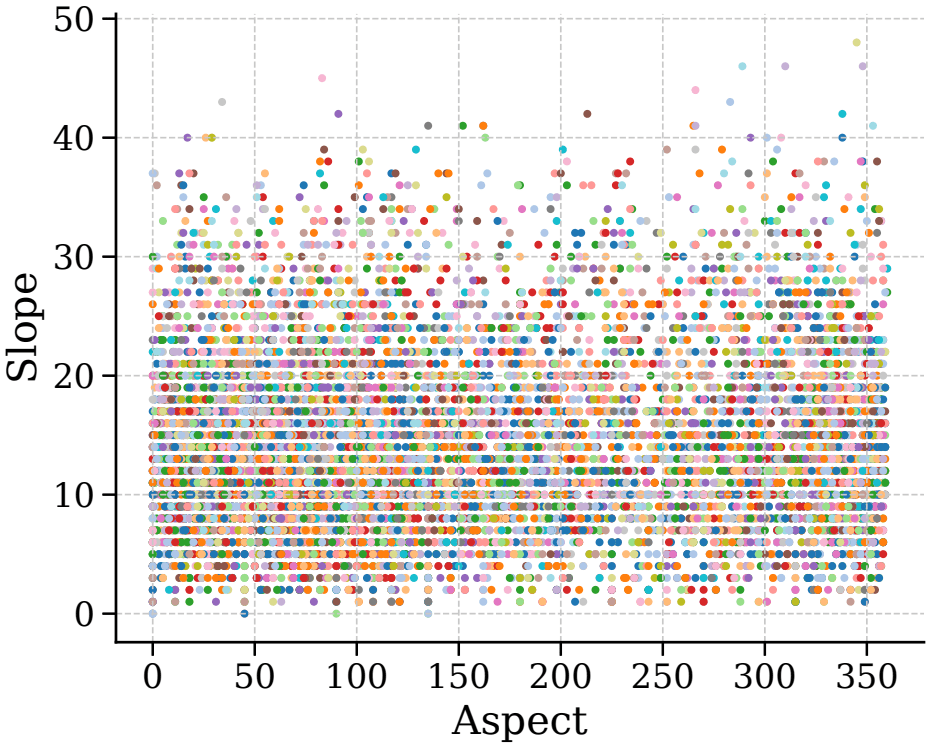}
}
\hfill
\subfloat[GrpADMM\label{fig:plan-grpadmm}]{
  \includegraphics[width=0.29\textwidth]{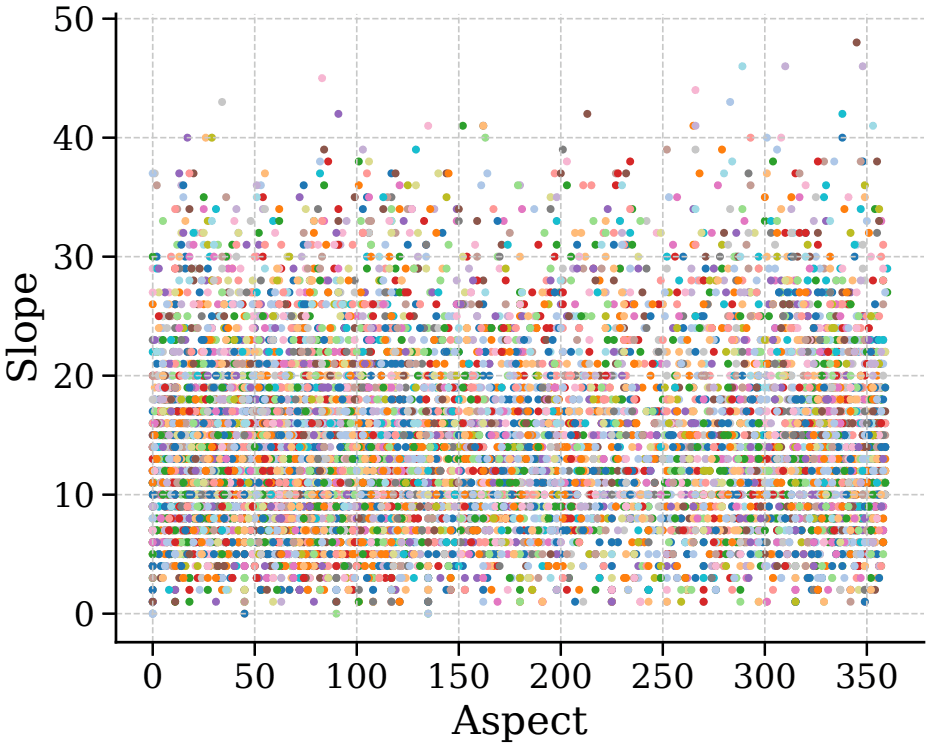}
}

\caption{Visualisation of the Forest CoverType graph-fused regression experiment in the feature plane. Panel (a) shows the response variable \texttt{Elevation}, while panels (b)--(e) display the consensus clusters produced by the four algorithms. Each colour represents one connected component of the graph obtained after thresholding the edge variables by $\|w_e\|\le 10^{-3}$.}
\label{fig:Forest_cover_cluster}
\end{figure}

The convergence plots for this dataset, shown in Figures~\ref{fig:Forest_cover_dataset}--\ref{fig:Forest_cover_cluster}, reveal the same overall pattern as in the California Housing experiment, but now in a genuinely large-scale regime. Algorithm~\ref{alg:2} again achieves the smallest relative objective residual and the smallest combined KKT residual. The gap between allowing non-decreasing step-sizes and fixed-step methods is even more pronounced in the objective and KKT curves. As before, GrpADMM and PADMM enforce feasibility more aggressively, but this does not translate into better overall performance. The clustering plots also indicate that Algorithm~\ref{alg:2} induces the strongest consensus pattern among the four methods.

Table~\ref{tab:covtype_results_short} reports runtime and graph-consensus statistics, where Algorithm~\ref{alg:2} again produces the largest fused-edge fraction and the smallest number of connected components. However, its runtime remains comparable to that of Algorithm~\ref{alg:extened psi}. The GrpADMM is somewhat faster, but the figures show that this speed advantage comes at the cost of a significantly weaker final objective value and a noticeably larger KKT residual.

Overall, the two real-data experiments yield a consistent conclusion. On both datasets, Algorithm~\ref{alg:2} provides the best overall balance between objective decrease and graph-induced consensus. The fixed step-size counterpart, especially GrpADMM, often achieves smaller raw feasibility residuals, but this advantage is offset by worse objective behaviour and larger KKT residuals.

\section{Conclusion}

We presented two practical step-size strategies for \ref{eq:GrpADMM} to solve separable convex problems of the form \eqref{eq:model}. In the first strategy, the primal steps were iteratively decaying, eliminating the need to estimate $\|A\|$. This step-size sequence converges to a positive limit, and using this crucial fact, we proved global convergence of the iterates. Furthermore, we modified the proximal terms in the $w$-update of \ref{eq:GrpADMM}, which enabled us to propose an eventually increasing step-size strategy and to prove the algorithm's global convergence. Numerical experiments on LASSO signal recovery, image deblurring, optimal transport problems, and graph-fused regression problems on large real datasets demonstrate that the proposed methods are effective and competitive in practice.

Several directions remain open for further investigation.
\begin{itemize}
    \item A natural next step is to extend the step-size strategies of Algorithms~\ref{alg:extened psi} and~\ref{alg:2} to the broader class of separable optimisation problems considered in \cite{yin2024golden}.

\item 
The present analysis of Algorithms~\ref{alg:2} leaves the question of whether the admissible range of the parameter $\psi$ can be further enlarged. Although Algorithm~\ref{alg:extened psi} allows a wider range $\psi\in(1,1+\sqrt{3})$ under a modified setting, it remains an interesting open question whether this interval can be further extended. Addressing this question may require a sharper Lyapunov estimate or a different coupling between the extrapolation parameter and the proposed step-size rule. We leave this issue for future investigation.

\item It was observed in \cite{he2016proximal} that indefinite proximal terms can sometimes lead to better numerical
performance in ADMM-type methods. It would therefore be interesting to investigate whether the
convergence theory of Algorithms~\ref{alg:extened psi} and~\ref{alg:2} can be extended, possibly under suitable modifications,
to the case where the proximal matrices $S$ and $T$ are indefinite.

\item When $g$ is strongly convex, accelerated variants of GrpADMM have been developed in \cite{Chen2023GRPADMM} with improved convergence properties. It would be desirable to derive analogous accelerated versions of the two proposed algorithms.

\item It would also be worthwhile to study whether the proposed step-size rules can be extended to more general settings, such as multi-block, stochastic, or certain structured nonconvex variants of prox-based ADMM-type algorithms.
\end{itemize}

\begin{acknowledgements}
The authors sincerely thank the Editor and the anonymous referees for their careful reading of the manuscript and for their valuable comments and suggestions, which helped improve the clarity and presentation of the paper. Santanu Soe gratefully acknowledges A/Prof. Matthew K. Tam for his constant support, encouragement, and guidance throughout his PhD. The research of Santanu Soe was supported by the Prime Minister’s Research Fellowship program (Project number SB23242132MAPMRF005015), Ministry of Education, Government of India.
\end{acknowledgements}
\section*{Data Availability}
The \texttt{California Housing} and \texttt{Forest CoverType} datasets used in the numerical experiments are publicly available and can be accessed through the \texttt{scikit-learn} dataset library. The Python scripts used to generate the numerical results are available from the corresponding author upon reasonable request.
\section*{Conflict of Interest}
The authors declare that there are no conflicts of interest in this paper.

\bibliographystyle{spmpsci}
\bibliography{my_bib}

\end{document}